\documentclass[reqno,11pt]{amsart}

\usepackage{amsmath}
\usepackage{amsthm}
\usepackage{amssymb}
\usepackage{amscd}
\usepackage{xypic}
\usepackage{verbatim}
\usepackage{mathabx}
\usepackage{graphicx}
\usepackage{palatino}
\usepackage{bbm}
\usepackage{extarrows}

\usepackage{xr}
\usepackage[plainpages=false,colorlinks,hyperindex,pdfpagemode=None,bookmarksopen,linkcolor=red,citecolor=blue,urlcolor=blue]{hyperref}
\usepackage{pdflscape}
\usepackage{stmaryrd}

\usepackage{multirow}

\usepackage{comment}

\swapnumbers

\theoremstyle{plain}
\newtheorem{theorem}[subsubsection]{Theorem}

\newtheorem*{theorem*}{Theorem}
\newtheorem*{mainconclusions*}{Main Conclusions}
\newtheorem{proposition}[subsubsection]{Proposition}
\newtheorem*{proposition*}{Proposition}
\newtheorem{lemma}[subsubsection]{Lemma}
\newtheorem*{lemma*}{Lemma}

\newtheorem{corollary}[subsubsection]{Corollary}
\newtheorem*{corollary*}{Corollary}

\theoremstyle{definition}
\newtheorem{definition}[subsubsection]{Definition}
\theoremstyle{remark}
\newtheorem{remark}[subsubsection]{Remark}
\newtheorem{remarks}[subsubsection]{Remarks}

\DeclareFontFamily{OT1}{rsfs}{}
\DeclareFontShape{OT1}{rsfs}{n}{it}{<-> rsfs10}{}
\DeclareMathAlphabet{\mathscr}{OT1}{rsfs}{n}{it}

\newcommand{\inv}{{\operatorname{inv}}}

\newcommand{\sgn}{\mathrm{sgn}}

\newcommand{\Ad}{\mathrm{Ad}}
\newcommand{\Res}{\mathrm{Res}}

\newcommand{\Z}{\mathbb{Z}}
\newcommand{\C}{\mathfrak{C}}

\newcommand{\ad}{{\operatorname{ad}}}

\newcommand{\CC}{\mathbb{C}}

\newcommand{\RR}{\mathbb{R}}
\newcommand{\Rplus}{{\RR^\times_+}}

\newcommand{\QQ}{\mathbb{Q}}

\newcommand{\Ind}{\operatorname{Ind}}

\newcommand{\Aut}{{\operatorname{Aut}}}

\newcommand{\Gm}{\mathbb{G}_m}
\newcommand{\Ga}{\mathbb{G}_a}

\newcommand{\GL}{\operatorname{GL}}
\newcommand{\Mat}{\operatorname{Mat}}
\newcommand{\Sym}{\operatorname{Sym}}

\newcommand{\PGL}{\operatorname{PGL}}

\newcommand{\SL}{\operatorname{SL}}

\newcommand{\SO}{{\operatorname{SO}}}

\newcommand{\Gal}{\operatorname{Gal}}

\newcommand{\tr}{\operatorname{tr}}

\newcommand{\Vol}{\operatorname{Vol}}
\newcommand{\diag}{{\operatorname{diag}}}
\newcommand{\adiag}{{\operatorname{adiag}}}

\newcommand{\Id}{\operatorname{Id}}

\newcommand{\Std}{{\operatorname{Std}}}
\newcommand{\st}{{\operatorname{st}}}

\newcommand{\temp}{{\operatorname{temp}}}

\newcommand{\KTF}{{\operatorname{KTF}}}

\newcommand{\AvgVol}{{\operatorname{AvgVol}}}
\newcommand{\PW}{{\operatorname{PW}}}

\newcommand{\fin}{{\operatorname{fin}}}

\newcommand{\U}{\mathcal{U}}

\newcommand{\Meas}{{\operatorname{Meas}}}

\newcommand{\LG}{{^LG}}

\newcommand{\Dfrac}[2]{ 
  \ooalign{ 
    $\genfrac{}{}{2.0pt}0{\phantom{#1}}{\phantom{#2}}$\cr 
    $\color{white}\genfrac{}{}{1.2pt}0{\normalcolor{#1}}{\normalcolor{#2}}$} 
}

\begin{document}

\numberwithin{equation}{section}
\setcounter{tocdepth}{2}
\setcounter{section}{5}
\title[Transfer operators and Hankel transforms, II]{Transfer operators and Hankel transforms between relative trace formulas, II: Rankin--Selberg theory}
\author{Yiannis Sakellaridis}
\email{sakellar@jhu.edu}
\address{Department of Mathematics, Johns Hopkins University, Baltimore, MD 21218, USA.}


\begin{abstract}
The Langlands functoriality conjecture, as reformulated in the ``beyond endoscopy'' program, predicts comparisons between the (stable) trace formulas of different groups $G_1, G_2$ for every morphism $\LG_1\to \LG_2$ between their $L$-groups. This conjecture can be seen as a special case of a more general conjecture, which replaces reductive groups by spherical varieties and the trace formula by its generalization, the relative trace formula. 

The goal of this article and its precursor \cite{SaTransfer1} is to demonstrate, by example, the existence of ``transfer operators'' betweeen relative trace formulas, which generalize the scalar transfer factors of endoscopy. These transfer operators have all properties that one could expect from a trace formula comparison: matching, fundamental lemma for the Hecke algebra, transfer of (relative) characters. Most importantly, and quite surprisingly, they appear to be of abelian nature (at least, in the low-rank examples considered in this paper), even though they encompass functoriality relations of non-abelian harmonic analysis. Thus, they are amenable to application of the Poisson summation formula in order to perform the global comparison. Moreover, we show that these abelian transforms have some structure --- which presently escapes our understanding in its entirety --- as deformations of well-understood operators when the spaces under consideration are replaced by their ``asymptotic cones''. 

In this second paper we use Rankin--Selberg theory to prove the local transfer behind Rudnick's 1990 thesis (comparing the stable trace formula for $\SL_2$ with the Kuznetsov formula) and Venkatesh's 2002 thesis (providing a ``beyond endoscopy'' proof of functorial transfer from tori to $\GL_2$). As it turns out, the latter is not completely disjoint from endoscopic transfer --- in fact, our proof ``factors'' through endoscopic transfer. We also study the functional equation of the symmetric-square $L$-function for $\GL_2$, and show that it is governed by an explicit ``Hankel operator'' at the level of the Kuznetsov formula, which is also of abelian nature. A similar theory for the standard $L$-function was previously developed (in a different language) by Jacquet. 
\end{abstract}

\maketitle

\tableofcontents

\section{Introduction to the second part} \label{sec:intro2}

This paper is a continuation of \cite{SaTransfer1}.\footnote{Any references to sections or equations numbered 5 or lower refer to \cite{SaTransfer1}. I continue using the notation of that paper, see \S \ref{ssnotation}.} In this part, I use Rankin--Selberg theory to prove the geometric statements about transfer between the Kuz\-netsov formula and the stable trace formula for $\SL_2$, to calculate a formula for the Hankel transform (on the Kuznetsov formula) for the symmetric-square $L$-function of $\GL_2$, and to develop the local transfer behind Venka\-tesh's thesis, for the ``beyond endoscopy'' transfer from tori to the Kuznetsov formula of $\GL_2$. 

It is surprising, at first, that Rankin--Selberg theory is so powerful that it can be used to prove all these theorems that, in their previous treatments  using analytic number theory (whenever available, such as in Rudnick's and Venkatesh's theses), seemed like an unrelated collection of identities and tricks. It may also appear disappointing to whom hopes to discover a ``beyond endoscopy'' approach to functoriality that is unrelated to current methods for studying $L$-functions. Eventually, though, what seems to be happening is that different methods to study the same problem ``descend'' to the same operators of functorial transfer when viewed from the point of view of trace formulas. The hope is that this transfer operator generalizes to cases where methods such as the Rankin--Selberg method do not.

For example, in the course of proving functorial transfer from one-dimen\-sional tori to the Kuznetsov formula of $\GL_2$ (or, more correctly, of the group $G=\Gm\times \SL_2$), we run into the endoscopic transfer from tori to the trace formula of $\SL_2$. Our transfer operator, in that case, factors as:
\begin{equation}
\xymatrix{ \mathcal S(N,\psi\backslash G/N,\psi) \ar[r] &  \mathcal S^\kappa_\eta (\frac{\SL_2}{\SL_2}) \ar@{<->}[r]^{LL} & \mathcal S(T)^W }
\end{equation}
between the pertinent spaces of test functions, where $\mathcal S^\kappa_\eta (\frac{\SL_2}{\SL_2})$ denotes the space of ``kappa-orbital integrals'' on the conjugacy classes represented by the torus $T$ (which is indicated in our notation by the appearance of the quadratic character $\eta$ associated to the torus $T$), and the second arrow is the local endoscopic transfer of Labesse and Langlands. The first arrow can now be constructed directly using the Rankin--Selberg method.

The same method provides a trace formula-theoretic approach to the functional equation of the symmetric-square $L$-function of $\GL_2$. The result is a ``Hankel transform'' between test measures for the Kuznetsov formula of $G$, which acts on relative characters by the gamma factor of the local functional equation of this $L$-function. This is an analog of the descent of Fourier transform from the space of $2\times 2$ matrices, computed by Jacquet in \cite{Jacquet}, which corresponds to the standard $L$-function. Quite agreeably, these Hankel transforms are all expressed in terms of abelian Fourier transforms, despite the fact that they express functional equations for non-abelian $L$-functions. Thus, they are amenable to an application of a global Poisson summation formula, to give an independent proof of the functional equation (as was done for the standard $L$-function and its square in \cite{Herman, SaBE2}). 

The Hankel transform of the symmetric-square $L$-function and the transfer operator for the functorial lift from tori to $\GL_2$ (which, in terms of $L$-functions, is detected by poles of the symmetric-square $L$-function) are closely related: the former can naturally be written as a symmetric sequence of abelian operators, and the latter is, roughly, ``half'' of this sequence. I do not know if this is just a byproduct of the proof, or reflects something deeper.

Let me now describe in more detail the basic constructions of this paper.

\subsection{The Rankin--Selberg variety}

Unlike the first part of the paper, in this second part we fix notation for certain groups and spaces, that will be used consistently. We will denote by $\tilde G$ the group 
$$\tilde G=(\Gm\times\SL_2^2)/\{\pm 1\}^\diag,$$ which can also be identified with the subgroup of those pairs of elements in $\GL_2\times \GL_2$ which have the same determinant. The notation $G$ is used for the group 
$$ G=\Gm\times \SL_2.$$

More canonically, $\SL_2$ is understood as $\SL(V)$, the group of symplectic linear transformations of a two-dimensional symplectic space $V$, and the $\Gm$-factors appearing in the definitions of $\tilde G$ and $G$ have a different interpretation: the group $\tilde G$ can be more canonically written
$$\tilde G=(A\times\SL(V)^2)/\{\pm 1\}^\diag,$$
where $A=B/N$ is the universal Cartan of $\SL(V)$, identified with $\Gm$ via the positive half-root character, while 
$$ G=A_\ad\times \SL(V),$$
where $A_\ad = A/\{\pm 1\}$ is the universal Cartan of $\PGL(V)$, identified with $\Gm$ via the positive root character.

The group $\tilde G$ acts on the ``Rankin--Selberg variety''
$$\bar X= V\times^{\SL(V)^\diag} \SL(V)^2,$$ 
which is a two-dimensional symplectic vector bundle over $\SL(V)^\diag\backslash \SL(V)^2 \simeq \SL(V)$, with the ``factor'' $A$ of $\tilde G$ acting by a scalar on the fibers through the positive half-root character. The open $\tilde G$-orbit on $\bar X$ will be denoted by $X$.

Rankin--Selberg theory uses an Eisenstein series on $\PGL_2^\diag\subset (\Gm\backslash \tilde G)$, which can be constructed from a Schwartz function on $V$(=the fiber of $\bar X$ over $1\in \SL(V)$). There is no harm in smoothening this Schwartz function as a generalized function on $\bar X$, and multiplying it by an invariant measure, thus our point of departure will be the Schwartz space $\mathcal S(\bar X)$ of measures on $\bar X$.

We consider the diagonal action of $\SL(V)=\SL_2$ on $\bar X$. Notice that $X/\SL(V) = \frac{\SL(V)}{N}$ (where $N$ denotes, as before, a maximal unipotent subgroup, and the horizontal quotient line refers to the conjugation action), so the quotient $[\bar X/\SL_2^\diag]$ can be though of as a stacky embedding of the adjoint quotient $\frac{\SL_2}{N}$. The corresponding invariant-theoretic quotients are equal, and we fix an isomorphism 
\begin{eqnarray*} \Dfrac{\SL_2}{N} & \simeq & \mathbbm A^2 \\
                  \begin{pmatrix} a & b \\ c & d \end{pmatrix} & \mapsto & (c, t=\tr=a+d).
\end{eqnarray*}
We let $\mathcal S(\bar X/\SL_2)$ denote the push-forward of $\mathcal S(\bar X)$ to the invariant-theoretic quotient $\bar X\sslash \SL_2^\diag = \Dfrac{\SL_2}{N}$.

Taking the quotient by the $\SL_2^\diag$ action can be thought of as the analog of imposing the condition $\pi_1 = \widetilde{\pi_2}\otimes (\chi\circ \det)=:\pi$ to the Rankin--Selberg $L$-function. In that case, this $L$-function factors:
\begin{equation}\label{RSfactor}
L(\pi_1\times \pi_2, s) = L(\chi\times\Sym^2(\pi),s) L(\chi,s). 
\end{equation}
Thus, it is natural to think of the space $\mathcal S(\bar X/\SL_2)$ as the geometric incarnation of the $L$-function $L(\chi\times \Sym^2(\pi),1)L(\chi,1)$.

In  Section \ref{sec:RS} we extract a subspace $\mathcal S(\bar X/\SL_2)^\circ$ of $\mathcal S(\bar X/\SL_2)$, that is responsible for  the factor $L(\chi\Sym^2(\pi),s)$ in \eqref{RSfactor}. 
In the non-Archimedean case, this ``$\Sym^2$-space'' of measures is relatively easy to describe, since the inverse of the local $L$-function $L(\chi,s)$ can be thought of as an element of the completed Hecke algebra of the torus $A_\ad$, which acts on $\mathcal S(\bar X/\SL_2)$, and can be used to ``kill'' the factor $L(\chi,s)$. In the Archimedean case, one needs a more delicate description based on Mellin transforms. The technical, although quite elementary, results that are needed for these definitions are presented in Section \ref{sec:2dim}, which the reader can skip at first reading, and consult when necessary.

Fiberwise Fourier transform on the symplectic vector bundle $\bar X$ defines an endomorphism of $\mathcal S(\bar X)$, which descends to an endomorphism $\mathcal H_X$ of $\mathcal S(\bar X/\SL_2)$ --- this endomorphism acts on the appropriate notion of relative characters by local gamma factors $$\gamma(\chi\times \Sym^2(\pi),1,\psi)\gamma(\chi,1,\psi).$$

In Theorem \ref{thmsubspaceX} we describe a factor $\mathcal H_X^\circ$ of $\mathcal H_X$ which preserves the subspace $\mathcal S(\bar X/\SL_2)^\circ$ and acts on relative characters by the local gamma factor 
$$\gamma(\chi\times \Sym^2(\pi),1,\psi).$$

The space $\mathcal S(\bar X/\SL_2)^\circ$, and its endomorphism $\mathcal H_X^\circ$, are the basis for all the constructions in this paper.

\subsection{Functional equation of the symmetric square $L$-function} \label{Exsym2}

The space $\mathcal S(\bar X/\SL_2)^\circ$ is used, in Section \ref{sec:sym2}, to construct a non-standard space of test measures $$\mathcal S^-_{L(\Sym^2, 1)} (N,\psi\backslash G /N,\psi)$$
for the Kuznetsov formula of the group $G=\Gm\times \SL_2$, via the \emph{unfolding method}:
 at the level of spaces ``upstairs'', the Rankin--Selberg method relies on ``unfolding'' the Rankin--Selberg period to the Whittaker model; this can be understood as a morphism
$$\U: \mathcal S(\bar X) \xrightarrow\sim \mathcal S^-(\tilde N,\tilde\psi\backslash \tilde G),$$
where $\tilde N = N\times N$ is a maximal unipotent subgroup of $\tilde G$, $\tilde\psi = \psi\times \psi^{-1}$, and $\mathcal S^-$ denotes some enlargement of the usual space of test measures. If we ``descend'' the morphism modulo $\SL_2^\diag$ and restrict to the ``$\Sym^2$-subspace'' $\mathcal S(\bar X/\SL_2)^\circ$, we obtain the desired map
$$ \bar\U: \mathcal S(\bar X/\SL_2)^\circ \xrightarrow\sim \mathcal S^-_{L(\Sym^2, 1)} (N,\psi\backslash G /N,\psi).$$
(Notice that $[\tilde N,\tilde\psi\backslash \tilde G/\SL_2^\diag] = [N,\psi\backslash G /N,\psi]$; the characters here define complex line bundles over the $F$-points of the stacks indicated.)

The unfolding map has various applications. First of all, taking $\Gm$-coinvariants, we obtain the geometric comparison between the Kunzetsov formula and the stable trace formula for $\SL_2$ (see Theorem \ref{Rudnickpt3}), mentioned already in the first part of this paper (Theorem \ref{thmRudnick}). The reason is that through the embedding $\frac{\SL_2}{N}\hookrightarrow[\bar X/\SL_2]$, inducing $\frac{\SL_2}{B}\hookrightarrow[\bar X/\SL_2\times \Gm]$, the push-forward of $\mathcal S(\SL_2)$ to $\Dfrac{\SL_2}{B} = \Dfrac{\SL_2}{\SL_2}$ coincides with the push-forward image of $\mathcal S(\bar X/\SL_2)^\circ$. Thus, the unfolding map descends to an isomorphism between the space $\mathcal S(\frac{\SL_2}{\SL_2})$ of test measures for the stable trace formula of $\SL_2$, and an extended space of test measures for the Kuznetsov formula of $G/\Gm=\SL_2$. 

Moreover, combining $\bar\U$ with the endomorphism $\mathcal H_X^\circ$ of $\mathcal S(\bar X/\SL_2)^\circ$, we compute the \emph{Hankel transform} $\mathcal H_{\Sym^2}$ for the symmetric-square $L$-function.

The Hankel transform is best described as a map between spaces of half-densities, so let us denote by 
$$\mathcal D^-_{L(\Sym^2, \frac{1}{2})} (N,\psi\backslash G /N,\psi)$$
the space of half-densities for the Kuznetsov formula that is analogous to $\mathcal S^-_{L(\Sym^2, 1)} (N,\psi\backslash G /N,\psi)$. (The notational difference $L(\Sym^2, 1)$ vs.\ $L(\Sym^2, \frac{1}{2})$ is due to a volume factor.)
This space contains a ``basic vector'' which corresponds to the Dirichlet series of the local unramified $L$-value $L(\Sym^2, \frac{1}{2})$. By inverting the $\Gm$-coordinate, we get the corresponding space 
$$\mathcal D^-_{L((\Sym^2)^\vee, \frac{1}{2})} (N,\psi\backslash G /N,\psi)$$
for the dual of the symmetric square representation. 

The Hankel transform is an isomorphism:
$$\mathcal H_{\Sym^2}: \mathcal D^-_{L(\Sym^2, \frac{1}{2})} (N,\psi\backslash G /N,\psi) \xrightarrow\sim \mathcal D^-_{L((\Sym^2)^\vee, \frac{1}{2})} (N,\psi\backslash G /N,\psi)$$
which satisfies the fundamental lemma for basic vectors, as well as several other properties, the most important of which being that it acts by the gamma factor of the local functional equation of the symmetric-square $L$-function on relative characters of irreducible generic representations:
$$\mathcal H_{\Sym^2}^* J_\pi = \gamma(\pi, \Sym^2, \frac{1}{2}, \psi) \cdot J_\pi.$$

I use the unfolding map and the operator $\mathcal H_X^\circ$ to compute an explicit formula for $\mathcal H_{\Sym^2}$, see Theorem \ref{thmSym2}: 
\begin{equation}\label{Hankel-Sym2-intro}
  \mathcal H_{\Sym^2} = \lambda(\eta_{\zeta^2-4},\psi)^{-1} \mathscr F_{-\check\lambda_+,\frac{1}{2}} \circ \delta_{1-4\zeta^{-2}} \circ \eta_{\zeta^2-4} \circ \mathscr F_{-\check\lambda_0,\frac{1}{2}} \circ \eta_{\zeta^2-4} \circ \mathscr F_{-\check\lambda_-,\frac{1}{2}},\end{equation}
where the notation is as follows:
\begin{itemize}
 \item Generic orbits for the quotient $N\backslash G/N$ are represented by the elements $\left(a, \begin{pmatrix} & -\zeta^{-1} \\ \zeta \end{pmatrix}\right)\in \Gm\times\SL_2$; the set of such elements is identified with the universal Cartan $A_G$ of $G$, with $\zeta$ being the value of the character $\frac{\alpha}{2}$.
 \item The coweights $\check\lambda_-, \check\lambda_0, \check\lambda_+$ into $A_G$ are the weights of the symmetric square representation, with the first being anti-dominant and the last dominant. The multiplicative Fourier convolutions $\mathscr F_{-\check\lambda_*,\frac{1}{2}}$ associated to those cocharacters were defined in \S \ref{sssFourierconv} of the first part of this paper.
 \item Then, we have the intermediate factors $\delta_{1-4\zeta^{-2}}$ and $\eta_{\zeta^2-4}$. The first is multiplicative translation by the factor $(1-4\zeta^{-2})$ along the $\Gm$-coordinate. The second is multiplication by the quadratic character associated to the extension $F(\sqrt{\zeta^2-4})$, again in the same variable. Finally, $\lambda(\eta_{\zeta^2-4},\psi)$ denotes a certain scalar attached to that extension, expressing a ratio of abelian gamma factors.
\end{itemize}

We notice that every one of the factors of the above operator satisfies, in principle, a global Poisson summation formula for the sum over rational pairs $(a,\zeta)$. There will be analytic intricacies in order to prove such a formula, but one can expect that a variant of the methods employed in \cite{SaBE2} will work, leading to a \emph{trace formula-theoretic} proof of the functional equation of the symmetric square $L$-function. Moreover, the formula is similar to a formula proved by Jacquet \cite{Jacquet} for the Hankel transform associated to the \emph{standard} $L$-function for $\GL_2$, see \S \ref{ssstandard} (which should also underlie the proof of the functional equation of the standard $L$-function by Herman in \cite{Herman}). The formula there reads:
 \begin{equation}
  \mathcal H_{\Std} = \mathscr F_{-\check\epsilon_1,\frac{1}{2}} \circ \psi(-e^{-\alpha}) \circ \mathscr F_{-\check\epsilon_2, \frac{1}{2}},
 \end{equation}
where $\check\epsilon_1$, $\check\epsilon_2$ are the weights of the standard representation of the dual group, and $e^{-\alpha}$ denotes the negative root character. 

In both cases, if we omit the intermediate factors and keep only the multiplicative Fourier convolutions, we obtain \emph{the Hankel transform for the degenerate Kuznetsov formula} (when $\psi$ is replaced by the trivial character), see \S \ref{ssstandard}; this is operator on functions on the universal Cartan $A_G$ of the group which acts by the gamma factor
$$\gamma(\chi, r\circ j, \frac{1}{2}, \psi)$$
on characters, where $j$ denotes the canonical embedding of $L$-groups $j: {^LA_G} \hookrightarrow \LG$, and $r=\Sym^2$ or $\Std$, depending on the case that we are considering. Hence, we see that there is a lot of structure in these Hankel transforms, which hopefully can be generalized to other $L$-functions.

\subsection{Functorial transfer from tori to $\GL_2$}

Finally, the ``symmetric square subspace'' $\mathcal S(\bar X/\SL_2)^\circ$ gives us insights into Venkatesh's thesis, and functorial lifts from tori to $\GL_2$ or, more correctly, to $G=\Gm\times\SL_2$, corresponding to a map of $L$-groups:
\begin{equation}\label{Lgps}{^LT} \to \LG = \Gm\times \PGL_2.
\end{equation}
In fact, it gives a direct link between the construction of Venkatesh and the endoscopic construction of the same functorial lift.

Here is how this works (see Section \ref{sec:Venkatesh}): We have already seen that the inverse $\bar\U^{-1}$ of the unfolding map is an isomorphism, by definition, between the ``symmetric square space'' $\mathcal S^-_{L(\Sym^2, \frac{1}{2})} (N,\psi\backslash G /N,\psi)$ of orbital integrals for the Kuznetsov formula, and the symmetric square subspace $\mathcal S(\bar X/\SL_2)^\circ$ of measures on the Rankin--Selberg variety. An element $\varphi$ of the latter is understood, as we have seen, as a measure on the affine plane $\mathbbm A^2$ with coordinates $(c,t)$, where $t$ denotes the trace. This measure is smooth away from $c= 0$, but its behavior as $c\to 0$ depends, it turns out, \emph{on the $\kappa$-orbital integrals of a Schwartz function $\Phi$ on $\SL_2$}, where ``$\kappa$-orbital integral'' means, for us, the usual orbital interal over a split regular semisimple orbit (i.e., when $t^2-4$ is a square), and the alternating sum of orbital integrals inside of a non-split, stable regular semisimple orbit. The function $\Phi$ is the restriction on the zero section $\SL_2\subset \bar X$ of a measure $\Phi dx$ which maps to $\varphi$ under the quotient map composed with a certain projector $\mathcal S(\bar X/\SL_2)\to \mathcal S(\bar X/\SL_2)^\circ$.

It turns out that, if we fix the trace coordinate $t\ne \pm 2$, a measure $\varphi \in \mathcal S(\bar X/\SL_2)^\circ$ is oscillating in the $c$-coordinate, as $c\to 0$, by (essentially) the quadratic character $\eta_{t^2-4}$ attached to the extension $F(\sqrt{t^2-4})$. Thus, fixing a character $\eta$, the asymptotic behavior of $\varphi$ at points where $\eta_{t^2-4}=\eta$ (these are the conjugacy classes represented by the torus $T$ associated to $\eta$!), can be ``captured'' by the pole of a Tate integral, as in Venkatesh's thesis. Locally, this detects poles of the symmetric-square $L$-function at zero; the good news is that these local poles manifest themselves geometrically, at the ``boundary'' $c=0$ of the Rankin--Selberg quotient $\bar X\sslash \SL_2$.

More precisely, the transfer operator
\begin{equation}\label{Ttransfer-intro}\mathcal S^-_{L(\Sym^2, 1)} (N,\psi\backslash G /N,\psi) \to \mathcal S(T),
\end{equation}
which is $(\Gm,\eta)$-equivariant (this is the $\Gm$-component of the dual map \eqref{Lgps}) is constructed as follows: First, we will apply the inverse $\bar\U^{-1}$ of the unfolding map, followed by the endomorphism $\mathcal H_X^\circ$ of $\mathcal S(\bar X/\SL_2)^\circ$, and a Tate zeta integral against the character $\eta$. The pole of the local Tate integral at $s=0$ will pick up the $\kappa$-orbital integrals of some Schwartz function on $\SL_2$ only at those conjugacy classes represented by elements of $T$. This establishes the link between the Kuznetsov formula for $G$ and the endoscopic parts of the trace formula for $\SL_2$ (Theorem \ref{transfertokappa}). By the endoscopic transfer of Labesse and Langlands, these are the same (up to scalar transfer factors) as the ``orbital integrals'', that is, the values, of some Schwartz function on $T$. This is the transfer map \eqref{Ttransfer-intro} behind Venkatesh's thesis. We verify that it satisfies our usual list of properties for transfer operators, namely the fundamental lemma for the Hecke algebra, and transfer of relative characters (Theorem \ref{thmVenkatesh}).

\subsection{Acknowledgements} 

As with the first part, this paper would not have been possible without the constant encouragement, numerous conversations, and many references and ideas provided by Ng\^o Bao Ch\^au, who invited me to spend the winter and spring quarters of 2017 at the University of Chicago. I also thank Daniel Johnstone for a presentation of Venkatesh's thesis which initiated my understanding of it. I thank Valentin Blomer for various references on related results in analytic number theory. Finally, I thank the referees for their corrections and suggestions for improving the exposition. 

This work was supported by NSF grants DMS-1502270, DMS-1801429, DMS-1939672, and by a stipend to the IAS from the Charles Simonyi Endowment.

\section{Orbital integrals and Mellin transforms on a 2-dimensional space} \label{sec:2dim}

This section collects relatively elementary, but technical facts about the Schwartz space of the quotient of a  2-dimensional vector space by a torus, Paley--Wiener spaces and Mellin transforms. The reader might use it as a reference, rather than reading it linearly.

\subsection{Review of orbital integrals for $\mathbbm A^2/T$}\label{ssA2}

Let $E/F$ be a quadratic extension (possibly split, i.e., $F\oplus F$), $V_E=\Res_{E/F}\Ga$ under the action of $T=\ker N^E_F$, where $N^E_F$ is the norm map. For the rest of this subsection, we feel free to denote $V_E$ simply by $V$ --- in our application to the symplectic bundle $\bar X\to \SL(V)$ in later subsections, $V_E$ will be the fiber over a regular semisimple point of $\SL(V)$. 

The norm map defines an isomorphism of $V\sslash T$ with $\Ga$. We let $\mathcal S([V/T])$ denote a simplified version of the Schwartz space of the quotient stack $[V/T]$ over $F$, defined in \cite{SaStacks}; namely, $\mathcal S([V/T])$ will be \emph{a space of measures on $\Ga(F)=F$}, equal to the push-forward of $H^0$ of the Schwartz complex attached to the stack $[V/T]$ in \cite{SaStacks}.\footnote{Lemma \ref{coinvariantA2} shows that this push-forward can actually be identified with $H^0$ of the Schwartz complex.} This is less complicated than it sounds, and the reader does not need to know the formalism of Schwartz spaces on stacks in order to understand the definition; here is an explicit description of this space: If $E$ is split, it coincides with the image of the push-forward map $\mathcal S(E)\xrightarrow{(N^E_F)_!} \Meas(F)$. However, if $E$ is non-split, it is the image of the sum of push-forwards: 
$$ \mathcal S(E) \oplus \mathcal S(E^\alpha) \xrightarrow{(N^E_F)_!} \Meas(V\sslash T).$$
Here, $\alpha$ stands for the class of a non-trivial $T$-torsor, and for such a $T$-torsor $R^\alpha$, we set $V^\alpha = V\times^T R^\alpha$ and $E^\alpha = V^\alpha(F)$. Then $E^\alpha$ is a free $E$-module of rank one, but without a distinguished base point. Moreover, the norm map $N^E_F$ canonically extends to $E^\alpha$, and has image equal to the non-norms (and zero). This construction is completely analogous to the construction of the Schwartz space of $[V/\{\pm 1\}]$ in \S \ref{scattorus}. 

Notice that $E^\alpha$ is a vector space and, hence, has a distinguished zero point, which will be denoted by $0^\alpha$; moreover, any Haar measure $dx$ on $E$ induces a Haar measure $dx^\alpha$ on the vector space $E^\alpha$ as follows: choose a base point $\kappa\in E^\alpha$ to identify $E\ni e\overset{\sim}\mapsto e\kappa\in E^\alpha$, and define $dx^\alpha(\kappa e) = |\kappa| dx(e)$, where, by definition, $|\kappa| = |N_F^E\kappa|$. 

Choosing a Haar measure $dt$ on $T$, we can define an push-forward (orbital integral) map $\Phi\mapsto O_\bullet (\Phi)$ from Schwartz functions on $E$ to functions on $F\smallsetminus\{0\}$. Given a Haar measure $dx$ on $E$, we have an integration formula
\begin{equation}\label{integration-2dim}
 \int_{\tilde V} \Phi(v) dx(v) = \int_F O_\xi(\Phi) d\xi,
\end{equation}
for a suitable choice of Haar measure $d\xi$ on $F$. The orbital integrals and the integration formula extend to $E^\alpha$, for a unique Haar measure $dx^\alpha$ on $E^\alpha$. In practice, of course, we have fixed a measure $d\xi$ on $F$, so we will make sure to choose the measures $dx$ on $E$ and $dt$ on $T$ compatibly. The space of functions on $F^\times$ obtained this way will be denoted by $\mathcal F([V/T])$.

We can also define a push-forward of Schwartz half-densities, by sending the half-density $\Phi(v) (dv)^\frac{1}{2}$ on $E$ to the half-density $O_\xi(\Phi) (d\xi)^\frac{1}{2}$ on $F^\times$ (and similarly for $E^\alpha$). By abuse of language, we will be saying ``half-density on $V\sslash T$'', although it is not defined at zero. The push-forwards of functions and half-densities depend on choices of measures, but the image spaces $\mathcal F([V/T])$, $\mathcal D([V/T])$ do not.

There is another, equivalent, description of the space $\mathcal S([V/T])$, which will be the one relevant to our intended application. Consider an embedding $T\hookrightarrow \SL_2$, and the variety $\tilde V=V\times^T \SL_2$. Then the stack quotients $[V/T]$ and $[\tilde V/\SL_2]$ coincide, hence by \cite[Theorem 2]{SaStacks-erratum} have the same Schwartz space. In this case, this means the following: The affine quotient $\tilde V\sslash \SL_2$ coincides with $V\sslash T\simeq \mathbbm A^1$, and $\mathcal S([V/T])$ is equal to the image of the push-forward map: $\mathcal S(\tilde V)\to \Meas(F)$. If instead we work with functions on $\tilde V$, and $\SL_2$-orbital integrals, the integration formula \eqref{integration-2dim} remains true for an appropriate choice of Haar measure on $\SL_2$.

The elements of $\mathcal S([V/T])$ can be explicitly described as follows, cf.\ \cite[Propositions 2.5 and 2.14]{SaBE1}: they are smooth measures of rapid decay away from a neighborhood of $0\in \Ga(F)$, and in a neighborhood of zero they are of the form
\begin{equation}f(\xi) = C_1(\xi) + C_2(\xi) \eta(\xi),\label{germs-non-split}
\end{equation}
if $E$ is non-split, with $\eta$ the quadratic character associated to $E$, and 
\begin{equation} f(\xi) = C_1(\xi) + C_2(\xi) \log(|\xi|), \label{germs-split}
\end{equation}
 if $E$ is split, with $C_1, C_2$ two smooth measures.

We would like to think of elements of $\mathcal S([V/T])$ as sections of a \emph{cosheaf} over $\Ga(F)=F$. In particular, we define the \emph{fiber} $\mathcal S([V/T])_0$ as the quotient $\mathcal S([V/T])/I\mathcal S([V/T])$, where $I\subset C^\infty_\temp(F)$ is the ideal of those smooth, tempered functions that vanish at $0$. (The word ``tempered'' means ``moderate growth for all derivatives'', and it is included to ensure, in the Archimedean case, that these functions act on $\mathcal S([V/T])$.) Explicitly,  $\mathcal S([V/T])_0$ is a two-dimensional space, whose dual is spanned, in the notation of \eqref{germs-non-split}, \eqref{germs-split}, by the distributions
$$ f\mapsto \frac{C_1}{d\xi}(0), \,\,\, \mbox{ and } f\mapsto \frac{C_2}{d\xi}(0).$$

For any of the above descriptions of the space $\mathcal S([V/T])$, the push-forward map can be identified with the coinvariant quotient: 

\begin{lemma}\label{coinvariantA2}
 The quotient maps $\mathcal S(E)\oplus \mathcal S(E^\alpha) \to \mathcal S([V/T])$ (with $E^\alpha$ to be ignored in the split case), resp.\ $\mathcal S(\tilde V)\to \mathcal S([V/T])$ identify $\mathcal S([V/T])$ with the $T$-coinvariant space, resp.\ $\SL_2$-coinvariant space, of the source.
\end{lemma}

The coinvariant space, in the Archimedean case, is defined as the quotient of the original space by the \emph{closure} of the subspace spanned by vectors of the form $v-g\cdot v$.

\begin{proof}
 This is stated as Lemma 2.3 in \cite{SaBE1}, but a couple of typos and omissions oblige me to briefly revisit this argument: First, the coinvariant space $\mathcal S(\tilde V)_{\SL_2}$ (and similarly for the other model) can be seen as the space of sections of a \emph{cosheaf} over $\Ga= V\sslash T$. The quotient $[\tilde V/\SL_2]\to \Ga = \tilde V\sslash \SL_2 =V\sslash T$ is a stack isomorphism away from $0\in \Ga$, and this immediately identifies the subspace $\mathcal S(F^\times)$ with the coinvariants of the Schwartz space of $\tilde V^\circ:=$ the preimage of $\Ga\smallsetminus\{0\}$. Then, one checks that the distributions $f\mapsto \frac{C_1}{d\xi}(0)$ and $f\mapsto \frac{C_2}{d\xi}(0)$ above span the space of distributions on the \emph{fiber} of the cosheaf $\mathcal S(\tilde V)_{\SL_2}$ over $0$, i.e., on the space $\mathcal S(\tilde V)_{\SL_2}/I\mathcal S(\tilde V)_{\SL_2}$, where $I\subset C^\infty_\temp(F)$ is the ideal of smooth, tempered functions on $F$ vanishing at zero. This completes the proof in the non-Archimedean case, where the fiber coincides with the \emph{stalk} $\mathcal S(\tilde V)_{\SL_2}/\mathcal S(\tilde V^\circ)_{\SL_2}$. In the Archimedean case, \cite[Proposition B.4.1]{SaBE1} implies that the stalk is generated by any set of generators of the fiber under the action of $C^\infty_\temp(F)$, and this implies that the kernel of the map $\mathcal S(\tilde V)_{\SL_2}\to \mathcal S([V/T])$ is trivial.
\end{proof}

\subsection{Paley--Wiener spaces}\label{ssPaleyWiener}

Let $T$ be a torus over $F$. We want to describe various spaces of measures on $T=T(F)$ via  their Mellin transforms, generalizing the Paley--Wiener Theorem \ref{PW-torus} for the Schwartz space of $T$, and its extension of Proposition \ref{Tateimage} to the Schwartz space of the affine line. For that purpose, we generalize the Paley--Wiener space of Proposition \ref{Tateimage} as follows:

Let $\rho:{^LT}\to \GL_n(\mathbb C)$ be a diagonalizable representation of the $L$-group of $T$. Here, we are using the version ${^LT} = \check T\rtimes \mathcal W_F$ of the $L$-group which is an extension of the Weil group of $F$, but for simplicity of exposition we will only allow \emph{non-negative half-integer twists of the Galois version}, that is, we have $\rho = \bigoplus_{i=1}^n \rho_i$, where $\rho_i = $ the product of a character of $\check T \rtimes \Gal(\bar F/F)$ by the character $|\bullet|^s$ of the Weil group, where $|\bullet|$ is the absolute value (normalized, in the non-Archimedean case with residue field of order $q$, to send a Frobenius element to $q^{-1}$), and $2s \in \mathbb Z_{\ge 0}$. 

This gives rise to a local $L$-function $\chi\mapsto L(\chi, \rho, 0)$, which  we will consider as a function on the character group $\widehat{T_\CC}$. (Recall that $\widehat{T_\CC}$ denotes the group of all complex characters of $T$, and $\widehat T$ its subgroup of unitary characters.) Given the possibility of Weil twists, the point of evaluation $0$ is really a red herring, and everything that follows will be applied later, without further explanation, to $L$-values of the form $L(\chi, \rho, s)$, simply by understanding them as $L(\chi, \rho |\bullet|^s, 0)$. 

Our restrictions on $s$ imply the following about the poles of $L(\chi, \rho, 0)$ in the Archimedean case: Recall that there is a canonical factorization of the character group of $T$ as $\widehat{T_0} \times \mathfrak t_{\CC}^*$ (see \eqref{dualArchimedean}), with $T_0=$the maximal compact subgroup of $T$. We will define a meromorphic, non-vanishing function $\mathfrak G_\rho$ on $\mathfrak t_{\CC}^*$, such that, for every element of the discrete subgroup $\widehat{T_0}$, the restriction of $L(\chi,\rho,0)$ to the coset represented by this element is a holomorphic multiple of $\mathfrak G_\rho$. This function is 
\begin{equation}\label{Grho}
 \mathfrak G_\rho := \prod_i \rho_i^* \mathfrak G,
\end{equation}
where $\mathfrak G$ was defined in \S \ref{ssTate} as $\Gamma(s)$, when $F=\mathbb R$, and $\Gamma(2s)$, when $F=\mathbb C$, and by $\rho_i^*: \mathfrak t_{\CC}^* \to \mathfrak g_{m,\CC}^* = \mathbb C$ we denote the map induced by the restriction of the one-dimensional factor $\rho_i$ of $\rho$ to $\check T$. 

We state this as a lemma:
\begin{lemma}\label{lemmapoles}
For $F$ Archimedean, the function $L(\chi, \rho, 0)$ on $\widehat{T_\CC}$ is a holomorphic multiple of the pullback of $\mathfrak G_\rho$ from the factor $\mathfrak t_{\CC}^*$. 
\end{lemma}

\begin{proof}
It is enough to treat the case when $\rho$ is one-dimensional; then, by our assumptions, it has the form
\[ \check T \rtimes \mathcal W_F \ni (r,w)\mapsto \lambda(r) \mu (w) |w|^s \in \mathbb C^\times,\]
where $\lambda$ is a Galois-stable character of $\check T$, $\mu$ is a character of the Galois group pulled back to the Weil group, and $s$ is a non-negative half-integer. We can consider $\lambda$ as a cocharacter $\Gm\to T$, and $\mu$ as a (quadratic, at most) character of $F^\times$, and then the local $L$-factor can be written 
\[ L(\chi, \rho, 0) = L((\chi\circ\lambda) \cdot \mu, s).\]
The statement then follows from the analogous statement on Archimedean $L$-factors of Hecke characters, recalled in \S \ref{ssTate}.
\end{proof}

\begin{definition}\label{def:PaleyWiener}
The Paley--Wiener space
$$ \mathbb H^\PW_{L(\bullet,\rho,0)} (\widehat{T_\CC})$$
of ($\CC$-valued) functions on the character group $\widehat{T_\CC}$ consists of polynomial multiples, in the non-Archimedean case, or holomorphic multiples, in the Archimedean case, of the function
$\chi \mapsto L(\chi, \rho, 0)$, which have the following properties: 
 \begin{itemize} 
  \item in the non-Archimedean case, they are supported on a finite number of connected components of $\widehat{T_\CC}$;
  \item in the Archimedean case, factoring the character group of $T$ as $\widehat{T_0} \times \mathfrak t_{\CC}^*$, as above, they belong to the completed tensor product
  $$ \mathscr C(\widehat{T_0}) \hat\otimes \mathbb H^\PW_{\mathfrak G_\rho}(\CC),$$ 
  where $\mathbb H^\PW_{\mathfrak G_\rho}(\CC)$ denotes the Fr\'echet space of holomorphic multiples of the function $\mathfrak G_\rho$, which are of rapid decay in bounded vertical strips, and $\mathscr C(\widehat{T_0})$ is the Harish-Chandra--Schwartz space of the discrete abelian group $\widehat{T_0}$, consisting of those of functions $\varphi$ such that, for any norm $\Vert \bullet \Vert$ on the vector space $\widehat{T_0}\otimes_{\Z} \RR$, and any $N\ge 0$, the function $\Vert n \Vert^N \varphi(n)$ is bounded.
  
  In the Archimedean case, the space $\mathbb H^\PW_{L(\bullet,\rho,0)} (\widehat{T_\CC})$ inherits the structure of a  Fr\'echet space, as a closed subspace of $\mathscr C(\widehat{T_0}) \hat\otimes \mathbb H^\PW_{\mathfrak G_\rho}(\CC)$.
 \end{itemize}
\end{definition}

\subsection{Mellin transforms, and the space $\mathcal S([V/T])^\circ$}

We return to the 2-di\-mensional space of \S \ref{ssA2} with an action of the torus $T$, to study the Mellin transforms of measures.
 
For any $f\in \mathcal S([V/T])$, we define its \emph{Mellin transform} $\check f(\chi)$, where $\chi$ varies in the characters of $F^\times$, as
\begin{equation}\label{MellinVTdef} \check f(\chi) = \int_{F^\times} f(x) \chi^{-1}(x)|x|^{-\frac{1}{2}},
\end{equation}
for $\Re(\chi)\ll 0$ (i.e., $\chi(x) = \chi_0(x) |x|^s$ for some unitary $\chi_0$  and $\Re(s)\ll 0$ --- $\Re(s)<1$ suffices here), and by meromorphic continuation in general. The shift by $|x|^{-\frac{1}{2}}$ is such that, if we normalize the action of the group $A:=\Gm$ of scalar dilations on $\mathcal S(V)$ to be unitary:
\begin{equation}\label{normVmodT} a\cdot f(v) = |a|^{-1} f(v),
\end{equation}
and normalize the action of its quotient $A_\ad:=A/\{\pm 1\} \simeq \Gm$ on $\mathcal S([V/T])$ accordingly, the map $f\mapsto \check f(\chi)$ is $(A_\ad,\chi)$-equivariant.

The following proposition describes Mellin transforms of the elements of $\mathcal S([V/T])$, and the fiber at $0$ in terms of those, generalizing Proposition \ref{Tateimage}. 

\begin{proposition} \label{VmodTimage}
Mellin transform defines an isomorphism between $\mathcal S([V/T])$ and the space 
$$ \mathbb H^\PW_{L(\bullet^{-1}, \frac{1}{2}) L(\bullet^{-1}\eta, \frac{1}{2})} (\widehat{F^\times_\CC}).$$
 
 Moreover, if $f\in \mathcal S([V/T])$ is the image of a measure $\Phi dx \in \mathcal S(E)$, in the split case, and $(\Phi dx, \Phi^\alpha dx^\alpha) \in \mathcal S(E)\oplus \mathcal S(E^\alpha)$, in the non-split case (where the measures $dx$ and $dx^\alpha$ correspond as above), and $f$ has the form of \eqref{germs-non-split} or \eqref{germs-split}, we have the following expressions for the coefficients of the Laurent expansion of $\check f$ at $\chi=1$ and $\eta$: 
 \begin{itemize}
  \item in the non-split case, 
  \begin{align} 
   \Res_{s=\frac{1}{2}} \check f(|\bullet|^s) &= - \AvgVol(F^\times) \frac{C_1}{d\xi} (0) = - \AvgVol(E^\times) (\Phi(0) + \Phi(0^\alpha)) , \mbox{ and } \nonumber
  \\
 \label{residue-non-split} \Res_{s=\frac{1}{2}} \check f(\eta |\bullet|^s) &= - \AvgVol(F^\times) \frac{C_2}{d\xi} (0) =  - \AvgVol(E^\times) (\Phi(0)-\Phi(0^\alpha));
  \end{align}
  \item in the split case, 
  \begin{align} 
   \lim_{s\to \frac{1}{2}} (s-\frac{1}{2})^2  \check f(|\bullet|^s) =  \Res_{s=\frac{1}{2}} \widecheck{C_2}(|\bullet|^s) & = -\AvgVol(F^\times) \frac{C_2}{d\xi}(0) = \AvgVol(E^\times) \Phi(0),  \nonumber
  \\
 \label{residue-split} \mbox{and if }C_2=0, \,  \Res_{s=\frac{1}{2}} \check f(|\bullet|^s) &= - \AvgVol(F^\times) \frac{C_1}{d\xi}(0) .
  \end{align}
 \end{itemize}
 
\end{proposition}

The notion of ``average volume'' was introduced in \eqref{avgvol}. Note that \- $L(\bullet^{-1}, \frac{1}{2}) L(\bullet^{-1}\eta, \frac{1}{2})$ is shorthand (in the context of Definition \ref{def:PaleyWiener}) for $L(\bullet, \rho, 0)$, where $\rho = (\Std^\vee \oplus \Std^\vee\otimes \hat\eta)\otimes |\bullet|^\frac{1}{2}$, where $\hat\eta$ is the associated quadratic character of the Weil/Galois group.

\begin{proof}
The pullback of Mellin transform $\check f(\chi)$ to the space $\mathcal S(E)$ is the Tate integral 
$$ Z(\Phi, \chi^{-1}\circ N_F^E, \frac{1}{2}-s)$$
on $E^\times$. The results now follow from Proposition \ref{Tateimage} (and an easy adaptation to the space $E^\alpha$), except for the relations with the coefficients $C_1, C_2$, which can be established as follows:

In the non-split case, the Mellin transform of the sum $C_1(\xi) + C_2(\xi) \eta(\xi)$ can be broken up into two Tate integrals on $F^\times$, and the result follows again by Proposition \ref{Tateimage}; same for the split case, when $C_2=0$.

For the relation of $\frac{C_2}{d\xi}(0)$ with $\AvgVol(E^\times) \Phi(0)$ in the split case, see \cite[Proposition 2.5]{SaBE1}.
\end{proof}

We can also express the functionals above in terms of the space $\tilde V = V\times^T \SL_2$. Let $[v,g]$ represent the class of an element of $V\times \SL_2$ in $\tilde V$.

\begin{corollary}\label{corresidues}
Fix an invariant measure on $T\backslash G$, and use it to construct, for any Haar measure $dx$ on $V$, an invariant measure $\widetilde{dx}$ on $\tilde V$.

In the split case, if $f\in \mathcal S([V/T])$ is the image of an element $\Phi \widetilde{dx} \in \mathcal S(\tilde V)$, then we have
\begin{equation}\label{residue-split-group}
    \lim_{s\to \frac{1}{2}} (s-\frac{1}{2})^2  \check f(|\bullet|^s) =\AvgVol(E^\times) \int_{T\backslash \SL_2} \Phi([0,g]) dg,
\end{equation}
where $\AvgVol(E^\times)$ is taken with respect to the measure $d^\times x$. 

In the non-split case, if $f\in \mathcal S([V/T])$ is the image of an element $\Phi \tilde{dx} \in \mathcal S(\tilde V)$, then we have
  \begin{align} 
   \Res_{s=\frac{1}{2}} \check f(|\bullet|^s) &= - \AvgVol(E^\times) \int_{(T\backslash \SL_2)(F)} \Phi([0,g]) dg, \mbox{ and }
 \nonumber \\
 \label{residue-non-split-group} \Res_{s=\frac{1}{2}} \check f(\eta |\bullet|^s) &=- \AvgVol(E^\times) \int_{(T\backslash \SL_2)(F)} \Phi([0,g]) \eta(g) dg,
  \end{align}
  where, on the right hand side, $\eta$ is identified with the function on $(T\backslash \SL_2)(F)$ which is equal to $1$ on the $\SL_2(F)$-orbit represented by the identity, and $-1$ on the other orbit.
\end{corollary}

In our application, $T\backslash \SL_2$ will correspond to a stable semisimple conjugacy class in the group $\SL_2$, in which case the integrals on the right hand side of \eqref{residue-split-group}, \eqref{residue-non-split-group} are \emph{stable} and ``\emph{$\kappa$-orbital integrals}''.

We let $\mathcal S([V/T])^\circ \subset\mathcal S([V/T])$ denote the subspace whose elements are of the form
\begin{equation}\label{germs-circ}
 C(\xi) \eta(\xi)
\end{equation}
in a neighborhood of zero, in both the split and non-split cases, where $C$ is a smooth measure. The corresponding subspaces of functions and half-densities will be denoted by $\mathcal F([V/T])^\circ$, $\mathcal D([V/T])^\circ$, respectively.

Proposition \ref{VmodTimage} easily implies:

\begin{corollary} \label{MellinVcirc}
Mellin transform defines an isomorphism between $\mathcal S([V/T])^\circ$ and the space 
$$ \mathbb H^\PW_{L(\bullet^{-1}\eta, \frac{1}{2})} (\widehat{F^\times_\CC}).$$
 
 Moreover, if $f\in \mathcal S([V/T])^\circ$ is of the form \eqref{germs-circ} in a neighborhood of zero, then
   \begin{equation}\label{residue-circ}\Res_{s=\frac{1}{2}} \check f(\eta|\bullet|^s) = - \AvgVol(F^\times) \frac{C}{d\xi} (0).
   \end{equation}

In the non-Archimedean case, let $h$ be the element of the completed Hecke algebra $\widehat{\mathcal S(F^\times)}$ (see \S \ref{ssmultipliers}) whose Mellin transform is $\check h(\chi)=L(\chi, \frac{1}{2})^{-1}$. Then, the \emph{normalized} action  of $h$ (descending from \eqref{normVmodT}) gives rise to an isomorphism:
$$ \mathcal S([V/T])\xrightarrow{h\cdot }\mathcal S([V/T])^\circ.$$

\end{corollary}

\begin{proof}
 By its definition, and Proposition \ref{VmodTimage}, the space $\mathcal S([V/T])^\circ$ corresponds precisely to the subspace 
 $$ \mathbb H^\PW_{ L(\bullet^{-1}\eta, \frac{1}{2})} (\widehat{F^\times_\CC})\subset
 \mathbb H^\PW_{L(\bullet^{-1}, \frac{1}{2}) L(\bullet^{-1}\eta, \frac{1}{2})} (\widehat{F^\times_\CC})$$
under Mellin transform. 

The residue formula \eqref{residue-circ} follows from \eqref{residue-non-split} and \eqref{residue-split}.

In the non-Archimededean case, the function $L(\chi, \frac{1}{2})^{-1}$ is polynomial in $\chi$, hence is the Fourier transform of an element of $\widehat{\mathcal S(F^\times)}$. For the normalized action \eqref{normVmodT} on $\mathcal S(V)$ (we don't normalize the action on $\mathcal S(\Ga)$), we have $\widecheck{(h\cdot f)}(\chi)= \check h(\chi^{-1}) \check f(\chi)$ for any $h\in \widehat{\mathcal S(F^\times)}$, hence the result follows.
\end{proof}

In the Archimedean case, the corresponding element is not a multiplier (the reciprocal of the Gamma function is not of polynomial growth in vertical strips), and that creates some technical complications when studying the space $\mathcal S([V/T])^\circ$ and all related spaces.

\subsection{Hankel transform on $\mathbbm A^2/T$.}

We will define two Fourier transforms on a two-dimensional vector space $V$: one, denoted by $\mathscr F$, by identifying $V$ with the additive group of a quadratic extension $E$; and another, denoted $\mathfrak F$, by endowing $V$ with a symplectic structure.

For now, we continue as in the previous subsection, where $V=\Res_{E/F} \Ga$.
We use the trace pairing $(x,y)\mapsto \tr(xy)$ to identify $V$ with its linear dual,  
but keeping in mind that the $\Gm$-action is inverted under this identification.

Consider Fourier transform, as an endomorphism of the space $\mathcal D(V)$ of Schwartz half-densities on $V$,
defined by the formula
$$\mathscr F (\Phi(x)(dx)^\frac{1}{2}) (y) = \left(\int_V \Phi(x) \psi(\tr(xy)) dx\right) \cdot  (dy)^\frac{1}{2},$$
using a self-dual measure on $V$ with respect to the character $\psi\circ \tr$, or, equivalently, 
\begin{equation}\label{FourierE}  \mathscr F(\varphi)(y) = \int_{E^\times} \varphi(xy^{-1}) \psi(\tr(x)) |x|^\frac{1}{2} d^\times x.\end{equation}
Note that, since we are using the orthogonal pairing $(x,y)\mapsto \psi(\tr(xy))$ to define Fourier transform on $V$, we are using different notation, $\mathscr F$, from the Fourier transform $\mathfrak F$ defined by a symplectic pairing.

When $E/F$ is non-split, \eqref{FourierE} gives a natural way to extend this transform to the ``pure inner form'' $E^\alpha$, so Fourier transform becomes an endomorphism of $\mathcal D(E)\oplus \mathcal D(E^\alpha)$.
The transform is anti-equivariant with respect to the action of $T$ on $V$,  
hence descends to $T$-coinvariants, which can be identified with the space $\mathcal D([V/T])$ (Lemma \ref{coinvariantA2}).  
We denote the descent by $\mathscr H$, for ``Hankel'':
$$ \mathscr H : \mathcal D([V/T])\xrightarrow\sim \mathcal D([V/T]).$$
Let $\lambda(E/F,\psi)$ be the scalar that satisfies
\begin{equation}\label{lambda} \gamma(1, s, \psi)\gamma(\eta, s,\psi) = \lambda(E/F,\psi) \gamma_E(1,s, \psi\circ\tr)
\end{equation}
for the gamma factors of the Tate zeta integrals (see \S \ref{ssTate}), the one on the right hand side being on the field $E$; this is the $\lambda$-constant defined in \cite{Langlands-epsilon}, cf.\ also \cite[Lemma 1.1]{JL}.  We will also denote this scalar by $\lambda(\eta,\psi)$, where $\eta$ is the quadratic character associated to the extension $E/F$.

\begin{proposition}\label{propHankel2D}
 Fixing the norm coordinate for $V\sslash T$, we have 
\begin{align} \mathscr H \varphi(\xi^{-1}) &= \int^*_{E^\times} \varphi(\xi N^E_F(x)) \psi(\tr(x)) |x|^\frac{1}{2} d^\times x  \nonumber \\
& = \int^*_{F^\times} \varphi(\xi z) (N^E_F)_!\left(\psi(\tr(\bullet)) |\bullet|^\frac{1}{2} d^\times \bullet\right)(z)\end{align}
(for $\varphi \in \mathcal D([V/T])$),
where the notation $\int^*$ means that the integral, which does not, in general, converge, should be interpreted as a Fourier transform of a generalized half-density (or, equivalently, of an $L^2$-half density). Moreover, we have
\begin{equation}\label{Hankel2D}
 \mathscr H \varphi (\xi^{-1}) = \lambda(E/F,\psi)^{-1}(\psi(\frac{1}{\bullet}) |\bullet|^{-\frac{1}{2}} d^\times\bullet) \star (\psi(\frac{1}{\bullet}) \eta(\bullet)|\bullet|^{-\frac{1}{2}} d^\times\bullet) \star \varphi (\xi),
\end{equation}
where $\eta$ is the quadratic character of $F^\times$ associated to $E$, and the convolution is understood in the regularized sense of \S \ref{sssFourierconv}.
\end{proposition}

\begin{proof}
 This was proven in \cite[Propositions 2.15 and 2.16]{SaBE1}, but was missing the factor $\lambda$. (The mistake originates in an erroneous version of \eqref{lambda} in \cite[(2.29)]{SaBE1}.) This scalar of course is $1$ in the split case $E=F\oplus F$; in the non-split case, it can be computed by taking the leading coefficients at $s=1$ and employing \eqref{AvgVolFtimes}:
\begin{equation}\label{lambdaconstant}
\lambda(E/F,\psi) =   \gamma(\eta, 1, \psi) \frac{\gamma^*(1,1,\psi)}{\gamma_E^*(1,1,\psi\circ\tr)} = \gamma(\eta, 1, \psi) \frac{\AvgVol(F^\times)}{\AvgVol(E^\times)} = \frac{2\gamma(\eta, 1, \psi)}{\Vol(T)},
\end{equation}
where the average volume of $F^\times$ is taken with respect to the measure $\frac{dx}{x}$, where $dx$ is self-dual for the character $\psi$, the average volume of $E^\times$ is similarly taken using the self-dual measure for $\psi\circ \tr$, and $T$ is the kernel of the norm map, endowed with the fiber measure of the latter with respect to the former.

\end{proof}

\begin{remark}\label{remarkchangeorder}
 This formula is the result of the following formal calculation, for functions:
\begin{align}  
\int_T \int_E \Phi(x) \psi(\tr(xt\tilde \xi^{-1})) dx dt  &\overset{*}{=} |\xi|^{-1} \int_E^* \int_T \Phi(x t^{-1} \tilde\xi) dt  \psi(\tr(x)) dx \\
& =|\xi|^{-1} \int^*_{F^\times} O_{z\xi}(\Phi) (N^E_F)_! \left(\psi\circ\tr(\bullet) d\bullet\right)(z),
 \end{align}
where $\tilde\xi$ is a lift of $\xi\in V\sslash T$ to $V$, and we have set $O_\xi(\Phi) = \int_T \Phi(\tilde \xi t) dt$. By $(N^E_F)_! \left(\psi\circ\tr(\bullet) d\bullet\right)$ we denote the push-forward of the distribution $\psi\circ\tr(x) dx$ via the norm map, which can be decomposed into a convolution of distributions as in \eqref{Hankel2D}. I leave it to the reader to check the details of the exponents when one translates from functions to half-densities. The point of this remark is that, while interchanging the two integrals in the step denoted by $\overset{*}{=}$ is not justified in the sense of convergent integrals, it is justified when we interpret the resulting exterior integral in terms of Fourier transforms of generalized half-densities (or $L^2$-half densities).
\end{remark}

We now study the subspace $\mathcal D([V/T])^\circ \subset \mathcal D([V/T])$, defined by the asymptotic condition \eqref{germs-circ} in a neighborhood of zero.  

\begin{proposition}\label{H02D}
 The map 
 $$\mathscr H^\circ \varphi(\xi^{-1}) := \lambda(E/F,\psi)^{-1}(\psi(\frac{1}{\bullet}) \eta(\bullet)|\bullet|^{-\frac{1}{2}} d^\times\bullet) \star \varphi(\xi),$$
 where the convolution is understood again in the regularized sense of \S \ref{sssFourierconv}, is an automorphism of the space
 $ \mathcal D([V/T])^\circ$.
\end{proposition}

\begin{proof}
This follows immediately from Corollary \ref{MellinVcirc} and the fact that regularized convolution by $(\psi(\frac{1}{\bullet}) \eta(\bullet)|\bullet|^{-\frac{1}{2}} d^\times\bullet) $ multiplies Mellin transforms by the factor 
$\gamma(\chi^{-1}\eta_E, \frac{1}{2}, \psi)$, see \eqref{FE}. 
\end{proof}

Now assume that the same space $V$ is also endowed with an alternating form $\omega$. There is a unique $D\in E\smallsetminus\{0\}$ with $\tr(D)=0$ such that 
$$\omega(x, y) = \tr(D x \bar y),$$
where $\bar y$ denotes the Galois conjugate of $y$. A standard basis for this alternating form is the pair $(1,-(2D)^{-1})$, and the discriminant of the quadratic form $x\mapsto N_F^E(x)$ in such a basis is $- \frac{1}{4D^2}$. 

The symplectic Fourier transform $\mathfrak F$, defined as in \eqref{Fourierconvention} but now for half-densities:
$$ \mathfrak F(\Phi(v)|\omega|^\frac{1}{2}(v)) (v^\vee) = \left(\int_V \Phi(v) \psi(\omega(v, v^\vee)) |\omega|(v)\right) \cdot  |\omega|^\frac{1}{2}(v^\vee)$$
is related to the ``orthogonal'' Fourier transform by
\begin{equation}\label{symplorthoFourier}
 \mathfrak F\varphi (y) = \mathscr F\varphi(D\bar y).
\end{equation}

Again, the transform extends canonically to the ``pure inner form'' $E^\alpha$, in a completely analogous way as in \S \ref{scattorus}, and it keeps satisfying \eqref{symplorthoFourier}.
In particular, since $y$ and $\bar y$ are equivalent modulo $T$, it descends to the space $[V/T]$, and we have a commutative diagram

\begin{equation}\label{FourierHankelV}
 \xymatrix{
 \mathcal D(E)\oplus \mathcal D(E^\alpha) \ar[r]^{\mathfrak F} \ar[d] & \mathcal D(E)\oplus \mathcal D(E^\alpha) \ar[d] \\
 \mathcal D([V/T]) \ar[r]^{\mathcal H} & \mathcal D([V/T]),
 }
\end{equation}
(the space $E^\alpha$ is to be ignored in the split case), with the \emph{Hankel transform} $\mathcal H$ given by the following formula:
\begin{multline*} \mathcal H\varphi(\xi^{-1}) = \mathscr H\varphi((N_F^E D) \xi^{-1}) \\
= \lambda(E/F,\psi)^{-1}(\psi(\frac{1}{\bullet}) |\bullet|^{-\frac{1}{2}} d^\times\bullet) \star (\psi(\frac{1}{\bullet}) \eta_E(\bullet)|\bullet|^{-\frac{1}{2}} d^\times\bullet) \star \varphi ((N_F^E D)^{-1}\xi),
\end{multline*}
by applying \eqref{Hankel2D}.

We summarize the results, by taking a different point of view: assume that we start with a symplectic space $(V,\omega)$, and a quadratic form $Q$ on it of discriminant $d$. Then we can identify $V$ with the space of the quadratic extension $F(\sqrt{-d})$, in such a way that $Q$ corresponds to the norm map. Then we apply the above with $d = -\frac{1}{4D^2}$, and we get:

\begin{proposition}\label{symplecticHankel}
 Let $(V,\omega)$ be a symplectic space, $Q$ a quadratic form on it of discriminant $d$, and $T=\SO(V,Q)$. Identify the quotient $V\sslash T$ with $\Ga$ via $Q$, then the symplectic Fourier transform $\mathfrak F$ on $\mathcal D(V)$ (or $\mathcal D(V)\oplus \mathcal D(V^\alpha)$, if $T$ is non-split) descends to a Hankel transform on $\mathcal D([V/T])$, and we have the formula:
\begin{equation}\label{Hankel2D-sympl} \mathcal H\varphi(\xi^{-1}) = \lambda(\eta,\psi)^{-1} (\psi(\frac{1}{\bullet}) |\bullet|^{-\frac{1}{2}} d^\times\bullet) \star (\psi(\frac{1}{\bullet}) \eta(\bullet)|\bullet|^{-\frac{1}{2}} d^\times\bullet) \star \varphi (4d \cdot \xi),
\end{equation}
where $\eta$ is the quadratic character associated with the extension $E=F(\sqrt{-d})$, and $\lambda(\eta,\psi) = \lambda(E/F,\psi)$ for the same extension.
\end{proposition}

For the subspace $\mathcal D([V/T])^\circ$, we have the following corollary of Proposition \ref{H02D}:

\begin{corollary}\label{Hcirc}
 The map 
 \begin{equation}\label{Hankel2D-sympl-circ}\mathcal H^\circ \varphi(\xi^{-1}) := \lambda(\eta,\psi)^{-1}(\psi(\frac{1}{\bullet}) \eta(\bullet)|\bullet|^{-\frac{1}{2}} d^\times\bullet) \star \varphi(4d\cdot \xi),
 \end{equation}
 where the convolution is understood again in the regularized sense of \S \ref{sssFourierconv}, is an automorphism of the space
 $ \mathcal D([V/T])^\circ$.
\end{corollary}

\subsection{The Hankel transform in terms of Mellin transforms}\label{ssVcircMellin}

We will also need another description of the Hankel transform $\mathcal H^\circ$, which presents it as the descent of some ``Fourier transform'', as in \eqref{FourierHankelV}. For this, we need to find a way to ``erase'' the factor $(\psi(\frac{1}{\bullet}) |\bullet|^{-\frac{1}{2}} d^\times\bullet) $ of \eqref{Hankel2D-sympl} (compare with \eqref{Hankel2D-sympl-circ}) already from the Fourier transform $\mathfrak F$. Unfortunately, since the action of the multiplicative group $A_\ad$ on $V\sslash T$ does not lift to $V$, this is not directly possible on $\mathcal S(V)$, and we will need to use the language of Schwartz spaces on stacks, and the invariance of the Schwartz space on the presentation of a stack (see \cite[Theorem 2]{SaStacks-erratum}). 

Namely, we will use the isomorphism of stacks:
$$ [V/T] = [V_\ad/T_\ad],$$
where $V_\ad=[V/\{\pm 1\}]$ and $T_\ad = T/\{\pm 1\}$. (The notation is adapted to our later use where $T$ will be a torus in $\SL_2$, and $T_\ad$ its image in the adjoint group $\PGL_2$.) 
As we will see, the spaces $\mathcal S([V/T])^\circ$, $\mathcal D([V/T])^\circ$ descend from certain spaces $\mathcal S(V_\ad)^\circ$, $\mathcal D(V_\ad)^\circ$ of measures and half-densities on $V_\ad$, and the Hankel transform $\mathcal H^\circ$ descends (Proposition \ref{VHankelascoinv}) from a transform $\mathfrak F^\circ$ on them, which is a ``factor'' of Fourier transform.

\subsubsection{Mellin transforms on $V$}

Let us start by discussing Mellin transforms on $V$ with respect to the $\Gm$-action. The main result will be Proposition \ref{MellinVimage}, describing the Mellin image of $\mathcal S(V)$ as Paley--Wiener sections of a certain Fr\'echet bundle $\mathscr E_V$ over the character group $\widehat{A_\CC}$. 

In order to distinguish the $\Gm$-action on $V$ from the $\Gm$-action on $V\sslash T$ (where the former descends to the square of the latter), we will now denote the group of scalar multiplication on $V$ by $A$, its defining character by $e^{\frac{\alpha}{2}}$, and the absolute value of that by $\delta^\frac{1}{2}$. (This notation is adjusted to thinking of $A$ as the universal Cartan of $\SL(V)$.) We can define Mellin transform $f\mapsto \check f$ on $\mathcal S(V)$, where, for every $\chi\in \widehat{A}_\CC$, $\check f(\chi)$ is valued in the space $\mathcal S(A\backslash V^*, \mathcal L_\chi)$ of measures on $A\backslash V^*$ valued in a line bundle $\mathcal L_\chi$, where $V^*$ is, as before, the complement of zero on $V$, and $\mathcal L_\chi$ is the line bundle whose sections are $(A,\chi)$-equivariant functions on $V^*$ under the normalized action. It is given, for $|\chi(a)|=|a|^\sigma$ with $\sigma \ll 0$, by the formula
$$ \check f(v) = \int_{A} a\cdot f(v) \chi^{-1}(a),$$
where $a\cdot f$ denotes the normalized action \eqref{normVmodT}. Notice that we do not need a measure on $A$ --- the result is naturally a measure valued in the aforementioned line bundle over $A\backslash V^*$. The Mellin transform extends rationally/meromorphically (in the non-Archimedean, resp.\ Archimedean case) to all characters of $A$; we will give a complete description of its image.

We let $\mathscr D_V$ denote the bundle over $\widehat A_\CC$ whose fiber over the character $\chi$ is $\mathcal S(A\backslash V^*, \mathcal L_\chi)$. In the non-Archimedean case, it is naturally a direct limit of finite-dimensional algebraic vector bundles. In the Archimedean case, choosing a smooth section of the map $V^*\to \Rplus\backslash V^*$ (where $\Rplus \subset A$ is embedded in the split subtorus of the restriction of scalars to $\RR$), and a smooth non-vanishing measure on $A\backslash V^*$, we can identify all spaces $\mathcal S(A\backslash V^*, \mathcal L_\chi)$ as subspaces of $C^\infty(\Rplus\backslash V^*)$, and this allows us to consider it as a holomorphic Fr\'echet bundle. For convenience, by abuse of language, we will often just say ``Fr\'echet bundle'' to refer to both cases. 

In fact, in the Archimedean case, it has more structure than that of a Fr\'echet bundle: it can be identified as a subbundle of a constant Fr\'echet bundle, hence it makes sense (in the Archimedean case) to talk about its ``Paley--Wiener sections'', extending Definition \ref{def:PaleyWiener} to sections of this bundle. That is, for a diagonalizable representation $\rho$ of the $L$-group of $A$, we define $\mathbb H^\PW_{L(\bullet, \rho, 0)} (\widehat{A_\CC}, \mathscr D_V)$ as in Definition \ref{def:PaleyWiener}, with the only difference that it does not consist of scalar-valued meromorphic functions on $\widehat{A_\CC}$, but of meromorphic functions valued in the Fr\'echet bundle $\mathscr D_V$. Note that this notion does not depend on the choice of a smooth section of the quotient map $V^*\to \Rplus\backslash V^*$, since any two such sections are ``polynomially equivalent''.

We have an isomorphism $\mathcal S(V) = \mathcal F(V)dv$, where $dv$ is a Haar measure. Consider the blow-up $\tilde V\to V$ at the zero section. By pullback, we get an embedding $\mathcal F(V) \hookrightarrow \mathcal F(\tilde V)$. 
The measure $dv$ on $V$ pulls back to a measure on $\tilde V$ which is not smooth, but vanishes to order $1$ close to the exceptional divisor, i.e., if $\epsilon$ is a local coordinate for the divisor, it is of the form $|\epsilon(v)| \mu(v)$, where $\mu$ is a smooth measure --- these can be thought of as smooth measures valued in a complex line bundle $\mathscr L_{\delta^\frac{1}{2}}$ over $\tilde V$, where the notation is to suggest the character by which $A = \Gm$ acts on the fibers of this line bundle over the exceptional divisor of $\tilde V$. Thus, the pullback of smooth measures from $V$ to the blow-up gives rise to an injective map:
$$ \mathcal S(V)\hookrightarrow \mathcal S(\tilde V, \mathscr L_{\delta^\frac{1}{2}}).$$
If we embed $A \hookrightarrow \Ga$ via its defining character, we have an isomorphism $\tilde V= V^*\times^A \Ga$, in terms of which the space $ \mathcal S(\tilde V, \mathscr L_{\delta^\frac{1}{2}})$ has the following description: 

Notice that, for a space of the form $X\times^G Y = (X\times Y)/G^\diag$, where $G$ acts freely on the product $X\times Y$ and the quotient map is surjective on $F$-points, for the \emph{unnormalized} action of $G$ on the various Schwartz spaces of measures, the convolution (push-forward) with respect to the map $X\times Y\to X\times^G Y$ gives rise to an isomorphism:
$$ \mathcal S(X\times^G Y) \simeq \mathcal S(X)\hat\otimes_{\mathcal S(G)} \mathcal S(Y).$$
Here, $\hat\otimes_{\mathcal S(G)}$ denotes the quotient of the completed tensor product $\hat\otimes$ by the closed subspace generated by the kernel of the tensor product over $\mathcal S(G)$.
However, for the \emph{normalized} action of $A$ on $\mathcal S(V^*)$, we are multiplying the unnormalized action by the character $\delta^{-\frac{1}{2}}$. This gives rise to a canonical isomorphism:

\begin{equation}\label{blowupasproduct}\mathcal S(\tilde V, \mathscr L_{\delta^\frac{1}{2}} ) = \mathcal S(V^*)\hat\otimes_{\mathcal S(A)} S(\Ga).
\end{equation}

Under Mellin transform, the tensor product $\hat\otimes_{\mathcal S(A)}$ translates to multiplication. This shows:

\begin{lemma}\label{MellinVblowup}
 Mellin transform gives rise to an isomorphism: 
$$\mathcal S(\tilde V, \mathscr L_{\delta^\frac{1}{2}} ) \xrightarrow\sim \mathbb H^\PW_{L(\bullet^{-1}, 1)} (\widehat A_\CC, \mathscr D_V).$$
\end{lemma}

\begin{proof}
 This follows from \eqref{blowupasproduct} and Proposition \ref{Tateimage}. (Recall that Mellin transform on $\mathcal S(\Ga)$, evaluated at a character $\chi$, corresponds to the zeta integral $Z(\bullet, \chi^{-1}, 1)$.)
\end{proof}

The image of $\mathcal S(V)$ in $\mathbb H^\PW_{L(\bullet^{-1}, 1)} (\widehat A_\CC,\mathscr D_V)$ is characterized as follows:  Note that the poles of $\chi\mapsto L(\chi^{-1}, 1)$ coincide with the characters $\chi$ such that the $\SL(V)$-representation $\mathscr D_{V,\chi}=\mathcal S(A\backslash V^*, \mathcal L_\chi)$ has a finite-dimensional submodule; indeed, there is an isomorphism of representations
$$ \mathcal S(A\backslash V^*, \mathcal L_\chi) \simeq  I(\chi^{-1}),$$
where $I(\chi^{-1}) = \Ind_B^{\SL(V)} (\chi^{-1} \delta^\frac{1}{2})$, a principal series representation, where $\chi$ is understood as a character of the universal Cartan of $\SL(V)$ via the identification $\Gm \xrightarrow\sim A$ defined by the positive coroot. In the non-Archimedean case, $L(\chi^{-1}, 1)$ has a pole only at $\chi= \delta^\frac{1}{2}$, which is precisely the point where the trivial representation is a subrepresentation of $I(\chi^{-1})$. In the Archimedean case, compare with \cite[Theorems 5.11 and 6.2]{JL}.

Let $\mathscr D_V^\fin$ be the subsheaf of $\mathscr D_V$ whose sections belong to the finite-dimensional subrepresentation of $\mathcal S(A\backslash V^*, \mathcal L_\chi)$, at the poles of $L(\chi^{-1}, 1) = L(\chi, -\check\alpha, 1)$. Let $\mathscr E_V = \mathscr D_V^\fin(-[L(\bullet^{-1}, 1)])$, that is, the bundle whose sections are rational/meromorphic sections of $\mathscr D_V^\fin$, with poles bounded by the poles of the function $\chi\mapsto L(\chi^{-1}, 1)$. It is the subsheaf of $\mathscr D_V(-[L(\bullet^{-1}, 1)])$, where the Paley--Wiener sections of Lemma \ref{MellinVblowup} are valued, determined by the condition that the residues are valued in the finite-dimensional subrepresentation, at all poles of $L(\bullet^{-1},1)$.

We have the following description of the image of Mellin transform:

\begin{proposition}\label{MellinVimage}
Mellin transform defines an isomorphism 
$$ \mathcal S(V) \xrightarrow\sim \mathbb H^\PW(\widehat A_\CC,\mathscr E_V).$$ 
The sheaf $\mathscr E_V$ is generated by its global Paley--Wiener sections.
\end{proposition}

I remind that in the non-Archimedean case, the notation $\mathbb H^\PW$ for Paley--Wiener sections refers, simply, to polynomial sections, and that the sheaf being generated by a space $\Gamma$ of sections means that elements of $\Gamma$ generate the sections of the sheaf in any small neighborhood of any point, by multiplying them by the sections of the appropriate structure sheaf (=polynomial functions on $\widehat{A}_\CC$ in the non-Archimedean case, holomorphic functions in the Archimedean case).

\begin{proof}
Indeed, by the above, the Schwartz space $\mathcal S(\tilde V, \mathscr L_{\delta^\frac{1}{2}} )$ of the blow-up is identified with 
\[\mathbb H^\PW_{L(\bullet^{-1}, 1)} \left(\widehat A_\CC, \mathscr D_V\right) = H^\PW \left(\widehat A_\CC,\mathscr D_V(-[L(\bullet^{-1}, 1)])\right).\] The condition that the residue of a section at the poles of $L(\chi^{-1}, 1)$ lie in the finite-dimensional submodule is precisely the condition that the element of $\mathcal F(\tilde V) dv = \mathcal S(\tilde V, \mathscr L_{\delta^\frac{1}{2}} )$ descend to $\mathcal F(V) dv = \mathcal S(V)$. At every pole of $L(\bullet^{-1}, 1)$, then, the residue of the Mellin transform (considered as a meromorphic section of $\mathscr D_V$) is determined by derivatives of the delta function at zero, which give rise to a surjection from $\mathcal S(V)$ to the finite-dimensional subrepresentation in the fiber of $\mathscr D_V$. This shows that Paley--Wiener sections generate the sheaf $\mathscr E_V$ around those points; away from poles of $L(\bullet^{-1}, 1)$ this is obvious, already by considering Mellin transforms of elements of $\mathcal S(V^*)$.
\end{proof}

Locally around the poles of $L(\chi^{-1}, 1)$, those are precisely the sections that are obtained as the image of $\mathfrak R_\chi^{-1}$, the inverse of the standard intertwining operator (see \S \ref{scatWhittaker}). Thus, the bundle $\mathscr E_V$ can be identified with the bundle $\chi\mapsto \mathcal S(A\backslash V^*, \mathcal L_{\chi^{-1}})$ (notice the \emph{inverse} character!) around the poles of the $L$-function (but not globally over $\widehat A_\CC$!).

If $V$ is endowed with a symplectic form $\omega$, we can divide measures by the Haar half-density $|\omega|^\frac{1}{2}$ to replace $\mathcal S(V)$ in Proposition \ref{MellinVimage} by the space of Schwartz half-densities $\mathcal D(V)$. 
Notice that, by our normalization of the $A$-action, this map $\mathcal S(V)\to \mathcal D(V)$ is equivariant, and also recall that the argument $\chi$ of Mellin transform refers to the normalized action of $A$. The symplectic Fourier transform
$$ \mathfrak F: \mathcal D(V)\xrightarrow\sim \mathcal D(V)$$
is $A$-anti-equivariant, hence induces isomorphisms
\begin{equation}\label{Fchi} \mathfrak F_\chi: \mathscr E_{V,\chi^{-1}} \xrightarrow\sim \mathscr E_{V,\chi}.
\end{equation}
(Compare with \S \ref{scatWhittaker}.)

\subsubsection{Mellin transforms on $V_\ad$}

We now carry over this discussion to the stack $V_\ad = [V/\{\pm 1\}]$, which is a stack with an action of $\PGL(V)$. Its (isomorphism classes of) $F$-points can be identified, as in \S \ref{scattorus}, with the disjoint union 
$$ \bigoplus_\alpha V^\alpha(F)/\{\pm 1\},$$
where $\alpha$ runs over all quadratic extensions of $F$, including the split one $F\oplus F$, and $V^\alpha$ is the tensor product of $V$ with the imaginary line of that extension. Correspondingly, the Schwartz space $\mathcal S(V_\ad)$ is the direct sum:
$$ \mathcal S(V_\ad) = \bigoplus_\alpha \mathcal S(V^\alpha(F))_{\{\pm 1\}},$$
where the index $\{\pm 1\}$ denotes coinvariants, but of course we can identify those with invariants.

This has an action of $A_\ad=\Gm/\{\pm 1\}$, which also acts faithfully on the quotient $V\sslash T$. Thus, we will define Mellin transform $f\mapsto \check f$ on $\mathcal S(V_\ad)$, with the image lying in the space of sections over $\widehat{A_\ad}_\CC$ of the bundle $\chi\mapsto  \mathcal S(A_\ad\backslash V_\ad^*, \mathcal L_\chi)$. This is also a Fr\'echet bundle, which will be denoted by $\mathscr D_{V_\ad}$.
If $s: A\to A_\ad$ is the quotient map (for ``square'', when we identify these groups with $\Gm$), it induces a pullback map of characters $s^* :\widehat{A_\ad}_\CC \to \widehat A_\CC$, and there is a canonical isomorphism 
$$ \mathcal S(A_\ad\backslash V_\ad^*, \mathcal L_\chi) \simeq \mathcal S(A\backslash V^*, \mathcal L_{s^*\chi}),$$
induced by the isomorphism of stacks $A\backslash V = A_\ad \backslash V_\ad$, so $\mathscr D_{V_\ad}$, as a Fr\'echet bundle, is really the pullback: 
$$\mathscr D_{V_\ad} = s^{**}\mathscr D_V,$$
but endowed with an action of $\PGL(V)$, so that the fiber over $\chi\in \widehat{A_\ad}_\CC$ is isomorphic the principal series representation obtained by normalized induction from $\chi^{-1}$.

Let $\tilde V_\ad$ be the quotient of the blow-up $\tilde V$ by $\pm 1$; there is a pullback map
$$ \mathcal S(V_\ad) \hookrightarrow \mathcal S(\tilde V_\ad, \mathscr L_{\delta^\frac{1}{2}}),$$
where $\mathscr L_{\delta^\frac{1}{2}}$ denotes the pullback of the line bundle denoted by the same notation on $\tilde V$.
We have isomorphisms
$$ \tilde V_\ad = V^*_\ad \times^A \Ga =  V^* \times^A [\Ga/\{\pm 1\}] = V_\ad^* \times^{A_\ad} [\Ga/\{\pm 1\}],$$
which give rise to an isomorphism of Schwartz spaces:
\begin{equation}\label{blowupasproduct-ad} \mathcal S(\tilde V_\ad, \mathscr L_{\delta^\frac{1}{2}}) = \mathcal S(V_\ad^*) \hat\otimes_{\mathcal S(A_\ad)} \mathcal S([\Ga/\{\pm 1\}]),
\end{equation}
as in \eqref{blowupasproduct}. 

The $F$-points and Schwartz space of $[\Ga/\pm 1]$ are described in a completely analogous way as the $F$-points and Schwartz space of $V_\ad$, and the space $\mathcal S([\Ga/\pm 1])$ can be identified, through the push-forward under the map $\Ga\ni x\to \xi= x^2 \in \Ga \sslash \{\pm 1\}$, with the space of measures on the line which away from zero coincide with Schwartz measures, while in a neighborhood of zero are of the form
\begin{equation}\label{SchwartzGamodpm1} |\xi|^{-\frac{1}{2}} \sum_{\omega} C_\omega(\xi) \omega(\xi),\end{equation}
where $\omega$ runs over all quadratic characters of $F^\times$, and $C_\omega$ is a smooth measure. The analog of this statement for zeta integrals is the following: if we identify both $A$ and $A_\ad$ with $\Gm$, through, respectively, the cocharacters $\check\alpha$ and $\frac{\check\alpha}{2}$, the $L$-function $L(s^*\chi^{-1}, 1) = L(\chi,-\check\alpha, 1)$ has a pole at $\chi\in \widehat{A_\ad}_\CC$ if and only if the $L$-function 
$L(\omega\chi^{-1}, \frac{1}{2}) = L(\omega\chi, -\frac{\check\alpha}{2}, \frac{1}{2})$, has a pole for some quadratic character $\omega$. 

From \eqref{SchwartzGamodpm1} it follows that Mellin transform on $A_\ad=\Gm\subset \Ga\sslash \{\pm 1\}\simeq\Ga$ defines an isomorphism: 
$$\mathcal S([\Ga/\pm 1]) \xrightarrow\sim \sum_\omega \mathbb H^\PW_{L(\bullet^{-1} \cdot \omega, \frac{1}{2})} (\widehat{A_\ad}_\CC),$$ 
where the right hand side is understood as a subspace of the space of meromorphic functions on $\widehat{A_\ad}_\CC$, endowed with the quotient topology from the direct sum of the summands. 
Again, we do not introduce any normalization to the action of $A_\ad$ on $\Ga\sslash \{\pm 1\}$. Since the poles of the functions $L(\bullet^{-1} \cdot \omega, \frac{1}{2})$ are distinct for different $\omega$'s, the right hand side can also be written\footnote{When $F=\RR$, this requires an explanation: Suppose that $\Phi$ is a Paley--Wiener holomorphic multiple of $\prod_\omega L(\bullet^{-1} \cdot \omega, \frac{1}{2})$. There are two $\omega$'s here, the trivial and the sign character. We need to show that $\Phi = \Phi_1 + \Phi_{\sgn}$, where each summand is a Paley--Wiener holomorphic multiple of the corresponding $L(\bullet^{-1} \cdot \omega, \frac{1}{2})$. We will construct $\Phi_1$ as the Mellin transform of an element $f\in \mathcal S(\RR)$; notice that, in that case, the residues of $\Phi_1$ at the poles of $L(\bullet^{-1} , \frac{1}{2})$ are determined by the derivatives of arbitrary order of $f$ at $0$. We can prescribe those derivatives, so that the residues equal the residues of $\Phi$ at those points. Then, $\Phi_1 = \check f $ and $\Phi_{\sgn}= \Phi -\Phi_1$ satisfy our requirements.} 
$$\mathcal S([\Ga/\pm 1]) \xrightarrow\sim \mathbb H^\PW_{\prod_\omega L(\bullet^{-1} \cdot \omega, \frac{1}{2})} (\widehat{A_\ad}_\CC).$$ 

For $\tilde V_\ad$, this means, by \eqref{blowupasproduct-ad}: 
\begin{equation}\mathcal S(\tilde V_\ad, \mathscr L_{\delta^\frac{1}{2}}) \xrightarrow\sim   \mathbb H^\PW_{\prod_\omega L(\bullet^{-1} \cdot \omega, \frac{1}{2})} \left(\widehat{A_\ad}_\CC,\mathscr D_{V_\ad}\right).
\end{equation}
Moreover, as in the case of $\SL(V)$ (see the discussion before Proposition \ref{MellinVimage}), the poles of the sections on the right hand side coincide with the poles of $L(s^*\chi^{-1}, 1)$ which, as we saw above, correspond to the points where the space $\mathcal S(A_\ad\backslash V_\ad^*, \mathcal L_{\chi})\simeq \mathcal S(A\backslash V^*, \mathcal L_{s^*\chi})$ contains a finite-dimensional representation of $\PGL(V)$. 

Thus, we have analogous bundles over $\widehat{A_\ad}_\CC$:
$$\mathscr D_{V_\ad}^\fin = s^{**} \mathscr D_V^\fin$$ 
(characterized by the fact that sections are valued in the finite-dimensional subrepresentation at poles of $L(s^*\chi, 1)$, i.e., poles of some $L(\omega\chi, \frac{1}{2})$, with $\omega$ quadratic), and 
$$\mathscr E_{V_\ad}:=s^{**} \mathscr E_V = \mathscr D_{V_\ad}^\fin(-[\prod_\omega  L(\omega\chi, \frac{1}{2})]).$$

As a corollary of Proposition \ref{MellinVimage} we have:
\begin{corollary}\label{corMellinVad} 
Mellin transform defines an isomorphism 
\begin{equation}\label{MellinVad}\mathcal S(V_\ad) \xrightarrow\sim \mathbb H^\PW (\widehat{A_\ad}_\CC, \mathscr E_{V_\ad}).
\end{equation}
The sheaf $\mathscr E_{V_\ad}$ is generated by its global Paley--Wiener sections.
\end{corollary}

Let now $\mathscr E^\circ_{V_\ad}$ denote the bundle 
\begin{equation}\mathscr E^\circ_{V_\ad} :=   \mathscr E_{V_\ad} ([L(\bullet^{-1}, \frac{1}{2})]) =\mathscr D_{V_\ad}^\fin(-[\prod_{\omega\ne 1} L(\bullet^{-1} \cdot \omega, \frac{1}{2})]).
\end{equation}
That is, the sections of $\mathscr E^\circ_{V_\ad}$ are meromorphic sections of $\mathscr D_{V_\ad}^\fin$, as for $\mathscr E_{V_\ad}$, \emph{but without poles at the poles of $L(\bullet^{-1}, \frac{1}{2})$}.

We define a subspace $\mathcal S(V_\ad)^\circ$, consisting of those elements of $\mathcal S(V_\ad)$ whose Mellin transform \eqref{MellinVad} is valued in the subbundle $\mathscr E^\circ_{V_\ad}$: 
\begin{equation}\label{defVadcirc}
 \mathcal S(V_\ad)^\circ \simeq  \mathbb H^\PW (\widehat{A_\ad}_\CC, \mathscr E_{V_\ad}^\circ).
\end{equation}

Embed $A_\ad\hookrightarrow \Ga$ via the character $e^{\alpha}$, which pulls back to the square of the defining character of $A=\Gm\subset \GL(V)$. Then there is a well-defined push-forward map:
\begin{equation}\label{Vproduct-measures} \mathcal S(V_\ad)^\circ \hat\otimes_{\mathcal S(A_\ad)} \delta^{\frac{1}{2}}\mathcal S(\Ga) \to \mathcal S(V_\ad),\end{equation}
as can be seen by considering Mellin transforms. It is easily seen in the non-Archimedean case to be an isomorphism, and I conjecture that this is the case in the Archimedean case as well (but we will not need this).  The following will be enough for our purposes:

\begin{lemma}\label{globalsections-V}
 The sheaf $\mathscr E^\circ_{V_\ad}$ is generated by its global Paley--Wiener sections. In the non-Archimedean case, the action of the element $h\in \widehat{S(A_\ad)}$ whose Mellin transform is $L(\chi, \frac{\check\alpha}{2}, \frac{1}{2})^{-1}$ gives rise to an isomorphism:
 $$ \mathcal S(V_\ad)\xrightarrow{h\cdot} \mathcal S(V_\ad)^\circ.$$ 
\end{lemma}

\begin{proof}
The normalized action of an element $h$ of the completed Schwartz algebra $\widehat{S(A_\ad)}$ of $A_\ad$ translates on Mellin transforms to multiplication by $\check h(\chi^{-1})$.

In the non-Archimedean case, we can take $h$ to be the element of the completed Hecke algebra whose Mellin transform is $L(\chi, \frac{\check\alpha}{2}, \frac{1}{2})^{-1}$, and this defines a bijection between ``Paley--Wiener'' (polynomial, supported on finitely many components) sections of the sheaves $\mathscr E_{V_\ad}$ and $\mathscr E_{V_\ad}^\circ$. The result now follows from the analogous statement for $\mathscr E_{V_\ad}$ (Corollary \ref{corMellinVad}). 

In the Archimedean case, to apply the analogous statement for $\mathscr E_{V_\ad}$, we need to find a multiplier $h\in \widehat{S(A_\ad)}$, whose Mellin transform $\check h$ has zeroes at the poles of the $L$-function $L(\chi,\frac{\check\alpha}{2},\frac{1}{2})$, simple ones at any chosen point. Identifying $A_\ad$ with $\Gm$ as before, one can take a measure $h(x) = H(x) |x|^\frac{1}{2} d^\times x \in \mathcal S(F^\times)$, where $H$ is a Schwartz function on $F^\times$ whose Fourier transform (considered as a function on $F$) is also a Schwartz function on $F^\times$ (i.e., all its derivatives at zero vanish). Then, the functional equation \eqref{gammaZeta} of Tate zeta integrals implies that the Mellin transform of $h$ vanishes at all poles of $L(\chi,\frac{1}{2})$, and it can easily be arranged that the vanishing is simple at any chosen point. In that case, the normalized action of $h$ on $\mathcal S(V_\ad)$,  sends Paley--Wiener sections of the sheaf $\mathscr E_{V_\ad}$ to Paley--Wiener sections of the sheaf $\mathscr E_{V_\ad}^\circ$; and, at any pole of $L(\chi^{-1},\frac{1}{2})$ where $\check h(\chi^{-1})$ has a simple zero, it will identify the fibers. 
\end{proof}

\begin{remark}
 The space $\mathcal S(V_\ad)^\circ$ is \emph{not} the space of rapidly decaying sections of a sheaf over the affine quotient $V\sslash\{\pm 1\}$, since the condition that Mellin transforms lie in the subbundle $\mathscr E^\circ_{V_\ad}$ depends on integrals of the functions over $A_\ad$-orbits, i.e., it is not local. As we will see, however, its image in $\mathcal S([V/T])$ will be the space $\mathcal S([V/T])^\circ$ that we defined before, which is local over $V\sslash T$. 
\end{remark}

We define the space $\mathcal D(V_\ad)$ precisely as we defined the space $\mathcal S(V_\ad)$, by using half-densities instead of measures on the ``pure inner forms'' $V^\alpha(F)$. The isomorphism \eqref{MellinVad} carries over to half-densities, and we have an analogous subspace $\mathcal D(V_\ad)^\circ$. The analog of \eqref{Vproduct-measures} is now a convolution map
\begin{equation}\label{Vproduct-densities} P_V: \mathcal D(V_\ad)^\circ \hat\otimes_{\mathcal S(A_\ad)} \mathcal D(\Ga) \to \mathcal D(V_\ad),\end{equation}
depending on the choice of a Haar measure on $A_\ad$.

The symplectic Fourier transform $\mathfrak F$ on $\mathcal D(V)$ descends canonically to an involution of $\mathcal D(V_\ad)$, as was explained in \S \ref{scattorus}; this also follows from the isomorphisms \eqref{Fchi} between the fibers of the bundle $\mathscr E_V$, and the fact that $\mathscr E_{V_\ad}$ is the pullback of $\mathscr E_V$ to $\widehat{A_\ad}_\CC$.
It admits the following factorization:

\begin{proposition}\label{factorFourier}
 There is an $A_\ad$-anti-equivariant endomorphism $\mathfrak F^\circ$ of $\mathcal D(V_\ad)^\circ$, such that, if $\mathfrak F_1: \mathcal D(\Ga)\xrightarrow\sim \mathcal D(\Ga)$ denotes Fourier transform on the line defined by the character $\psi$, we have, in terms of the map of \eqref{Vproduct-densities}:
\begin{equation}
 \mathfrak F\circ P_V(\varphi_V^\circ\otimes \varphi_1) = P_V(\mathfrak F^\circ \varphi_V^\circ \otimes \mathfrak F_1 \varphi_1).
\end{equation}
In terms of Mellin transforms, under the isomorphism 
$$\mathcal D(V_\ad)^\circ \xrightarrow\sim \mathbb H^\PW (\widehat{A_\ad}_\CC, \mathscr E^\circ_{V_\ad})$$
the transform $\mathfrak F^\circ$ is induced by the isomorphism
$$\gamma(\chi, \frac{1}{2},\psi)^{-1}\mathfrak F_\chi: \mathscr E^\circ_{V_\ad,\chi^{-1}} \xrightarrow\sim \mathscr E^\circ_{V_\ad,\chi}$$
(see \eqref{Fchi}).
\end{proposition}

\begin{proof}
 Indeed, $\gamma(\chi, \frac{1}{2},\psi)^{-1}\mathfrak F_\chi$ is an isomorphism between the fibers of $\mathscr E^\circ_{V_\ad}$, as stated, because $\gamma(\chi, \frac{1}{2},\psi)$ has a simple pole where $L(\chi^{-1},\frac{1}{2})$ does, and a simple zero where $L(\chi, \frac{1}{2})$ has a pole. It is holomorphic and non-zero everywhere else, and, in the Archimedean case, of polynomial growth in vertical strips (by Stirling's formula), thus preserves the spaces of Paley--Wiener sections. Hence, by the definition \eqref{defVadcirc} of $\mathcal S(V_\ad)^\circ$, the family of isomorphisms $\gamma(\chi, \frac{1}{2},\psi)^{-1}\mathfrak F_\chi$ induces an endomorphism $\mathfrak F^\circ$ of this space, which is $A_\ad$-anti-equivariant.
 
 On the other hand, for $\varphi_1 \in \mathcal D(\Ga)$ we have
 $$ \widecheck{\mathfrak F_1\varphi_1}(\chi) = \gamma(\chi,\frac{1}{2},\psi) \widecheck{\varphi_1}(\chi^{-1}),$$
 by the functional equation \eqref{gammaZeta} of Tate's thesis,
 so if $\varphi = P_V(\varphi_V^\circ\otimes \varphi_1)$ and $\varphi' = P_V(\mathfrak F^\circ \varphi_V^\circ\otimes \mathfrak F_1\varphi_1)$, we have
 \[\check\varphi'(\chi) =   \gamma(\chi,\frac{1}{2},\psi)^{-1}\mathfrak F_\chi \widecheck{\varphi_V^\circ}(\chi^{-1}) \cdot \gamma(\chi,\frac{1}{2},\psi) \widecheck{\varphi_1}(\chi^{-1}) = \mathfrak F_\chi \left(\widecheck{\varphi_V^\circ}(\chi^{-1}) \cdot  \widecheck{\varphi_1}(\chi^{-1})\right)\]
 (notice that $\widecheck{\varphi_1}(\chi^{-1})$ is a scalar),
 in other words $\check\varphi'(\chi) = \mathfrak F_\chi'\check\varphi(\chi^{-1})$, thus $\varphi' = \mathfrak F\varphi$.
\end{proof}

Finally, we descend to coinvariants: 

\begin{proposition}\label{VHankelascoinv}
The image of $\mathcal S(V_\ad)^\circ$ under the push-forward map to measures on $V\sslash T=V_\ad\sslash T_\ad$ is dense in the space $\mathcal S([V/T])^\circ$, and we have a commutative diagram
\begin{equation}
 \xymatrix{
 \mathcal D(V_\ad)^\circ \ar[r]^{\mathfrak F^\circ} \ar[d] & \mathcal D(V_\ad)^\circ \ar[d] \\
 \mathcal D([V/T])^\circ \ar[r]^{\mathcal H^\circ} & \mathcal D([V/T])^\circ,
 }
\end{equation}
where $\mathcal H^\circ$ is the Hankel transform of Corollary \ref{Hcirc}.

\end{proposition}

\begin{proof}
First, let us compare Mellin transforms of elements of $\mathcal S(V_\ad)$ and of $\mathcal S([V/T])$: the former are, by Corollary \ref{corMellinVad}, Paley--Wiener sections of the sheaf $\mathscr E_{V_\ad}$; the latter were defined in \eqref{MellinVTdef} as scalar valued functions on $\widehat{F^\times_\CC}=\widehat{A_\ad}_\CC$, after fixing an identification $V\sslash T\simeq \mathbbm A^1$, but we can also think of them as valued in a bundle $\CC_\bullet: \chi\mapsto \CC_\chi$ over $\widehat{A_\ad}_\CC$, where $\CC_\chi$ denotes $(A_\ad,\chi)$-equivariant functions on $V\sslash T\smallsetminus\{0\}$ (for the normalized action descending from \eqref{normVmodT}). 

The rationality/meromorphicity of Mellin transforms of elements of $\mathcal S([V/T])$, together with the generation of $\mathscr E_{V_\ad}$ by its global sections (Corollary \ref{corMellinVad}) imply that 
there is a rational/meromorphic map of sheaves over $\widehat{A_\ad}_\CC$:
$$\mathscr E_{V_\ad} \to \CC_\bullet,$$
defined by a convergent integral when $\chi$ vanishes sufficiently fast at infinity. This relates the Mellin transforms of elements of $\mathcal S(V_\ad)$ with those of their push-forwards to $\mathcal S([V/T])$. The description of the image of Mellin transform on $\mathcal S([V/T])$ in Proposition \ref{VmodTimage} implies that the poles of this map are precisely the poles of $L(\bullet^{-1}, \frac{1}{2}) L(\bullet^{-1} \eta, \frac{1}{2})$, so we get a surjective, holomorphic map of sheaves:
\begin{equation}\label{betweenMellin} \mathscr E_{V_\ad} \to \CC_\bullet(-[L(\bullet^{-1}, \frac{1}{2}) L(\bullet^{-1} \eta, \frac{1}{2})]).\end{equation}

In particular, we get a map of subsheaves 
\begin{equation}\label{betweenMellin-circ}\mathscr E^\circ_{V_\ad} :=   \mathscr E_{V_\ad} ([L(\bullet^{-1}, \frac{1}{2})]) \to \CC_\bullet(-[L(\bullet^{-1} \eta, \frac{1}{2})]),\end{equation}
hence the Mellin transforms of push-forwards of elements of $\mathcal S(V_\ad)^\circ$ are Paley--Wiener sections of the subsheaf $\CC_\bullet(-[L(\bullet^{-1} \eta, \frac{1}{2})])$, hence, by Corollary \ref{MellinVcirc}, the image of $\mathcal S(V_\ad)^\circ$ belongs to $\mathcal S([V/T])^\circ$.

In the non-Archimedean case, the element $h$ of the completed Hecke algebra $\widehat{\mathcal S(A_\ad)}$ whose Mellin transform is $\check h(\chi) = L(\chi,\frac{1}{2})$ acts on $\mathcal S(V_\ad)$, resp.\ $\mathcal S([V/T])$, and maps it onto $\mathcal S(V_\ad)^\circ$, resp.\ $\mathcal S([V/T])^\circ$ (or, the corresponding spaces of half-densities); the action of $A_\ad$ commutes with all arrows in the diagram, and the result follows easily from the commutative diagram \eqref{FourierHankelV}. 

For the Archimedean case, where no such multiplier exists, we will need a different argument. The difficulty lies in showing that the map $\mathcal S(V_\ad)^\circ\to \mathcal S([V/T])^\circ$ has dense image; the commutativity of the diagram then follows as in the non-Archimedean case. The argument is quite technical, and the reader might choose to skip it at first reading.

The case of non-split $T$ is easiest: 
Let us identify again $V$ with the additive group of a quadratic extension $E$, and $T$ with the kernel of the norm map. (The precise coordinates do not matter here.)
In the non-split case, $T$ is a compact subgroup of $E^\times$; thus, we can identify invariants and coinvariants, i.e., the map 
 $$ \mathcal S(E)^T\oplus \mathcal S(E^\alpha)^T \to \mathcal S([V/T])$$
 is an isomorphism, where $E^\alpha$ is the ``pure inner form'' attached to the quadratic extension splitting $T$. We can thus lift any element $f^\circ\in \mathcal S([V/T])^\circ$ to $\mathcal S(E)^T\oplus \mathcal S(E^\alpha)^T \subset \bigoplus_\beta \mathcal S(E^\beta)$, where $\beta$ runs over all quadratic extensions. The image of the lift under $\mathcal S(E)\oplus \mathcal S(E^\alpha)\to \mathcal S(V_\ad)$ will have Mellin transform with no poles at the poles of $L(\chi^{-1},\frac{1}{2})$, and with values in $\mathscr E^\circ_{V_\ad}$, since the trivial $T$-type belongs to the finite-dimensional subrepresentation of the principal series $I(\chi^{-1})$.

 For the split case, we will show that the push-forward map $\mathcal S(V_\ad)^\circ \to \mathcal S([V/T])^\circ$, interpreted in terms of Mellin transforms as a map of Paley--Wiener sections: 
 $$\mathbb H^\PW \left(\widehat{A_\ad}_\CC, \mathscr E_{V_\ad}^\circ \right) \to \mathbb H^\PW_{L(\bullet^{-1}\eta, \frac{1}{2})} \left(\widehat{A_\ad}_\CC, \CC_\bullet\right)$$
 has dense image. (Here, $\eta=1$, but I will maintain it in the notation in order to make the comparison with previous statements easier.) Knowing already that the corresponding map \eqref{betweenMellin} is surjective on Paley--Wiener sections (by Proposition \ref{VmodTimage}), the difficulty lies in showing that by imposing the extra condition of vanishing at the poles of $L(\bullet^{-1}, \frac{1}{2})$ that defines $\mathscr E_{V_\ad}^\circ$, we still have enough Paley--Wiener sections to generate a dense subspace of Paley--Wiener sections of the image sheaf $\CC_\bullet(-[L(\bullet^{-1} \eta, \frac{1}{2})])$.
 
 It would be desirable to have a simple argument of  complex analysis for this; unfortunately, I do not know such an argument. The difficulty lies in the fact that $L(\bullet^{-1}, \frac{1}{2})^{-1}$, a reciprocal Gamma function, is not of polynomial growth in vertical strips, hence cannot be used as a multiplier.
 
 Therefore, I will use some representation theory. Notice that the map $V^*\to V\sslash T$ is smooth and surjective on $F$-points (when $T$ is split), therefore the push-forward map $\mathcal S(V^*)\to \mathcal S(V/T)$ has image equal to the space $\mathcal S(V\sslash T)$ of Schwartz measures on the affine line, whose Mellin transforms are precisely the Paley--Wiener sections $\mathbb H^\PW_{L(\bullet^{-1}\eta, \frac{1}{2})} \left(\widehat{A_\ad}_\CC, \CC_\bullet\right)$, by Proposition \ref{Tateimage}. (Recall that we are here parametrizing Mellin transforms on $V\sslash T$ according to the normalized action of $A_\ad$, which explains why the poles appear at poles of $L(\bullet^{-1}\eta, \frac{1}{2})$.) The same will, obviously, hold if we replace $V^*$ by $V_\ad^*$. Thus, we have two maps into the same space:
 
 \begin{equation}\label{xyPW} \xymatrix{
 \mathcal S(V_\ad^*) \simeq H^\PW \left(\widehat{A_\ad}_\CC, \mathscr D_{V_\ad} \right) \ar[drr] \\
 && \mathbb H^\PW_{L(\bullet^{-1}\eta, \frac{1}{2})} \left(\widehat{A_\ad}_\CC, \CC_\bullet\right).\\
 \mathcal S(V_\ad)^\circ \simeq \mathbb H^\PW \left(\widehat{A_\ad}_\CC, \mathscr E_{V_\ad}^\circ \right) \ar[urr]
 }\end{equation}

 The upper arrow is surjective, and we want to show that the lower arrow has dense image. The spaces $\mathcal S(V_\ad^*)$ and $\mathcal S(V_\ad)^\circ$ are not directly comparable, as the study of their Mellin transforms shows: the latter have Mellin transforms in $\mathscr E_{V_\ad}^\circ$, which are \emph{meromorphic} sections of $\mathscr D_{V_\ad}$ with poles at the poles of $\prod_{\omega\ne 1} L(\bullet^{-1} \cdot \omega, \frac{1}{2})$; but also with the condition that at any pole $\chi$ of $L(\bullet^{-1}\omega , \frac{1}{2})$, \emph{including $\omega=1$}, they are valued in the finite-dimensional subrepresentation of the fiber of $\mathscr E_{V_\ad}^\circ$; this condition at poles of $L(\bullet^{-1}, \frac{1}{2})$ shows that $\mathscr D_{V_\ad}$ is not a subsheaf of $\mathscr E_{V_\ad}^\circ$. Thus, $\mathcal S(V_\ad^*)$ is not a subspace of $\mathcal S(V_\ad)^\circ$.
 
 But both spaces are smooth Fr\'echet representations of $\PGL(V)$, and, fixing a good maximal compact subgroup $K$, their spaces of $K$-finite vectors are dense. For every $K$-type (=irreducible representation of $K$) $\tau$, we will denote by an exponent $(K,\tau)$ the subspaces of vectors belonging to this $K$-type.
 
 We have the following observation:
 
 \begin{lemma}\label{lemma-finite}
  For every $K$-type $\tau$, the quotient 
  $$\mathcal S(V_\ad^*)^{(K,\tau)} / \mathcal S(V_\ad^*)^{(K,\tau)}  \cap \mathcal S(V_\ad)^{\circ, (K,\tau)}$$
  is finite-dimensional, and supported on a finite number of poles of $L(\bullet^{-1}, \frac{1}{2})$, as an $A_\ad$-module.
 \end{lemma}
 
 \begin{proof}[Proof of the lemma]
  The essential observation, here, is that at poles of $L(\bullet^{-1}, \frac{1}{2})$ the given $K$-type belongs to the finite-dimensional subrepresentation of $\mathscr D_{V_\ad, \chi}$, for all but a finite number of $\chi$'s. The polynomial $\check D_\tau$ on $\widehat{A_\ad}$ which has simple zeroes at precisely those $\chi$'s (it can be thought of as the Mellin transform of a multiplier $D_\tau\in \widehat{\mathcal S(A_\ad)}$) gives rise to a map:
  $$D_\tau: \mathcal S(V_\ad^*)^{(K,\tau)} \to \mathcal S(V_\ad^*)^{(K,\tau)}  \cap \mathcal S(V_\ad)^{\circ, (K,\tau)}$$
  whose cokernel, as an $A_\ad$-module, is supported only on this finite set of $\chi$'s.
 \end{proof}
 
 Let $J$ denote the space $\mathbb H^\PW_{L(\bullet^{-1}\eta, \frac{1}{2})} \left(\widehat{A_\ad}_\CC, \CC_\bullet\right)$,  $J^\circ$ the closure of the image of $\mathcal S(V_\ad)^\circ$ under \eqref{xyPW}, and $J'=J/J^\circ$ --- it is a smooth Fr\'echet representation of $A_\ad$ of moderate growth, or \emph{$SF$-representation}, in the language of \cite{BeKr}. Let $I_\tau^\circ\subset I_\tau$ denote, respectively, the images of $\mathcal S(V_\ad^*)^{(K,\tau)}  \cap \mathcal S(V_\ad)^{\circ, (K,\tau)}$ and $\mathcal S(V_\ad^*)^{(K,\tau)}$ in $J$. By the fact that $I_\tau^\circ \subset J^\circ$, and Lemma \ref{lemma-finite}, we deduce that $J'$ is has a (countable) dense subspace of $A_\ad$-finite vectors, with eigencharacters among the poles of $L(\bullet^{-1}, \frac{1}{2})$. If we enumerate these poles $\chi_1, \chi_2, \dots$, and let $J'_n$ be the $\chi_n$-eigenspace, we claim that any continuous seminorm on $J'$ is zero on all but a finite number of $J'_n$'s. Indeed, any continuous seminorm is bounded by some seminorm that defines (after completion) a Banach representation of $A_\ad$, in particular has the property that the action of any $a\in A_\ad$ is bounded; but if we take $a\in A_\ad$ with $\delta(a)>1$, its eigenvalues on the spaces $J'_n$ (i.e., the values $\chi_n(a)$ for $\chi_n$  ranging in the poles of $L(\bullet^{-1}, \frac{1}{2})$) are unbounded. Hence, the seminorm should be zero for $n\gg 0$.
 
 Hence, $J'$ is the inverse limit of its $A_\ad$-finite quotients, i.e., is defined by a system of $A_\ad$-finite seminorms. But the only $A_\ad$-finite seminorms on $J=\mathbb H^\PW_{L(\bullet^{-1}\eta, \frac{1}{2})} \left(\widehat{A_\ad}_\CC, \CC_\bullet\right)$ are bounded by derivatives of the sections at various points of $\widehat{A_\ad}_\CC$ (let us call them ``delta functions''). 
 On the other hand, by Lemma \ref{globalsections-V}, Paley--Wiener sections generate the sheaf $\mathscr E_\ad^\circ$, hence their images under the surjective map \eqref{betweenMellin-circ} generate global sections of $\CC_\bullet(-[L(\bullet^{-1} \eta, \frac{1}{2})])$. In particular, no non-trivial linear combination of derivatives of delta functions is zero on $J^\circ$. We deduce that $J'=0$, concluding the proof of Proposition \ref{VHankelascoinv}. 
\end{proof}

\begin{remark}
 Besides issues about the existence of sufficiently many Paley--Wiener sections, we can explain representation-theoretically why the sheaves $ \mathscr D_{V_\ad}$ (representing $\mathcal S(V_\ad^*)$) and $\mathscr E_{V_\ad}^\circ$ (representing $\mathcal S(V_\ad)^\circ $) both surject onto $\CC_\bullet(-[L(\bullet^{-1} \eta, \frac{1}{2})])$. On one hand, sections of $\mathscr E_{V_\ad}^\circ$, considered as rational sections of $\mathscr D_{V_\ad}$, have poles at the poles of $\prod_{\omega\ne 1} L(\bullet^{-1} \cdot \omega, \frac{1}{2})$, with residues in finite-dimensional representations. But the torus $T_\ad$ acts by non-trivial characters on these finite-dimensional spaces of residues (when $\eta=1$), thus they do not appear when we mod out by $T_\ad$. On the other hand, at any pole $\chi$ of $L(\bullet^{-1} \eta, \frac{1}{2})$, the sections of $\mathscr E_{V_\ad}^\circ$ land in the finite-dimensional subrepresentation $W_\chi$ of the fiber $\mathscr D_{V_\ad,\chi} \simeq I(\chi^{-1})$ (a principal series representation for $\PGL_2$). In this case, it is not true that the map of coinvariants $(W_\chi)_{T_\ad}\to (I(\chi^{-1}))_{T_\ad}$ is an isomorphism. However, push-forward to $V\sslash T$ does not see the difference. The reason is that, in the split case, the quotient stack $[V_\ad^*/T_\ad] = [V^*/T]$ is isomorphic to a non-separated scheme: the affine line with the origin doubled. The push-forward map $\mathcal S(V^*)\to \mathcal S([V/T])$ has image in the usual Schwartz space $\mathcal S(V\sslash T)$ of the affine line, and does not see the ``doubled'' origin, i.e., the map $\mathcal S(V^*_\ad)_{T_\ad} = \mathcal S(V^*)_T \to \mathcal S([V/T])$ has a non-trivial kernel, modulo which the map   $(W_\chi)_{T_\ad}\to (I(\chi^{-1}))_{T_\ad}$ becomes an isomorphism.
\end{remark}

\section{The Rankin--Selberg variety} \label{sec:RS}

\subsection{The space and its orbital integrals} \label{ssRSthespace}

Let $(V,\omega)$ be a two-dimensional symplectic space, and let $\bar X$ be the variety of pairs $(v, g)$ with $v\in V, g\in \SL(V)$, with an action of the group $\tilde G = (\Gm\times \SL(V)^2)/\{\pm 1\}^\diag$ by
$$ (v, g) \cdot (a, g_1, g_2) = (a v g_1, g_1^{-1} g g_2).$$
In particular, $\bar X$ has a $\tilde G$-equivariant map to $V\times V$ by $(v,g)\mapsto (v, vg)$, and is a two-dimensional symplectic vector bundle over $Y:= \SL(V)= \SL(V)^\diag\backslash \SL(V)^2$. If we identify $V$ with the fiber over the identity, we can also write 
$$\bar X= \Ind_{\SL(V)^\diag}^{\SL(V)^2}V = V\times^{\SL(V)^\diag} \SL(V)^2.$$ The space $\bar X$ is a spherical variety under the action of $\tilde G$, whose open $\tilde G$-orbit, the complement of the zero section of the bundle, we will be denoting by $X$. 

If we fix a standard symplectic basis $(e_1,e_2)$ of $V$ (i.e., $\omega(e_1,e_2)=1$) to identify $\SL(V)$ with $\SL_2$ and $V^*=V\smallsetminus\{0\}$ with $N\backslash \SL_2$, where $N =$ the stabilizer of $e_2=$ the upper triangular unipotent subgroup, then $X = N^\diag\backslash (\SL_2 \times \SL_2)$, with the coset of $1$ corresponding to the element $(e_2\in V, 1\in \SL(V))$, and $\bar X$ is its affine closure. 

Throughout this paper, we denote by $A$ the universal Cartan of $\SL(V)$, and by $A_\ad=A/\{\pm 1\}$ the universal Cartan of $\PGL(V)$. We let $G$ be the group $(A\times \SL(V))/\{\pm 1\}^\diag$, and identify $\tilde G$ with the group $(A\times \SL(V)^2)/\{\pm 1\}^\diag$ via the character $\frac{\alpha}{2}$, the positive half-root on $A$. By this identification, we maintain our convention of considering the universal Cartan $A$ as a subgroup of the automorphism group of such a variety $X$ by using a Borel \emph{opposite} to the stabilizer $N^\diag$ of a point. In other words, if $B$ is the normalizer of $N$ in $\SL_2^\diag$, we let the quotient $A=B/N$ act by $G$-automorphisms on $X$ as:
$ a\cdot Nx = N({^{w_0}a}) x,$
where $w_0$ is the longest element of the Weyl group.
By this convention, the points of $Y$ are limits of the form $\lim_{t\to 0} (\lambda(t) x)$, where $x\in X$ and $\lambda$ is a \emph{dominant} cocharacter of $A$. The action of $A$ on functions or measures on $\bar X$ will be normalized, again, as
\begin{equation}\label{actionnormagain-measures}a\cdot \mu(x) = \delta(a)^{-\frac{1}{2}}\mu(a\cdot x)
\end{equation}
and on functions by 
\begin{equation}\label{actionnormagain-functions}a\cdot \Phi(x) = \delta(a)^{\frac{1}{2}}\Phi(a\cdot x),
\end{equation}
so that it is an $L^2$-isometry with respect to the $(\SL_2)^2$-invariant measure. (We will call an $(\SL_2)^2$-invariant measure on $X$ a ``Haar measure'' --- notice that it is a smooth measure on $\bar X$.) For half-densities, no such normalization is needed. Under these normalizations, the maps
$$ \mathcal F(\bar X) \to \mathcal D(\bar X) \to \mathcal S(\bar X)$$
from Schwartz functions to Schwartz half-densities and measures, that are given at every step by multiplication by a Haar half-density, are equivariant.

The symplectic structure gives rise to a Fourier transform:
\begin{equation}\label{FourierX} \mathfrak F: \mathcal D(\bar X)\xrightarrow\sim \mathcal D(\bar X),\end{equation}
defined fiberwise as in \S \ref{scatWhittaker}. Multiplication or division by a fixed Haar half-density turns \eqref{FourierX} into a morphism between spaces of measures or functions.
 
The transform is \emph{anti-equivariant} with respect to the $\tilde G$-action, in the sense that it twists the $\tilde G$-action by the automorphism of $\tilde G$ that is induced from the inversion map on $A$.

We now consider the quotients $X/\SL(V)$ and $\bar X/\SL(V)$, where $\SL(V)$ acts diagonally. If we use a symplectic basis to identify $X=N^\diag\backslash \SL_2^2$, as above, we have an isomorphism $X/\SL(V) = \frac{\SL_2}{N}$, \emph{which we fix to be the following:}
\begin{equation}\label{isommodN}
 N^\diag(g_1,g_2)\mapsto \frac{g_1^{-1} g_2}{N}.
\end{equation}
This is compatible with the isomorphism $\SL_2\simeq \SL_2^\diag\backslash \SL_2^2$ by the right orbit map on the identity element.

In terms of invariant theory, we have
$$ \mathfrak C_X := X\sslash \SL(V) = \bar X \sslash \SL(V) \simeq  \mathbbm A^2,$$
with coordinates $(c,t)$ coming from the invariants of the projections $\bar X\to V\times V$ and $\bar X\to \SL_2$; more specifically,
\begin{eqnarray}
 \nonumber
 c(v,g) &=& -\omega(v, vg), \mbox{ and } \\
 t(v,g) &=& \tr(g).
\end{eqnarray}
We have put a negative sign in the first invariant so that, in terms of the isomorphism $X\sslash \SL_2^\diag = \Dfrac{\SL_2}{N}$ of \eqref{isommodN}, the coordinates $(c,t)$ are: 
$$ \begin{pmatrix} a & b \\ c & d \end{pmatrix} \mapsto (c, t:=\tr = a+d).$$

We will eventually define several spaces of measures on $\mathfrak C_X$:
$$ \mathcal S(X/\SL_2) \subset \mathcal S(\bar X/\SL_2)^\circ \subset \mathcal S(\bar X/\SL_2).$$

The first and last are self-explanatory: $\mathcal S(X/\SL_2)$ and $\mathcal S(\bar X/\SL_2)$ are, correspondingly, the push-forwards of the spaces of Schwartz measures on $X$ and on $\bar X$. The intermediate space $\mathcal S(\bar X/\SL_2)^\circ$ is the most important one. In a sense that will be discussed,  the space $\mathcal S(\bar X/\SL_2)$ corresponds to the ($L$-function of the) sum of the symmetric square with the trivial representation of the group $G$, while $\mathcal S(\bar X/\SL_2)^\circ$ corresponds to just the symmetric square representation (for which the space $\mathcal S(X/\SL_2)$ would be too small). If we think of spaces, such as (suitable) algebraic varieties and stacks, as incorporating $L$-functions --- for instance, the $\Gm$-space $\mathbb A^1$ incorporating the standard $L$-function of $\Gm$, with the subspace $\Gm\subset \mathbb A^1$ incorporating the trivial (degree $0$) $L$-function --- it might be appropriate to think of $\mathcal S(\bar X/\SL_2)^\circ$ as the Schwartz space of a ``motive'' between $X/\SL_2$ and $\bar X/\SL_2$.  Unfortunately, I do not know how to make sense of such a ``motive''.

\subsection{Hankel transform for the Rankin--Selberg variety}

Let $ \Phi \in \mathcal F(\bar X)$, and let $f$ be the push-forward of $\Phi dx$ to $\mathfrak C_X$, where $dx$ is an invariant measure on $\bar X$. It is easy to see that there is a Haar measure on $\SL_2$ such that the following integration formula holds:
\begin{equation}\label{integrationX} \int_X \Phi(x) dx = \int_{\mathfrak C} O_{(c,t)} (\Phi) dc dt,\end{equation}
where $O_{(c,t)}$ is the orbital integral of $\Phi$ for the diagonal $\SL_2$-action, and $dc, dt$ denotes our fixed Haar measure on $F$. (Notice that $\SL_2$ acts freely over the open subset with $c\ne 0$.) Thus, the push-forward of the measure $\Phi dx$ is 
$$f(c,t) = O_{(c,t)}(\Phi) dc dt = O_{(c,t)}(\Phi) |c| \cdot |t| \cdot d^\times c d^\times t .$$
We define the push-forward of the half-density $\Phi (dx)^\frac{1}{2}$ to be   
$$O_{(c,t)}(\Phi) (dc dt)^\frac{1}{2}.$$
This way, we obtain spaces $\mathcal F(\bar X/\SL_2)$ and $\mathcal D(\bar X/\SL_2)$ of densely defined ``push-forward'' functions and half-densities on $\C_X$.

The fiberwise symplectic Fourier tranform $\mathfrak F$ on $\mathcal D(\bar X)$ descends to a ``Hankel transform'' $\mathcal H_X$ on $\SL_2^\diag$-coinvariants, which can identified with the space $\mathcal D(\bar X/\SL_2)$.
Rather than proving this identification, we will directly prove that the Fourier transform descends to an endomorphism of $\mathcal D(\bar X/\SL_2)$, together with an explicit formula for it.  
 
\begin{theorem}\label{RSHankel}
 We have a commutative diagram
 $$ \xymatrix{
 \mathcal D(\bar X) \ar[r]^{\mathfrak F}\ar[d] & \mathcal D(\bar X) \ar[d] \\ 
 \mathcal D(\bar X/\SL_2) \ar[r]^{\mathcal H_X} & \mathcal D(\bar X/\SL_2),
 } $$
 where the bottom horizontal arrow is given by
\begin{multline}\label{HankelX-densities} \mathcal H_X \varphi(c,t) =  \lambda(\eta_{t^2-4},\psi)^{-1} \left((\psi(\frac{1}{\bullet}) |\bullet|^{-\frac{1}{2}} d^\times\bullet) \star_c \right.\\
\left. (\psi(\frac{1}{\bullet}) \eta_{t^2-4}(\bullet)|\bullet|^{-\frac{1}{2}} d^\times\bullet) \star_c  \varphi\right)((4-t^2) c^{-1}, t), 
\end{multline}
where:
\begin{itemize}
\item $E_D$ denotes the quadratic extension $F(\sqrt{D})$, and $\eta_{D}$ the associated quadratic character of $F^\times$;
\item $\star_c$ denotes multiplicative convolution in the variable $c$ (with fixed $t$), understood in the regularized sense of \S \ref{sssFourierconv}.
\end{itemize}

\end{theorem}

\begin{remark}
We also translate this formula to push-forwards of Schwartz measures and push-forwards (orbital integrals) of Schwartz functions on $\bar X$, using the integration formula \eqref{integrationX}. Here we will use the regular multiplicative convolution in the $c$-coordinate, without any normalization: 

Fourier transform on measures descends to the involution 
 \begin{multline}\label{HankelX-measures} \mathcal H_X  f(c,t) = \lambda(\eta_{t^2-4},\psi)^{-1}  \left|\frac{c^2}{4-t^2}\right|^{\frac{1}{2}}\cdot\\
\cdot \left( (\psi(\frac{1}{\bullet}) d^\times\bullet) \star_c (\psi(\frac{1}{\bullet}) \eta_{t^2-4}(\bullet) d^\times\bullet) \star_c  f\right)((4-t^2)c^{-1}, t)
\end{multline}
on $\mathcal S(\bar X/\SL_2)$.

Fourier transform on functions descends to the involution
\begin{multline}\label{HankelX-orbital} O_{(c,t)} (\mathfrak F \Phi) = \lambda(\eta_{t^2-4},\psi)^{-1} \left|\frac{4-t^2}{c^2}\right|^{\frac{1}{2}}\cdot
\\ \cdot \left((\psi(\frac{1}{\bullet}) |\bullet|^{-1} d^\times\bullet) \star_c (\psi(\frac{1}{\bullet}) \eta_{t^2-4}(\bullet)|\bullet|^{-1} d^\times\bullet) \star_c O_{(\bullet, t)}\Phi\right)((4-t^2)c^{-1}, t)
\end{multline}
on $\mathcal F(\bar X/\SL_2)$, for $t\ne \pm 2$ and $c\ne 0$.

\end{remark}

 \begin{proof}

The fiber of $\bar X$ over a fixed value $t\ne \pm 2$ is isomorphic to $\tilde V_t = V\times^{T_t} \SL(V)$, where $T_t$ is the centralizer of a $g\in \SL(V)$ with $\tr(g)=t$, and $V$ is identified with the fiber $V\times\{g\}$ of $\bar X$ over $g$. Equivalently, $T_t$ can be identified with the special orthogonal group of the quadratic form $c(v) = -\omega(v, vg)$. If $t^2-4\ne 0$, this quadratic form is non-degenerate, and $T_t$ is a torus. For any $\varphi\in \mathcal D(\bar X/\SL_2)$, the section
$$ \frac{\varphi}{(dt)^\frac{1}{2}} : t\mapsto \frac{\varphi(\bullet, t)}{(dt)^\frac{1}{2}}$$ 
is a section of elements of the spaces $\mathcal D([V/T_t])$ (when $t\ne \pm 2$), and since symplectic Fourier transform is defined fiberwise, the Hankel transform $\mathcal H_X$ exists, and satisfies:
$$ \frac{\mathcal H_X \varphi}{(dt)^\frac{1}{2}} (t) = \mathcal H_t \left(\frac{\varphi(\bullet, t)}{(dt)^\frac{1}{2}}\right),$$
where $\mathcal H_t$ is the symplectic Hankel transform of Proposition \ref{symplecticHankel}, now with an index $t$ to denote the specific fiber. 
 
One immediately computes that the discriminant of the quadratic form $c(v) = -\omega(v, vg)$ is $d = -\frac{1}{4} (\tr(g)^2-4) = -\frac{1}{4} (t^2-4)$. Applying now Proposition \ref{symplecticHankel}, the result follows.

\end{proof}

\begin{remark}
Let us also explicitly see how \eqref{HankelX-orbital} is obtained from a formal calculation (which is justified by Proposition \ref{symplecticHankel}):

We identify $X$ with $N^\diag\backslash (\SL_2)^2$, and represent its points by pairs $(g_1, g_2)$  of elements in $\SL_2$. The fiber of $\bar X$ over $1\in Y=\SL_2$ is thus identified with $\mathbbm A^2$, the affine closure of $N\backslash \SL_2$, and Fourier transform on this fiber is obtained by identifying the point $(x,y)\in \mathbbm A^2$ with the coset $\begin{pmatrix} * & * \\ x & y \end{pmatrix} \in N\backslash \SL_2$, and using the symplectic form $dx \wedge dy$ on $\mathbbm A^2$. Explicitly, if $\varphi$ is a function on $\mathbbm A^2$, its symplectic Fourier transform on the fiber is given by 
\begin{equation}\label{Fouriertransformeq} \hat\varphi(c,d) = \int_{F^2} \varphi(x,y) \psi(xd-yc),
\end{equation}
  and the Fourier transform $\mathfrak  F$ on $\mathcal S(\bar X)$ is induced from this fiberwise Fourier transform. 
  In particular, if $\varphi$ is the restriction of $\Phi$ over the fiber of $1\in Y$, the evaluation of $\mathfrak F\Phi$ at the point represented by $1$ in $N^\diag\backslash (\SL_2)^2$ is given by   \eqref{Fouriertransformeq} setting $c=0$ and $d=1$. Translating by $(g_1,g_2)$ on the right, we obtain the formula
  $$\mathfrak F\Phi(g_1, g_2) = \int_{F^2} \Phi \left( \begin{pmatrix} y^{-1} \\ x & y \end{pmatrix} g_1, \begin{pmatrix} y^{-1} \\ x & y \end{pmatrix} g_2 \right) \psi(x) dx dy.$$
  
  For $c\ne 0$, we can choose a section of the map $X \to X\sslash \SL_2$ by $(c, t=\tr)\mapsto (\begin{pmatrix} t & -c^{-1} \\ c\end{pmatrix},1)$ (modulo $N^\diag$ on the left). We compute the orbital integrals for the $\SL_2^\diag$-action:
  
  \[O_{(c,t)} (\mathfrak F\Phi)  = \int_{\SL_2} \int_{F^2} \Phi \left( \begin{pmatrix} y^{-1} \\ x & y \end{pmatrix} \begin{pmatrix} t & -c^{-1} \\ c\end{pmatrix} g, \begin{pmatrix} y^{-1} \\ x & y \end{pmatrix} g \right) \psi (x) dx dy dg.\]
  
  We are justified to change the order of integration (see Remark \ref{remarkchangeorder}) and write   
  \begin{multline*}  \int_{F^2}^* \int_{\SL_2}  \Phi \left( \begin{pmatrix} y^{-1} \\ x & y \end{pmatrix} \begin{pmatrix} t & -c^{-1} \\ c\end{pmatrix} \begin{pmatrix} y^{-1} \\ x & y \end{pmatrix}^{-1} g, g\right)  dg \psi (x) dx dy  = \\
   = \int_{F^2}^* O_{(c^{-1} x^2 + txy + cy^2, t)}(\Phi)  \psi^{-1} (x) dx dy .
   \end{multline*}
  
  We can write 
  \[c^{-1} x^2 + txy + cy^2 = c^{-1} (4-t^2) \frac{x^2 + (4-t^2)^{-1} v^2}{4} = c^{-1} (4-t^2) N(z),\] where $v = tx + 2cy$, 
  and $z = \frac{x+ (t^2-4)^{-\frac{1}{2}} c v}{2}$ is an element in the quadratic extension $E=E_{t^2-4}=F(\sqrt{t^2-4})$ (so that $x=\tr(z)$). The measure $dx dy = |2c|^{-1} dx dv$ is $|c^{-2}(4-t^2)|^{\frac{1}{2}}$ times the measure $dz$ on $E_{4-t^2}$ which is self-dual with respect to the character $\psi_E=\psi\circ\tr$ on $E$. Hence, we can write
  $$ O_{(c,t)} (\mathfrak F\Phi) = |c^{-2}(4-t^2)|^{\frac{1}{2}} \int_{E}^* O_{(c^{-1}(4-t^2) N(z), t)} (\Phi) \psi_E^{-1}(z) dz$$
  which, again by Remark \ref{remarkchangeorder}, is \eqref{HankelX-orbital}. 
\end{remark}

\subsection{The $\Sym^2$-subspace}\label{ssSym2}

The Rankin--Selberg method allows us to study the tensor product $L$-function of two automorphic representations. This is because, locally, there is an unfolding map (that we will study in the next section)
$$ \U: \mathcal F(\bar X) \to C^\infty((N,\psi)^2\backslash \tilde G),$$
such that the characteristic function of $1_{\bar X(\mathfrak o)}$, for $F$ non-Archimedean with ring of integers $\mathfrak o$, maps to a Whittaker function which is the generating series for the local unramified $L$-function $L(\pi_1\times\pi_2, s)$. We will study this unfolding operator in the next section.

In this paper, we study the descent of $\mathcal F(\bar X)$ modulo the diagonal action of $\SL_2$, which can be spectrally decomposed in terms of representations (of $\SL_2^2/\{\pm 1\}^\diag$) of the form $\pi_1\otimes\pi_2$ with $\pi_1=\widetilde{\pi_2}$. If we set $\pi:=\pi_1 = \widetilde{\pi_2}$, the $L$-function $L(\pi_1\times\pi_2, s)$ decomposes as $L(\pi,\Sym^2,s) \zeta(s)$. 

In this subsection, we want to extract a subspace $\mathcal S(\bar X/\SL_2)^\circ$ of $\mathcal S(\bar X/\SL_2)$ that is ``responsible'' for the factor $L(\pi,\Sym^2,s)$. Geometrically, it will correspond to the spaces $\mathcal S([V/T_t])^\circ$ on the various fibers, which we saw in \S \ref{ssA2}, in the sense that its elements, divided by the measure $dt$ and restricted to fibers with $t\ne \pm 2$ will indeed be elements of $\mathcal S([V/T_t])^\circ$. However, this geometric property does not fully describe the space here, because it gives us no control over the behavior as $t\to \pm 2$.

In the non-Archimedean case, the space $\mathcal S(\bar X/\SL_2)^\circ$ can be obtained from $\mathcal S(\bar X/\SL_2)$, as in the case of $\mathcal S([V/T])^\circ$, by applying the element of the completed Hecke algebra $h\in \widehat{\mathcal S(A_\ad)}$ with Mellin transform $\check h(\chi) = L(\chi,\check\alpha, \frac{1}{2})^{-1}$. Again, such a multiplier is not available in the Archimedean case, so we will work with Mellin transforms, imitating the description of $\mathcal S([V/T])^\circ$ in \S \ref{ssVcircMellin}.

\subsubsection{Mellin transforms}

Let $\mathscr D$ be the Fr\'echet bundle over the character group $\widehat{A}_\CC$ whose fiber over $\chi$ is the space $\mathcal S(A\backslash X, \mathcal L_\chi)$ of Schwartz measures on $A\backslash X$, valued in the line bundle $\mathcal L_\chi$ whose sections are $(A,\chi)$-equivariant functions on $X$, for the normalized action. The space $\mathcal S(A\backslash X, \mathcal L_\chi)$ is in non-degenerate duality with the space $C^\infty_\temp(A\backslash X,\mathcal L_{\chi^{-1}})$ of tempered $(A,\chi^{-1})$-equivariant functions, and integrating such functions against Schwartz measures on $X$ we get a Mellin transform:
$$ \mathcal S(X)\ni f\mapsto \check f(\chi) \in \mathcal S(A\backslash X, \mathcal L_\chi).$$

As in \S \ref{ssVcircMellin}, there is a natural notion of Paley--Wiener sections $\mathbb H^\PW(\widehat{A}_\CC,\mathscr D)$ of this
Fr\'echet bundle. Moreover,  since 
$$X\simeq V^*\times \SL(V) = V^*\times^{\SL(V)^\diag} \SL(V)^2,$$ 
where $V^*$ denotes the complement of zero in $V$, we have 
$$\mathscr D_\chi = \mathcal S(A\backslash X, \mathcal L_\chi) \simeq \mathcal S(A\backslash V^*, \mathcal L_\chi)\hat\otimes \mathcal S(\SL(V)) = \mathscr D_{V,\chi}\hat\otimes \mathcal S(\SL(V)),$$
in the notation of \S \ref{ssVcircMellin}, 
and the Paley--Wiener sections of $\mathscr D$ are 
$$\mathbb H^\PW(\widehat{A}_\CC,\mathscr D) = \mathbb H^\PW(\widehat{A}_\CC,\mathscr D_V)\hat\otimes \mathcal S(\SL(V)).$$

It immediately follows by induction from $V^*$ that the image of Mellin transform on $\mathcal S(V^*)$ is the space of Paley--Wiener sections of $\mathscr D$, giving rise to an isomorphism
\begin{equation}\label{SX}
\mathcal S(X) \xrightarrow\sim \mathbb H^\PW (\widehat A_\CC, \mathscr D).
\end{equation}
Proposition \ref{MellinVimage} allows us, by induction, to extend the transform and the description of its image to $\mathcal S(\bar X)$:

\begin{proposition}
Mellin transform converges on $\mathcal S(\bar X)$ for $|\chi|=\delta^s$ with $\Re(s)\ll 0$, and admits rational/meromorphic continuation to all $\widehat{A}_\CC$, giving rise to an isomorphism
\begin{equation}\label{SbarX}\mathcal S(\bar X) \xrightarrow\sim \mathbb H^\PW (\widehat A_\CC, \mathscr E),
\end{equation}
where $\mathscr E\subset \mathscr D(-[L(\bullet,-\check\alpha,1)])$
is the subsheaf of those sections whose residues at the poles of $L(\bullet,-\check\alpha,1)$, where the fiber 
$$\mathscr D_\chi = \Ind_{\SL(V)^\diag}^{\SL(V)^2} \mathscr D_{V,\chi}$$ 
contains the induction $\tilde W_\chi:=\Ind_{\SL(V)^\diag}^{\SL(V)^2} W_\chi$ of a finite-dimensional representation $W_\chi$ of $\SL(V)$, lie in $\tilde W_\chi$.
\end{proposition}

We will define the $\Sym^2$-subspace $\mathcal S(\bar X/\SL_2)^\circ$ by descending from a space ``upstairs'' but again, for that purpose, we need to work with the quotient stack $\bar X_\ad = [\bar X/\{\pm 1\}]$ which already carries an action of $A_\ad$. Notice the isomorphism of stacks:
\begin{equation}\label{stacksmodpm1} [\bar X/\SL_2] = [\bar X_\ad/\PGL_2],
\end{equation}
which implies that the push-forward space $\mathcal S(\bar X/\SL_2)$ is also the image of the Schwartz space $\mathcal S(\bar X_\ad)$.
In terms of (isomorphism classes of) $F$-points, $\bar X_\ad$ admits a similar description as the stack $V_\ad$ that we encountered in \S \ref{ssVcircMellin}, namely,  
$$\bar X_\ad(F)= \bigsqcup_\alpha \bar X^\alpha(F)/\{\pm 1\},$$
where $\alpha$ runs over all quadratic extensions of $F$, and $\bar X^\alpha$ is the $\bar X$-torsor over the same base $Y = \SL(V)$ that we obtain by twisting $\bar X$ by a $\Z/2$-torsor. Thus, the Schwartz space $\mathcal S(X_\ad)$ can be identified with the space of $\{\pm 1\}$-coinvariants (or invariants) of the sum 
$$ \bigoplus_\alpha \mathcal S(\bar X^\alpha(F)).$$

We have an isomorphism
$$ \bar X_\ad \simeq V_\ad \times^{\SL(V)^\diag} \SL(V)^2,$$
and this immediately implies from Corollary \ref{corMellinVad} that the Schwartz space $\mathcal S(X_\ad)$ admits a Mellin transform with respect to the action of $A_\ad$ (I leave the definitions to the reader, since they are completely analogous to the previous ones), whose image is described by the following:

\begin{proposition}\label{MellinXadimage}
 Mellin transform gives rise to isomorphisms:
 \begin{equation}\label{SXad}
\mathcal S(X_\ad) \xrightarrow\sim \mathbb H^\PW (\widehat{A_\ad}_\CC, \mathscr D_\ad), 
\end{equation}
and
\begin{equation}\label{SbarXad}\mathcal S(\bar X_\ad) \xrightarrow\sim \mathbb H^\PW (\widehat{A_\ad}_\CC, \mathscr E_\ad),
\end{equation}
where $\mathscr D_\ad, \mathscr E_\ad$ are the pullbacks of $\mathscr D, \mathscr E$ under the natural map $s^*: \widehat{A_\ad}_\CC \to \widehat{A}_\CC$.
\end{proposition}

The space $\mathcal S(\bar X/\SL_2)^\circ$ that we are after will arise, as in the case of $\mathcal S([V/T])^\circ$, from the subsheaf $\mathscr E_\ad^\circ = \mathscr E_\ad([L(\bullet, -\frac{\check\alpha}{2}, \frac{1}{2})])$, i.e.\ those sections of $\mathscr E_\ad$ which vanish at poles of the local $L$-function $L(\chi, -\frac{\check\alpha}{2}, \frac{1}{2})$. Notice that, as meromorphic sections of the bundle $\mathscr D_\ad = s^{**}\mathscr D$, those are holomorphic at the poles of this $L$-function, and valued in the induction of the finite-dimensional subrepresentation of the principal series $I(\chi^{-1})$ of $\SL(V)^\diag$; they may have simple poles at the poles of $L(\chi\omega, -\frac{\check\alpha}{2}, \frac{1}{2})$, for $\omega$ a non-trivial quadratic character, but their residues also lie in the induction of the finite-dimensional subrepresentation.
We define $\mathcal S(\bar X_\ad)^\circ$ as  the subspace of $\mathcal S(\bar X_\ad)$ which, under the Mellin transform of Proposition \eqref{MellinXadimage} coincides with the subspace $\mathbb H^\PW (\widehat{A_\ad}_\CC, \mathscr E_\ad^\circ)$. 
Pushing forward to $\mathfrak C_X = \bar X\sslash \SL_2$, we define $\mathcal S(\bar X/\SL_2)^\circ$ as the closure of the image of $\mathcal S(\bar X_\ad)^\circ$. We have a corresponding space $\mathcal D(\bar X_\ad)^\circ$ of half-densities, dividing by the half-density $(dc)^\frac{1}{2} (dt)^\frac{1}{2}$.

\begin{lemma}\label{globalsections} 
 The sheaves $\mathscr E_\ad$ and $\mathscr E_\ad^\circ$ are generated by their global Paley--Wiener sections. In the non-Archimedean case, the action of the element $h\in \widehat{S(A_\ad)}$ whose Mellin transform is $L(\chi, \frac{\check\alpha}{2}, \frac{1}{2})^{-1}$ gives rise to an isomorphism:
 $$ \mathcal S(\bar X_\ad)\xrightarrow{h\cdot} \mathcal S(\bar X_\ad)^\circ.$$
\end{lemma} 

\begin{proof}
The proof is identical to the one of the analogous statements for $V$ (Corollary \ref{corMellinVad} and Lemma \ref{globalsections-V}).
\end{proof}

Recall that $\bar X\simeq V\times \SL(V)$ and $\bar X_\ad\simeq V_\ad\times \SL(V)$. If we fix a Haar measure $dg$ on $\SL(V)$, the quotient $\frac{\mathcal D(\bar X_\ad)^\circ }{(dg)^\frac{1}{2}}$ consists of functions: 
$$\SL(V)\to \mathcal D(V_\ad)^\circ.$$ 

On the spaces $\mathcal D(V_\ad)^\circ$ we have, by Proposition \ref{factorFourier}, an endomorphism $\mathfrak F^\circ$. In this subsection (and later), we will denote by $\mathfrak F^\circ$ the induced endomorphism of $\mathcal D(\bar X_\ad)^\circ$, defined fiberwise. It is $A_\ad$-anti-equivariant under the normalized action. 

Here is the main result of this subsection:

\begin{theorem}\label{thmsubspaceX}
We have a commutative diagram
\begin{equation}
 \xymatrix{
 \mathcal D(\bar X_\ad)^\circ \ar[r]^{\mathfrak F^\circ} \ar[d] & \mathcal D(\bar X_\ad)^\circ \ar[d] \\
 \mathcal D(\bar X/\SL_2)^\circ \ar[r]^{\mathcal H_X^\circ} & \mathcal D(\bar X/\SL_2)^\circ,
 }
\end{equation}
where 
\begin{multline}\label{HXcirc} \mathcal H_X^\circ \varphi(c,t) =  \lambda(\eta_{t^2-4},\psi)^{-1}\left( (\psi(\frac{1}{\bullet}) \eta_{t^2-4}(\bullet)|\bullet|^{-\frac{1}{2}} d^\times\bullet) \star_c  \varphi\right)((4-t^2) c^{-1}, t). 
\end{multline}
\end{theorem}

\begin{remark}
 The operators $\mathcal H_X^\circ$, $\mathcal H_X$ are related by
 \begin{equation}\label{relationbetweenH}\mathcal H_X = \mathscr F_{\frac{\check\alpha}{2}, \frac{1}{2}} \circ \mathcal H_X^\circ = \mathcal H_X^\circ \circ \mathscr F_{-\frac{\check\alpha}{2}, \frac{1}{2}},
 \end{equation}
 where $\mathscr F_{\pm \frac{\check\alpha}{2}, \frac{1}{2}}$ are the Fourier convolutions of \S \ref{sssFourierconv} for the $A_\ad$-action on half-densities on $\bar X\sslash \SL_2$.
\end{remark}

\begin{proof}
Notice that for every $g\in \SL(V)$ with $t:=\tr(g)\ne \pm 2$, the centralizer of $g$ is a torus $T_t$, and the fiber of the stack $[\bar X/\SL_2]=[\bar X_\ad/\PGL_2]$ over $g$ is isomorphic to $[V_t/T_t]=[V_{t,\ad}/T_{t,\ad}]$, where $V_t\simeq V$ is the fiber of $\bar X$ over $g$, and the notation is otherwise analogous to the one of Proposition \ref{VHankelascoinv}; the result now follows from that proposition. 
\end{proof}

\begin{corollary}
 On Mellin transforms, we have
 \begin{equation}\label{relationbetweenH-Mellin}
  \widecheck{\mathcal H_X\varphi} (\chi) = \gamma(\chi,\frac{\check\alpha}{2}, \frac{1}{2},\psi)\cdot \widecheck{\mathcal H_X^\circ \varphi}(\chi) 
 \end{equation}
 for every $\varphi\in \mathcal D(\bar X/\SL_2)^\circ$. 
\end{corollary}

\begin{proof}
 This follows from the above and \eqref{FE}.
\end{proof}

\begin{remark}
On the analogous space of measures $\mathcal S(\bar X/\SL_2)^\circ$, \emph{again without normalization of the convolution action in the $c$-variable}, the Hankel transform $\mathcal H_X^\circ$ will read:
  \begin{equation}\label{HankelXcirc-measures} \mathcal H_X^\circ  f(c,t) = \lambda(\eta_{t^2-4},\psi)^{-1}  \left|\frac{c^2}{4-t^2}\right|^{\frac{1}{2}} \left( (\psi(\frac{1}{\bullet}) \eta_{t^2-4}(\bullet) d^\times\bullet) \star_c  f\right)((4-t^2)c^{-1}, t),
\end{equation}
cf.\ \eqref{HankelX-measures}.
\end{remark}

\subsection{Basic vectors}\label{ssbasicvectorsX} 

We will be talking about ``unramified data'' for the space $\bar X$ when $F$ is non-Archimedean, unramified over the base field $\mathbb Q_p$ or $\mathbb F_p((t))$, and, if we denote by $\mathfrak o$ the ring of integers of $F$, the symplectic space $V$ used to define $\bar X$ is defined over $\mathfrak o$, with an integral and residually non-vanishing symplectic form $\omega$. Assume this to be the case for this subsection.

Let $\Phi^0 \in \mathcal F(\bar X)$ be the characteristic function of $\bar X(\mathfrak o)$. We define the \emph{basic vector} $f_{\bar X}$ of $\mathcal S(\bar X/\SL_2)$ to be the image of $\Phi^0 dx$ in $\mathcal S(\bar X/\SL_2)$, where $dx (\bar X(\mathfrak o))=1$. We have a normalized action of $A_\ad$ on $\bar X$; if we act on $f_{\bar X}$ by the element $h=h_{L(\frac{\check \alpha}{2},\frac{1}{2})^{-1}}$ of the completed Hecke algebra $\widehat{\mathcal S(A_\ad)}$ whose Mellin transform is $\chi\mapsto L(\chi, \frac{\check \alpha}{2},\frac{1}{2})^{-1}$, we obtain an element $f_{\bar X}^\circ \in \mathcal S(\bar X/\SL_2)^\circ$ which will be our \emph{basic vector} for this space: 
\begin{equation}\label{deff00}
 f_{\bar X}^\circ:= h_{L(\frac{\check \alpha}{2},\frac{1}{2})^{-1}}\cdot f_{\bar X}.
\end{equation}
Notice that $f_{\bar X}$ is $A_\ad(\mathfrak o)$-invariant, so we can replace $h$ by its image $h^{\rm unr}$ in the unramified Hecke algebra of $A_\ad$; our conventions on Mellin transforms are inverse to usual conventions on Satake transforms, hence the Satake transform of $h^{\rm unr}$, as a polynomial on the dual torus $\check A_\ad$, will be is $(1-q^{-\frac{1}{2}}e^{-\frac{\check\alpha}{2}})$, where $e^{-\frac{\check\alpha}{2}}$ is understood as a character of $\check A_\ad$. By construction, the Mellin transforms of these two basic functions satisfy:
\begin{equation}\label{Mellinbasic}
\check f_{\bar X}(\chi) = L(\chi, -\frac{\check \alpha}{2},\frac{1}{2}) \check f^\circ_{\bar X}(\chi).
\end{equation}

For later use, I note that we have one more distinguished unramified vector in $\mathcal S(\bar X)$, namely, the invariant probability measure on $X(\mathfrak o)$; its image in $\mathcal S(\bar X/\SL_2)$ will be denoted by $f_X$. The element in the unramified Hecke algebra of $A$ whose Satake transform is $(1-q^{-1}e^{-\check\alpha})$ maps (under the normalized action) the characteristic function of $1_{\bar X(\mathfrak o)}$ to the characteristic function of $1_{X(\mathfrak o)}$; multiplying by an invariant measure, and taking into account that the ratio between the measures of $X(\mathfrak o)$ and $\bar X(\mathfrak o)$ is $(1-q^{-2})$, we obtain that 
\begin{equation}\label{barXtoX}
 h_{L(\check\alpha, 1)^{-1}}\cdot f_{\bar X} = (1-q^{-2}) f_X,
\end{equation}
where $h_{L(\check\alpha, 1)^{-1}}\in \widehat{\mathcal S(A_\ad)}$ has Mellin transform $\chi\mapsto L(\chi,\check\alpha, 1)^{-1}$. 

On the spaces of half-densities $\mathcal D(\bar X/\SL_2)$, $\mathcal D(\bar X/\SL_2)^\circ$, we define the basic vectors, denoted by the same symbols $f_{\bar X}$ and $f_{\bar X}^\circ$, as the quotients of the corresponding basic measures by the half-density $(dc\,  dt)^\frac{1}{2}$. 

As usual, it makes sense to act on the basic vectors by an element of the unramified Hecke algebra of $\tilde G$, by acting on the defining measures on $\bar X$. 

\begin{theorem}\label{thmbasicvectorX}
 The Hankel transforms $\mathcal H_X$, resp $\mathcal H_X^\circ$, map
 $$ \mathcal H_X: h\cdot f_{\bar X} \mapsto (\iota^* h) \cdot f_{\bar X},$$
 resp.
 $$ \mathcal H_X^\circ: h\cdot f^\circ_X \mapsto (\iota^* h) \cdot f^\circ_{\bar X},$$
 for any element of the unramified Hecke algebra of $\tilde G$, where $\iota:\tilde G\to \tilde G$ is the involution induced from inversion on the ``factor'' $A$ of $\tilde G = (A\times \SL(V)^2)/\{\pm 1\}^\diag$.
\end{theorem}

\begin{proof}
 Symplectic Fourier transform preserves the characteristic function of $\bar X(\mathfrak o)$,  and is $(\tilde G,\iota)$-equivariant, hence the statement about $\mathcal H_X$. 
 
 For $\mathcal H_X^\circ$ it follows by considering Mellin transforms, and invoking \eqref{relationbetweenH} and \eqref{FE}.
\end{proof}

\section{Hankel transform for the symmetric square $L$-function of $\GL_2$.} \label{sec:sym2}

The representation $\Sym^2$ of $\GL_2$ factors through $\GL_2/\{\pm 1\}\simeq \Gm\times \PGL_2 $, the dual group of $G:=\Gm\times \SL_2$, and coincides with the tensor product of the identity representation of $\Gm$ with the adjoint representation $\Ad$ of $\PGL_2$. We wish to study the functional equation for its $L$-function, at the level of the Kuznetsov formula of $G$. In what follows, the group $G$ will be identified with the group that, in the notation of the previous section, was $A_\ad\times\GL(V)$, with $A_\ad$ identified with $\Gm$ via the positive root character.

We denote the three weights of $\Sym^2$ by $\check\lambda_+, \check\lambda_0$ and $\check\lambda_-$, considered as coweights of the universal Cartan of $G$, so that $\check\lambda_-$ is anti-dominant and $\check\lambda_+$ is dominant.

The main goal of this section is to describe a local transformation (``Hankel transform'') between certain spaces of non-standard half-densities for the Kuznetsov formula, which is responsible for the functional equation of the symmetric-square $L$-function: 
$$\mathcal H_{\Sym^2}:\mathcal D^-_{L(\Sym^2, \frac{1}{2})} (N,\psi\backslash G /N,\psi )\xrightarrow\sim \mathcal D^-_{L((\Sym^2)^\vee, \frac{1}{2})} (N,\psi\backslash G /N,\psi ).$$

To describe it, we denote by $\eta_{D}$ the quadratic character associated to the quadratic extension $F(\sqrt{D})$, considered as a character of the $\Gm$-factor, and identified with the operator of multiplication by this character. We denote by $\delta_a$ the operator of multiplicative translation by $a$, under the $\Gm$-action on $N\backslash G\sslash N \simeq \Gm \times (N\backslash \SL_2\sslash N)$. The letter $\zeta$ denotes the coordinate on $N\backslash \SL_2\sslash N$ (hence also on $N\backslash G\sslash N$, by projection) that was fixed in \S \ref{ssnonstandard}. 

Notice that, since $\Gm$ is a direct factor of $G$ (hence also of $N\backslash G\sslash N$), for any space of half-densities on $N\backslash G\sslash N$ and any character $\chi$ of $\Gm(F)$, one has a $(\Gm,\chi)$-equivariant projection map to half-densities on $N\backslash \SL_2\sslash N$, given by an integral over the $\Gm$-factor --- assuming that this integral converges. This integral will depend on the choice of a Haar measure on $\Gm$, which we fix to be the standard one used throughout this paper ($=\frac{dx}{x}$, where $dx$ is self-dual with respect to the fixed additive character $\psi$). We will call this integral \emph{the $(\Gm,\chi)$-equivariant integral} of the given space of half-densities. Explicitly, for a half-density $\varphi$ on $N\backslash G\sslash N$, its ``$(\Gm,|\bullet|^s)$-equivariant integral'' is the following half-density on $N\backslash \SL_2\sslash N$:
$$ \zeta\mapsto \int_{\Gm} \varphi(a,\zeta) |a|^{-s} (d^\times a)^\frac{1}{2}.$$

Throughout this section, including in the main theorem that follows, we make use of the non-standard spaces $\mathcal S^-_{L(\Ad, \chi)} (N,\psi\backslash \SL_2/N,\psi)$ of test measures for the Kuznetsov formula of $\SL_2$, introduced in \S \ref{ssnonstandard} --- where, we recall, when $\chi = |\bullet|^s$, we replace it simply by $s$ in the notation --- and the related spaces of half-densities $\mathcal D^-_{L(\Ad, s)} (N,\psi\backslash \SL_2/N,\psi)$ introduced at the end of that section. These spaces were defined explicitly in terms of their germs; in contrast, we will not provide such descriptions for non-standard spaces of test densities for the Kuznetsov formula of $G$, which live over a 2-dimensional space of orbits --- rather, they will be defined abstractly in the beginning of \S \ref{ssHankel} as images of certain spaces associated to the Rankin--Selberg variety.

\begin{theorem} \label{thmSym2}
There is a space $\mathcal D^-_{L(\Sym^2, \frac{1}{2})} (N,\psi\backslash G /N,\psi)$ of (densely defined) half-densities on $N\backslash G\sslash N$, containing $\mathcal D (N,\psi\backslash G /N,\psi)$, whose $(\Gm,|\bullet|^s)$-equivariant integrals converge for $\Re(s)\gg 0$, admit meromorphic continuation to the entire complex plane, and have image equal to $\mathcal D^-_{L(\Ad, \frac{1}{2}+ s)} (N,\psi\backslash \SL_2/N,\psi)$, for every $s$ away from the poles. 

Moreover the transform:
  \begin{equation}\label{HSymformula} 
  \mathcal H_{\Sym^2} = \lambda(\eta_{\zeta^2-4},\psi)^{-1} \mathscr F_{-\check\lambda_+,\frac{1}{2}} \circ \delta_{1-4\zeta^{-2}} \circ \eta_{\zeta^2-4} \circ \mathscr F_{-\check\lambda_0,\frac{1}{2}} \circ \eta_{\zeta^2-4} \circ \mathscr F_{-\check\lambda_-,\frac{1}{2}}\end{equation}
is a $\Gm$-equivariant isomorphism 
$$\mathcal D^-_{L(\Sym^2, \frac{1}{2})} (N,\psi\backslash G /N,\psi )\xrightarrow\sim \mathcal D^-_{L((\Sym^2)^\vee, \frac{1}{2})} (N,\psi\backslash G /N,\psi ),$$
where the space on the right is the analogous space with the $\Gm$-coordinate inverted, hence descends for every $s\in \CC$ away from the poles to an isomorphism
$$\mathcal H_{\Ad,s} : \mathcal D^-_{L(\Ad, \frac{1}{2}+ s)} (N,\psi\backslash \SL_2/N,\psi) \xrightarrow\sim \mathcal D^-_{L(\Ad, \frac{1}{2}- s)} (N,\psi\backslash \SL_2/N,\psi).$$

As $\pi$ varies in a family of representations twisted by the character $|\bullet|^s$ of $\Gm$, the Hankel transform satisfies:
\begin{equation}\mathcal H_{\Sym^2}^* J_\pi = \gamma(\pi, \Sym^2, \frac{1}{2}, \psi) \cdot J_\pi,\end{equation}
as meromorphic families of functionals on $\mathcal D^-_{L(\Sym^2,\frac{1}{2})}(N,\psi\backslash G/N,\psi)$, where $J_\pi$ are the relative characters for the Kuznetsov formula (see \S \ref{ssrelchars-Kuz}), and $\gamma(\pi, \Sym^2, \frac{1}{2}, \psi)$ is the local gamma factor for the functional equation of the symmetric-square $L$-function (see \S \ref{ssHankelrelchars}).

Finally, the space $\mathcal D^-_{L(\Sym^2, \frac{1}{2})} (N,\psi\backslash G /N,\psi)$ contains the ``basic vector'' $f_{L(\Sym^2, \frac{1}{2})}$ attached to $L(\Sym^2,\frac{1}{2})$ (defined in \S \ref{ssnonstandard}). The Hankel transform $\mathcal H_{\Sym^2}$ maps $h\cdot f_{L(\Sym^2, \frac{1}{2})}$ to $h\cdot f_{L((\Sym^2)^\vee, \frac{1}{2})}$, for any element $h$ of the unramified Hecke algebra.
\end{theorem}

The proof of this theorem will occupy this section. The statement will be obtained by the ``unfolding'' method from the Hankel transform \eqref{RSHankel} for the Rankin--Selberg variety. In the process, we will also prove the outstanding statements of Theorem \ref{thmRudnick} for the comparison between the group and Kuznetsov trace formulas for $\SL_2$.

\subsection{Unfolding} \label{ssunfolding}

Let $V$ a two-dimensional symplectic vector space, and $X$, $\bar X$, $A$, $A_\ad$ as in the previous section. I remind that the group acting on the Rankin--Selberg variety $\bar X$ is $\tilde G = (A\times \SL(V)^2)/\{\pm 1\}$. The group $G = \Gm \times \SL_2$ will, more canonically, be identified with the group $A_\ad\times\SL(V)$, and as the quotient of the group $\tilde G$ (as varieties), with the quotient map induced by
\begin{equation}\label{tildetoG} (a,g_1,g_2) \mapsto (a_\ad , g_1 g_2^{-1}),
\end{equation}
where $a_\ad$ is the image of $a$ through $A\to A_\ad$.

The name of \emph{unfolding} is given to the method which, globally, proves that the Rankin--Selberg integral of two cusp forms is equal to an Euler product, by relating it with the Whittaker/Fourier coefficients of the cusp forms. Locally, it translates to an explicit $L^2$-isometry:
$$\U: L^2(X) \xrightarrow\sim L^2(\tilde N,\tilde\psi\backslash \tilde G),$$
where $\tilde N = N\times N$ is a maximal unipotent subgroup of $\tilde G$, endowed with the non-degenerate character $\tilde\psi = \psi \times \psi^{-1}$. The goal of this subsection is to compute the descent of the unfolding map to coinvariants, giving rise to a morphism 
\[\bar{\U}: \mathcal S(\bar X/\SL_2) \to \mathcal S^{--}(N,\psi\backslash G/N,\psi)\]
(the image being a non-standard space of test measures for the Kuznetsov formula), see Proposition \ref{prop:unfoldingdescent}.

The map $\U$ is described as follows:

Recall that $\bar X\simeq V\times \SL(V)$ comes with a map $\pi:(v,g)\mapsto (v, vg)$ to $V\times V$. This can be used to identify $\bar X$ with the space classifying triples $(v, w, g)$, where $v, w\in V$ and $g\in \SL(V)$ with $vg=w$. The open subset $X = V^*\times \SL(V)$ is a $\Ga$-torsor over its image $V^*\times V^*$, as follows: As we have seen in \S \ref{ssinvariant}, the symplectic structure on $V$ gives rise to a trivialization of the inertia group scheme over $V^*$:
\begin{equation}\label{trivS}
 S:= \{(g,v) \in \SL(V)\times V^* | vg = v\} \simeq \Ga \times V^*.
\end{equation}
We recall that we fixed this isomorphism so that $\Ga \simeq \SL(V)_v$ acts on the set of vectors $w$ with $\omega(v, w)=1$ by $x\cdot w = w - xv$. 
This group scheme acts on $X$ over $V^*\times V^*$, turning it into a $\Ga$-bundle; our convention will be that $S$ is identified with the inertia group scheme of the \emph{first} copy of $V^*$. 
 
We let $S^\vee$ be the dual vector bundle over $V^*$; of course, this is again isomorphic to $\Ga\times V^*$, but it will be good to distinguish one from its dual. We let $\bar Z = S^\vee \times V^*$ --- it is canonically isomorphic to $\Ga \times V^*\times V^*$. Finally, we set 
\begin{equation}\label{barZdef}  \overline{\tilde Z} = X\times_{V^*\times V^*} \bar Z \simeq  X \times_{V^*} S^\vee \simeq \Ga \times X
\end{equation}
for the base change of this vector bundle to $X$ (where the second fiber product is over the first copy of $V^*$). The complements of the zero sections of $\bar Z$, $\overline{\tilde Z}$ will be denoted by $Z$, resp.\ $\tilde Z$.

Thus, we have a Cartesian diagram, where the labels on the arrows denote the group or group scheme for which they are torsors:
$$ \xymatrix{ & \tilde Z \ar[dl]_{\Gm} \ar[dr]^{\Ga} \\ 
X\ar[dr]_{\Ga} && Z\ar[dl]^{\Gm} \\
&V^*\times V^*. }$$

To reformulate, $Z$ classifies $\SL(V)$-equivariant isomorphisms $S\xrightarrow\sim \Ga\times V^*$, together with points on $V^*\times V^*$, and it has a canonical section 
$$ V^*\times V^*\to Z$$
corresponding to the canonical trivialization \eqref{trivS}.

Group-theoretically, the above diagram corresponds to  the diagram 
\begin{equation}\label{unfoldingquotients} \xymatrix{ & \tilde Z = N^\diag\backslash \tilde G \ar[dl]\ar[dr]& \\ X = A_\ad^\diag N^\diag\backslash \tilde G \ar[dr] & & Z=\tilde N\backslash \tilde G \ar[dl]\\
 & V^*\times V^* = A_\ad^\diag \tilde N \backslash \tilde G.
 }\end{equation}
 Here, $A_\ad^\diag$ denotes the diagonal embedding of $A_\ad$ in all three factors of the product $\tilde N\backslash \tilde G \simeq A\times N\backslash \SL(V) \times N\backslash \SL(V) / \{\pm 1\}^\diag$.
  Recall our convention that the universal Cartan $A=B/N$ acts on $N\backslash \SL(V)$ via a twist of the natural action by the non-trivial Weyl group element, which is why the stabilizer of a point on $X$ over the diagonal of $V^*$ is, indeed, $A_\ad^\diag N^\diag$. The identification of $\tilde Z\to X$ as a $\Gm$-torsor corresponds to the identification $A_\ad\xrightarrow\sim \Gm$ via the positive root character. Similarly, we have an isomorphism
  \begin{equation}\label{Zasprod}Z \simeq A_\ad \times (V^*\times V^*),\end{equation}
 when we identify $A_\ad$ with $\Gm$ via the positive root character.

 \begin{remark}\label{AonZ} To avoid any confusion, I stress that the ``factor'' $A$ of $\tilde G$ acts \emph{diagonally} on this product (namely, via its canonical quotient on $A_\ad$, and via the positive half-root character on $(V^*\times V^*)$).
 \end{remark}

 For an element $(t, g_1, g_2)$ in $A\times N\backslash \SL_2 \times N\backslash \SL_2 $, we will be denoting by $[t, g_1, g_2]$ its image in $Z$.
  The group $A$ acts ``on the left'' on $N\backslash \SL_2$ by the same convention, hence $A^3/\{\pm 1\}^\diag$ acts on $Z$. We let $A_\ad^\adiag$ denote the \emph{anti-diagonal} embedding $a\mapsto (a, a^{-1}, a^{-1})$ of $A_\ad$ into $A^3/\{\pm 1\}$, where the first factor corresponds to the ``factor'' $A$ of $\tilde G$; \emph{we identify $A_\ad^\adiag$ with $A_\ad$ through this first factor}.
 The embedding $\bar Z$ is then the space 
 \begin{equation}\label{barZ} \bar Z= \Ga \times^{A_\ad^\adiag} Z,
 \end{equation}
 where $A_\ad$ acts on $\Ga$ via the \emph{positive} root character. Explicitly, by taking $N$ to be the upper triangular unipotent subgroup of $\SL_2$, this is the quotient of $\Ga\times N^2\times \SL_2^2$ by the equivalence relation
 $$ (x a, [t, Ng_1, Ng_2]) \sim (x, [ t\cdot \begin{pmatrix} a^{\frac{1}{2}} \\ & a^{-\frac{1}{2}}\end{pmatrix} , N \begin{pmatrix} a^{\frac{1}{2}} \\ & a^{-\frac{1}{2}}\end{pmatrix} g_1, N \begin{pmatrix} a^{\frac{1}{2}} \\ & a^{-\frac{1}{2}}\end{pmatrix} g_2 ]),$$
 where $\begin{pmatrix} a^{\frac{1}{2}} \\ & a^{-\frac{1}{2}}\end{pmatrix}$ is written symbolically for an element of $A_\ad = A/\{\pm 1\}$. Again, by our conventions for the action of $A$ on $N\backslash \SL(V)$, the quotient by the action of $A_\ad^\adiag$ in \eqref{barZ} translates to the embedding of $A_\ad$ in the stabilizer being the \emph{diagonal} one.  
  Notice also that the added orbit $A_\ad^\diag\backslash Z$ corresponding to $x=0$ lies at the ``funnel'' --- that is: at the infinity of the affine closure of $Z$. 

 Having fixed the additive character $\psi$, and its self-dual measure on $F$, there is a natural notion of Fourier transform of functions on $X$ along fibers of the $\Ga$-bundle $X\to V^*\times V^*$, with image on a certain space of functions on ${\tilde Z}$. The reader can consult \cite[\S 9.5]{SV} for a more general discussion. The issue here is that the fibers of $X$ over $V^*\times V^*$ do not have a canonical base point, in order to identify them with vector spaces. Here is where the space $\overline{\tilde Z}$ comes to play a role, because it fixes not only a linear functional on the structure group of $X$, but also a base point on it. Thus, a point of $\overline{\tilde Z}$ can be represented as a pair $(x,\ell)$, where $x\in X$ and $\ell$ is a linear functional on the fiber $S_x$ of the structure group over $x$. If, moreover, the pair belongs to $\tilde Z$ (i.e., $\ell\ne 0$), then $\ell$ is an isomorphism: 
 $$ \ell: S_x\xrightarrow\sim \Ga,$$
 so a point of $\tilde Z$ gives both a base point on $X$ and a trivialization of the structure group (not necessarily the same as the canonical one \eqref{trivS}). 
 
 We use these data to define Fourier transform: For a function $\Phi$ on $X$, we define a function on $\tilde Z$ by
 \begin{equation}\label{unfolding-formula} \mathcal U\Phi(x, \ell) = \int_{\Ga} \Phi(x+ \ell^{-1} z) \psi^{-1}(z) dz.
 \end{equation}
 
 Whenever the above integral is convergent, the function $\mathcal U\Phi$ has the following properties:
 \begin{itemize}
  \item It is $(\Ga,\psi)$-equivariant on the $\Ga$-torsor $\tilde Z\to Z$.
  \item Along the $\Gm$-torsor $\tilde Z\to X$, we have
  $$ \lim_{a\to 0} |a|^{-1} \mathcal U\Phi(x, a\cdot \ell) = \int_{\Ga} \Phi(\ell^{-1}(z)) dz,$$
  the zeroth Fourier coefficient of $\Phi$ along $\pi^{-1}(\pi(x))$ with the measure defined by $\ell^{-1}$, assuming that the Fourier transform of $\Phi$ along this $\Ga$-torsor is continuous at zero.
 \end{itemize}

 Thus, we can think of $\mathcal U\Phi$ as a section of a certain complex line bundle $\mathbb L_1$ over $\bar Z$. In general, we will define, for every complex number $s$, a line bundle $\mathbb L_s$ as the tensor product of two line bundles on $\bar Z$:
 $$ \mathbb L_s = \mathcal L_\psi \otimes \mathscr L_{\delta^s}.$$
 The former, $\mathcal L_\psi$, restricts to the Whittaker line bundle on $Z$, whose sections are $(\Ga,\psi)$-equivariant functions on $\tilde Z$. This can be extended to a line bundle over $\bar Z$, whose restriction to $\bar Z\smallsetminus Z$ is the trivial line bundle; its sections are functions on $\overline{\tilde Z}$ which satisfy $\Phi(x,\ell) = \psi(\left<z,\ell\right>) \Phi(x',\ell)$ for any quadruple $(x,x',z, \ell)\in X^2\times S \times \bar Z$ over the same point of $V^*\times V^*$ such that $x=z\cdot x'$. The other line bundle, $\mathscr L_{\delta^s}$, is the line bundle over $\bar Z$ whose sections are smooth functions on $Z$ which in any small neighborhood of the boundary are of the form $|\epsilon|^s\cdot \Phi$, where $\Phi$ is a smooth function on $\bar Z$ and $\epsilon$ is a local coordinate compatible with the $\Gm$-action. The notation $\delta^s$ is justified by the isomorphism \eqref{barZ}.
 
 Similarly, we can define Fourier transform for half-densities and measures. This will land in half-densities on $\bar Z$ valued in $\mathbb L_0$, and measures on $\bar Z$ valued in $\mathbb L_{-1}$, respectively. Restricting to Schwartz functions, densities or measures on $X$, Fourier theory on the line implies:
 
 \begin{proposition} \label{unfoldingprop}
 The unfolding map defines isomorphisms
 \begin{align*}  \mathcal U: \mathcal F(X) &\xrightarrow\sim \mathcal F(\bar Z, \mathbb L_1),\\
   \mathcal U: \mathcal D(X) &\xrightarrow\sim \mathcal D(\bar Z, \mathbb L_0),\\
   \mathcal U: \mathcal S(X) &\xrightarrow\sim \mathcal S(\bar Z, \mathbb L_{-1}).
  \end{align*}
 \end{proposition}

 For these to be equivariant with respect to the $\tilde G$-action, we need to be careful about normalizations, since the action of $A$ on $\mathcal F(X)$ and $\mathcal S(X)$ has been normalized in \eqref{actionnormagain-functions}, \eqref{actionnormagain-measures}. We therefore we need to let $A$ act on sections of $\mathbb L_s$ over $\bar Z$ by 
 \begin{equation}\label{actionnorm-unfolding-fns} a\cdot \Phi(x) = \delta(a)^{\frac{1}{2}}\Phi(a\cdot x),
\end{equation}
 and on measures valued in $\mathbb L_s$ by 
 \begin{equation}\label{actionnorm-unfolding-meas} a\cdot \mu(x) = \delta(a)^{-\frac{1}{2}}\mu(a\cdot x).
\end{equation}
\emph{In other words, we twist the usual action of $\tilde G$ on its Whittaker functions or measures by the characters $\delta^{ \frac{1}{2}}$.} As usual, no normalization is needed for half-densities. In terms of this normalized action, the group $A_\ad^\adiag$ acts on the fibers of the sheaf of sections of $\mathscr L_{\delta^s}$ over $\bar Z\smallsetminus Z$ by the character $\delta^{s+\frac{1}{2}}$, and on the fibers of the sheaf of measures valued in $\mathscr L_{\delta^s}$ by the same character. There are now equivariant isomorphisms, depending on the choice of a measure:
\begin{equation}\label{fnsdensmeas}\mathcal F(\bar Z, \mathbb L_1) \to \mathcal D(\bar Z, \mathbb L_0) \to \mathcal S(\bar Z, \mathbb L_{-1}),
\end{equation}
that are obtained as follows: Choose a $\tilde G$-invariant measure $d^\times z$ on $Z$. (The notation $d^\times z$ will be explained below.) Notice that a $\tilde G$-invariant measure $d^\times z$ on $Z $ has a \emph{triple} pole at $\bar Z\smallsetminus Z$, i.e., in a local coordinate $\epsilon$ is is a multiple of $|\epsilon|^{-3}$ by a smooth measure; indeed, under the unnormalized action $A_\ad^\adiag$ acts on $d^\times z$ by the character $\delta^{-2}$, and on the fibers over $\bar Z\smallsetminus Z$ of the bundle of smooth measures by $\delta$ --- thus, the invariant measure is of the form $|\epsilon|^{-3}$ times a smooth measure on $\bar Z$. The character $\delta$ can be understood as a function on $Z$ via \eqref{Zasprod}. Now define the maps \eqref{fnsdensmeas} by multiplying by $(\delta \cdot d^\times z)^\frac{1}{2}$. It is immediately seen that these maps are equivariant for the normalized action.

Let us be a bit more careful about choices of measures. Suppose an $\SL(V)^2$-invariant measure chosen on $V^*\times V^*$ --- for example, the one determined by the symplectic form. Any identification of the group scheme $S$ with $\Ga\times V^*\times V^*$ induces, by our fixed measure on $\Ga$,  measures $dx$ on $X$ and $dz$ on $Z$ which are dual with respect to Fourier transform. These measures are $\SL(V)^2$-equivariant but vary by the character $\delta$ of the quotient $\tilde G\to A_\ad$; this is the meaning of the notation $d^\times z$ above, since by $dz$ we will denote a $(\tilde G,\delta)$-equivariant measure on $Z$. In any case, a choice of such dual measures induces a commutative diagram
\begin{equation}\label{unfoldingfunctionsmeasures}\xymatrix{ \mathcal F(X)  \ar[d]_{\cdot dx} \ar[r]^{\mathcal U} & \mathcal F(\bar Z, \mathbb L_1) \ar[d]^{\cdot dz} \\
\mathcal S(X) \ar[r]^{\mathcal U} & \mathcal S(\bar Z, \mathbb L_{-1})}
\end{equation}
(and similarly for half-densities).

We extend the unfolding morphism, by the same formula, to the space of Schwartz functions on $\bar X$. 

\begin{lemma}\label{unfoldingextends}
 The unfolding map \eqref{unfolding-formula} converges absolutely on $\mathcal F(\bar X)$. 
\end{lemma}

\begin{proof}
 The map is given by an integral over a unipotent orbit; this orbit is closed in the affine space $\bar X$, where elements of $\mathcal F(\bar X)$ are of rapid decay. Thus, the integral is convergent.
\end{proof}

We extend it similarly to half-densities or measures to consider it as map from $\mathcal D(\bar X)$, resp.\ $\mathcal S(\bar X)$; their images are spaces $\mathcal D^{--}(Z,\mathcal L_{\psi})$, $\mathcal S^{--}(Z,\mathcal L_{\psi})$ of Whittaker half-densities and measures for $\tilde G$ which contain $\mathcal D(\bar Z, \mathbb L_0)$, resp.\ $\mathcal S(\bar Z, \mathbb L_{-1})$.

Now fix a point $v\in V^*$, a preimage $z = (v, v, \ell)$ of $(v,v)$ in $Z$, let $N$ be the stabilizer of $v$ in $\SL(V)$, and use $\ell$ to identify $N\xrightarrow\sim \Ga$. 
This identifies $Z$ with $N^2\backslash \tilde G$, and the map \eqref{tildetoG} identifies $Z/\SL(V)^\diag$ with the quotient $N\backslash G/N$. Twisted push-forward as in \S \ref{sstwistedpf} gives a map 
$$ \mathcal S(Z,\mathcal L_{\psi}) \to \mathcal S(N,\psi\backslash G/N,\psi).$$

Now we fix coordinates: We recall that on the two-dimensional affine space $\bar X\sslash \SL_2$ we have fixed coordinates $(c,t)$ in \S \ref{ssRSthespace}. On the other hand, we will identify $N\simeq \Ga$ with the upper triangular subgroup of $\SL_2$, and the space of $\mathcal S^{--}(N,\psi\backslash G/N,\psi)$ with a space of scalar-valued measures on $N\backslash G\sslash N$, by choosing the section described in \S \ref{sstwistedpf} (or, more invariantly, at the end of \S \ref{ssinvariant}). Here our group is $G=A_\ad \times \SL(V)$, so we have coordinates $(a,\zeta)$, where $\zeta$ is as in \S \ref{sstwistedpf} and $a\in \Gm$ is the positive root character applied to the factor $A_\ad$.

Consider a diagram
 $$\xymatrix{
 \mathcal S(\bar X) \ar[r]^{\U} \ar[d] & \mathcal S^{--}(Z,\mathcal L_{\psi})\ar[d]\\
 \mathcal S(\bar X/\SL_2) \ar@{-->}[r]^{\bar{\U}} & \mathcal S^{--}(N,\psi\backslash G/N,\psi),
 }$$
where the last space is defined as the (twisted) push-forward of measures in $\mathcal S^{--}(Z,\mathcal L_{\psi})$.

\begin{proposition}\label{prop:unfoldingdescent}
The dotted arrow in the bottom making the above diagram commute exists, and is given by the absolutely convergent integral
\begin{equation}\label{unfoldingdescent} \bar{\U}(f) (a,\zeta) = \int_{\Gm} f(\zeta a , \zeta w) \psi^{-1}(w) d^\times w.\end{equation}

For the analogous diagrams for functions and half-densities we have:
\begin{equation}\bar{\U}(\Phi) (a,\zeta) = \int_{\Gm} \Phi(\zeta a , \zeta w) \psi^{-1}(w) d w.\end{equation}
for functions and 
\begin{equation}  \bar{\U}(\varphi) (a,\zeta) = \int_{\Gm} \varphi(\zeta a , \zeta w) \psi^{-1}(w) |w|^\frac{1}{2} d^\times w.\end{equation}
for half-densities.
\end{proposition}

In other words, if we identify the coordinates on the two sides by 
\begin{equation}\label{coordinatechange} \begin{cases} 
    t = \zeta \\
    c = \zeta a
   \end{cases} 
   \Leftrightarrow
   \begin{cases}
    \zeta = t \\
    a =  c t^{-1},
   \end{cases}
\end{equation}
and set $\tilde f(a,\zeta) = f(\zeta a, \zeta)$, we have, for measures, 
\begin{equation}\label{Ubar-measures} \bar{\U}(f) (a,\zeta) =   \int_{\Gm} \tilde f( a w^{-1}, \zeta w) \psi^{-1}(w) d^\times w 
 =  \left( \psi^{-1}({\bullet}) d^\times \bullet \right)\star_{\check\lambda} \tilde f (a,\zeta),
\end{equation}
where $\star_{\check\lambda}$ denotes multiplicative convolution along the cocharacter $\check\lambda(x):= (a=x, \zeta=x^{-1})$ (without any normalization for the action of $\Gm$ on measures or functions). The inverse is given by 
$$\tilde f(a,\zeta) = \left( \psi({\bullet}) |\bullet| d^\times \bullet \right)\star_{-\check\lambda} \bar{\U} f(a,\zeta),$$
hence
\begin{equation} \label{barUinverse} \bar\U^{-1}h(c,t) = \int h\left(\frac{wc}{t}, \frac{t}{w}\right) \psi(w) dw.\end{equation}

For half-densities, with $\tilde \varphi(a,\zeta) = \varphi(\zeta a, \zeta)$,
we have 
\begin{equation}\label{Ubar-densities}
\bar{\U}(\varphi) (a,\zeta) = \left(\psi^{-1}(\bullet) |\bullet|^\frac{1}{2} d^\times\bullet\right) \star_{\check\lambda} \tilde\varphi (a,\zeta).
\end{equation}

\begin{proof}
 
 For what follows, for any element $y\in \Ga$ we will denote by $(y)_1$ the point on $X$ represented by the pair $(\begin{pmatrix} 1 & y \\ & 1\end{pmatrix}, 1) \in (\SL_2)^2$. We also denote by $\delta$ the modular character of the Borel, identified as a character of $G$ via projection to the $A_\ad$-quotient. 

 Let $\Phi dx \in \mathcal S(\bar X)$, so $\Phi \in \mathcal F(\bar X)$. Let $dz$ be the $(\tilde G,\delta)$-equivariant measure on $Z$ that is dual to $dx$, see the discussion before Proposition \ref{unfoldingfunctionsmeasures}. 
 
 Let $f$ be the push-forward of $\Phi dx$ to $\mathfrak C_X=\bar X\sslash \SL_2$. By the integration formula \eqref{integrationX}, if $O_{(c,t)}(\Phi)$ denotes the $\SL_2$-orbital integral of $\Phi$ at the point $(c,t)$ with $c\ne 0$, for a suitable choice of Haar measure on $\SL_2$, we have 
 $$f(c,t) = O_{(c,t)}(\Phi) dc dt = O_{(c,t)}(\Phi) |c| \cdot |t| \cdot d^\times c d^\times t .$$

 Similarly, the push-forward of $(\U\Phi) \cdot dz$ to $\C_Z = Z\sslash \SL_2 = A_\ad \times (N\backslash \SL_2\sslash N)$, is computed in terms of the orbital integrals of the function $\U\Phi$ as follows: 
 Recall that our trivialization of Kuznetsov orbital integrals used the subvariety of anti-diagonal elements, which here can be represented by the elements
 $$ (a , \begin{pmatrix} & -\zeta^{-1} \\ \zeta \end{pmatrix}_1)$$
 inside of $G$. The subscript $1$ indicates that the element belongs to the first copy of $\SL_2$. The push-forward of $\U\Phi dz$ will be, here,
 $$ \bar{\U} f(a,\zeta) := O_{(a,\zeta)}(\U\Phi)  |a| \cdot |\zeta|^2 d^\times a d^\times \zeta,$$
 where $d^\times a$ is a Haar measure on $A_\ad$; we identify the latter with $\Gm$ via the positive root character. 

 The unfolding formula \eqref{unfolding-formula}, in our present coordinates, reads
 \begin{equation}\label{unfolding-formula-explicit} \U\Phi ([t,g_1, g_2]) = \int_{\Ga} \Phi(N^\diag \cdot (t^{-1} \begin{pmatrix} 1 & y \\ &1\end{pmatrix} g_1, t^{-1} g_2)) \psi^{-1} (y) dy.\end{equation}
 
We compute the regular orbital integrals of $\U\Phi$, using this formula. Let $a'\in A$ with image $a\in A_\ad$, with $a$ identified as an element of $\Gm$ through the positive root character. We have
 \begin{multline*} O_{(a,\zeta)}(\U\Phi) = \int_{\SL_2}  \int_{\Ga} \Phi \left(a'^{-1} \cdot (w)_1 \cdot \begin{pmatrix} & -\zeta^{-1} \\ \zeta \end{pmatrix}_1 g\right) \psi^{-1}(w) dw dg 
 \\ =  \int_{\SL_2} \int_{\Ga} \Phi\left(\begin{pmatrix} 1 & a^{-1} w \\ & 1 \end{pmatrix}_1 \begin{pmatrix} & -\zeta^{-1} a^{-1}\\ \zeta a \end{pmatrix}_1 g \right) \psi^{-1}(w) dw dg.\end{multline*}

 This clearly remains true even if $a$ is not in the image of $A(F)\to A_\ad(F)$.
 
 I claim that the above integral is absolutely convergent as a double integral. Indeed, if we let the group $N\times \SL_2$ act on $X= N^\diag\backslash \SL_2^2$, with $N$ acting as $N_1$, i.e.\ the unipotent subgroup in the first $\SL_2$-factor, and $\SL_2$ acting diagonally, then this action does not extend to $\bar X$, but the orbits represented by $\begin{pmatrix} & -\zeta^{-1} \\ \zeta \end{pmatrix}$ with $\zeta\ne 0$ are closed in $\bar X$: indeed, in the quotient $\bar X\sslash \SL_2\simeq\mathbbm A^2_{(c,t)}$ they live over fixed, non-zero values for $c$, while $\bar X\smallsetminus X$ lives over $c=0$. Thus, the double integral is absolutely convergent.
 
 Thus, we can interchange the order of integration, and then, with the identification $A_\ad \ni a \overset{\sim}{\mapsto} \delta(a) = \delta(a') \in \Gm$, this reads
 $$ O_{(a,\zeta)}(\U\Phi) = \int_{\Ga} O_{(\zeta a, \zeta w)} (\Phi) \psi^{-1}(w) dw.$$
 
 Replacing functions by measures, we see that the push-forward measures satisfy:
 \begin{align*}\bar{\U} f(a,\zeta)&= O_{(a,\zeta)}(\U\Phi) |a| |\zeta|^2 d^\times a d^\times \zeta \\
 &=  \left(\int_{\Ga} O_{(\zeta a, \zeta w)} (\Phi) \psi^{-1}(w) dw\right) |a| \cdot |\zeta|^2 d^\times a d^\times \zeta \\
  &= \left(\int_{\Ga} |\zeta a| \cdot |\zeta w| O_{(\zeta a, \zeta w)} (\Phi) |w|^{-1} \psi^{-1}(w) dw\right)  d^\times a d^\times \zeta \\
  &= \int_{\Gm} f(\zeta a , \zeta w) \psi^{-1}(w) d^\times w.
\end{align*}
 \end{proof}

\subsection{Hankel transform} \label{ssHankel}

 Now recall the subspace $\mathcal S(\bar X/\SL_2)^\circ$ of \S \ref{ssSym2}, and the corresponding subspace of half-densities, $\mathcal D(\bar X/\SL_2)^\circ$. We \emph{define} the space $\mathcal D^-_{L(\Sym^2, \frac{1}{2})}(N,\psi\backslash G/N,\psi)$ (resp.\ $\mathcal S^-_{L(\Sym^2, 1)}(N,\psi\backslash G/N,\psi)$) to be the image of the subspace $\mathcal D(\bar X/\SL_2)^\circ$ (resp.\ $\mathcal S(\bar X/\SL_2)^\circ$) under $\bar{\U}$: 
 $$ \bar{\U}: \mathcal D(\bar X/\SL_2)^\circ \xrightarrow\sim \mathcal D^-_{L(\Sym^2, \frac{1}{2})}(N,\psi\backslash G/N,\psi)$$
 (and similarly for $\mathcal S^-_{L(\Sym^2, 1)}(N,\psi\backslash G/N,\psi)$).

 \begin{remark}\label{denstomeas}
 The passage from half-densities to measures on $Z$ involves multiplication by the half-density $(\delta(z)\cdot d^\times z)^\frac{1}{2}$, see the discussion after Proposition \ref{unfoldingprop} on $Z$. On the quotient $Z/\SL_2 = N\backslash G/N$, this corresponds to multiplication by $(\delta(a)\cdot d^\times a)^\frac{1}{2} \cdot (\delta(\zeta)\cdot d^\times \zeta)^\frac{1}{2} = (|a| d^\times a)^\frac{1}{2} \cdot (|\zeta|^2\cdot d^\times \zeta)^\frac{1}{2} $, where $(a,\zeta)$ are coordinates as above, identified also with elements of the universal cartan $A_G=A_\ad\times A$. The first factor is responsible for the fact that we use the notation $\mathcal D^-_{L(\Sym^2, \frac{1}{2})}$ for half-densities, but $\mathcal S^-_{L(\Sym^2, 1)}$ for measures --- it has to do with the $L$-function that one will obtain after pairing with relative characters for the Kuznetsov formula. Of course, if we descend the normalized action \eqref{actionnorm-unfolding-meas} to the space $Z/\SL_2$, this map from half-densities to measures is equivariant.
 \end{remark}
 
 We define 
 $$\mathcal D^-_{L((\Sym^2)^\vee, \frac{1}{2})}(N,\psi\backslash G/N,\psi) = j^* \mathcal D^-_{L(\Sym^2, \frac{1}{2})}(N,\psi\backslash G/N,\psi),$$ where $j$ is the involution on $N\backslash G\sslash N$ induced by the inversion map on the $A_\ad$-factor of $G$.  
 We also define an unfolding map 
 $$ \bar{\U}^\vee: \mathcal D(\bar X/\SL_2)^\circ \xrightarrow\sim \mathcal D^-_{L((\Sym^2)^\vee, \frac{1}{2})}(N,\psi\backslash G/N,\psi)$$
 by 
 $$ \bar{\U}^\vee = j^* \circ \bar\U.$$

 Now recall the Hankel transform $\mathcal H_X^\circ$ of Theorem \ref{thmsubspaceX}, which is an endomorphism of $\mathcal D(\bar X/\SL_2)^\circ$. The composition 
 $$ \bar{\U}^\vee \circ \mathcal H_X^\circ \circ \bar{\U}^{-1}$$
 is an isomorphism:
 $$ \mathcal H_{\Sym^2}: \mathcal D^-_{L(\Sym^2, \frac{1}{2})}(N,\psi\backslash G/N,\psi) \xrightarrow\sim \mathcal D^-_{L((\Sym^2)^\vee, \frac{1}{2})}(N,\psi\backslash G/N,\psi).$$ 

 Putting together the formulas already proved, we have
 \begin{proposition}
The operator $\mathcal H_{\Sym^2}$ is given by the formula
$$  \mathcal H_{\Sym^2} = \lambda(\eta_{\zeta^2-4},\psi)^{-1} \mathscr F_{-\check\lambda_+,\frac{1}{2}} \circ \delta_{1-4\zeta^{-2}} \circ \eta_{\zeta^2-4} \circ \mathscr F_{-\check\lambda_0,\frac{1}{2}} \circ \eta_{\zeta^2-4} \circ \mathscr F_{-\check\lambda_-,\frac{1}{2}},$$
where the Fourier convolutions are understood in the regularized sense of \S \ref{sssFourierconv}.  
 \end{proposition}
 
 This is the formula of Theorem \ref{thmSym2}.
 
 \begin{proof}
  
  By \eqref{Ubar-densities}, the map that sends $\bar\U\varphi \mapsto \tilde\varphi$ is given by the convolution 
  $$\left(\psi(\bullet) |\bullet|^\frac{1}{2} d^\times\bullet\right) \star_{-\check\lambda}.$$
  The cocharacter $-\check\lambda$ is equal to the cocharacter that we denoted by $-\check\lambda_-$. The half-density $\tilde\varphi$, here, is expressed in coordinates $a=ct^{-1}$ and $\zeta=t$ on a dense open subset of $\C_X = \bar X\sslash \SL_2$, so $\varphi(c,t) = \tilde\varphi(ct^{-1}, t)$. Applying the operator $\mathcal H_X^\circ$ to $\varphi$ we get, according to \eqref{HXcirc}, 
  \begin{align*}
    \mathcal H_X^\circ \varphi(c,t) &= \lambda(\eta_{t^2-4},\psi)^{-1}\left( (\psi(\frac{1}{\bullet}) \eta_{t^2-4}(\bullet)|\bullet|^{-\frac{1}{2}} d^\times\bullet) \star_c  \varphi\right)((4-t^2) c^{-1}, t)\\
 &= \lambda(\eta_{t^2-4},\psi)^{-1}\left( (\psi({\bullet}) \eta_{\zeta^2-4}(\bullet)|\bullet|^{\frac{1}{2}} d^\times\bullet) \star_{-\check\lambda_0}  \tilde\varphi\right)(\frac{4-t^2}{ct}, t),
  \end{align*}
  where $-\check\lambda_0$ is the cocharacter $x\mapsto (a = x^{-1}, \zeta = 1)$. 
  If we set $\widetilde{ \mathcal H_X^\circ\varphi} (a , \zeta) = \mathcal H_X^\circ\varphi(\zeta a, \zeta)$ then we have, again by \eqref{Ubar-densities},
  \begin{align*} \bar\U^\vee \mathcal H_X^\circ \varphi (a, \zeta) &= \bar\U^\vee \mathcal H_X^\circ \varphi (a^{-1}, \zeta) \\ & = \left(\psi^{-1}(\bullet) |\bullet|^\frac{1}{2} d^\times\bullet\right) \star_{\check\lambda} \widetilde{ \mathcal H_X^\circ\varphi} (a^{-1} , \zeta) \\
  &= \left(\psi(\bullet) |\bullet|^\frac{1}{2} d^\times\bullet\right) \star_{\check\lambda'} \tilde{\tilde\varphi}(a, \zeta),
  \end{align*}
  where we have set $\tilde{\tilde\varphi}(a,\zeta) = \widetilde{ \mathcal H_X^\circ\varphi} (-a^{-1},\zeta)$, and $\check\lambda'$ is the cocharacter $x\mapsto (a=x^{-1}, \zeta = x^{-1})$, which we can identify with $-\check\lambda_+$. By the above, we have
  \begin{align*} \tilde{\tilde\varphi}(a,\zeta) &= \widetilde{ \mathcal H_X^\circ\varphi} (-a^{-1},\zeta) = \mathcal H_X^\circ\varphi(-\zeta a^{-1}, \zeta)\\
  &=  \lambda(\eta_{\zeta^2-4},\psi)^{-1}\left( (\psi({\bullet}) \eta_{\zeta^2-4}(\bullet)|\bullet|^{\frac{1}{2}} d^\times\bullet) \star_{-\check\lambda_0}  \tilde\varphi\right)(\frac{\zeta^2-4}{\zeta^2} a , \zeta)\\
  &= \lambda(\eta_{\zeta^2-4},\psi)^{-1}\cdot \delta_{1-4\zeta^{-2}} \circ \eta_{\zeta^2-4} \circ \mathscr F_{-\check\lambda_0,\frac{1}{2}} \circ \eta_{\zeta^2-4}  \tilde\varphi(a,\zeta).
  \end{align*}

  Moreover,
  $$ \tilde\varphi (a, \zeta) = \left(\psi(\bullet) |\bullet|^\frac{1}{2} d^\times\bullet\right) \star_{-\check\lambda_-} (\bar\U\varphi),$$
  and the result follows.
\end{proof}

\subsection{Descent to $A_\ad$-coinvariants} \label{ssdescent}

Any character $\chi: A_\ad\to \CC^\times$ can be understood as a densely-defined function on $Z\sslash \SL(V) = N\backslash G\sslash N$ via the product \eqref{Zasprod}. We have twisted push-forwards
$$ p_\chi:\mathcal S^-_{L(\Sym^2, 1)} (N,\psi\backslash G/N,\psi) \to \Meas(N\backslash \SL(V)\sslash N),$$
defined by $$ f \mapsto \pi_!(\delta^{-\frac{1}{2}}\chi^{-1} f),$$
whenever this push-forward converges, where 
\[\pi: Z\sslash \SL(V)\to N\backslash \SL(V)\sslash N\] is the canonical quotient map. When $\chi=\delta^s$, we will denote $p_\chi$ by $p_s$. The factor $\delta^{-\frac{1}{2}}$ is used in order to make the twisted push-forward $(A_\ad,\chi)$-equivariant  under the normalized action descending from \eqref{actionnorm-unfolding-meas}. If we also let $A_\ad$ act on measures on $\bar X\sslash \SL_2$ by the normalization descending from \eqref{actionnormagain-measures}, the unfolding map $\bar{\U}: \mathcal S(\bar X/\SL_2)^\circ \to \mathcal S^-_{L(\Sym^2, 1)}(N,\psi\backslash G/N,\psi) $ is $A_\ad$-equivariant for the normalized actions on both sides.

Dividing by the appropriate half-density (see Remark \ref{denstomeas}), these are the $(\Gm,\chi)$-equivariant integrals of Theorem \ref{thmSym2}.

\begin{proposition}\label{propdescent}
The twisted push-forward $p_\chi$ converges when $|\chi|=\delta^\sigma$ with $\sigma \ll 0$, and extends to a rational  (in the non-Archimedean case) or meromorphic (in the Archimedean case) family, in the variable $\chi$, of maps
$$ p_\chi: \mathcal S^-_{L(\Sym^2, 1)} (N,\psi\backslash G/N,\psi) \to \mathcal S^-_{L(\Ad,\chi^{-1}\delta^{\frac{1}{2}}\circ e^\frac{\check\alpha}{2})} (N,\psi\backslash \SL_2/N,\psi).$$
These maps have at most simple poles at the poles of the local $L$-functions $L(\chi\eta,-\frac{\check\alpha}{2} , \frac{1}{2})$, where $\eta$ ranges over all quadratic characters, and are surjective away from these poles. 
\end{proposition}

Recall that the space $\mathcal S^-_{L(\Ad,\chi')}$ associated to the character $\chi' = \chi^{-1}\delta^{\frac{1}{2}}\circ e^\frac{\check\alpha}{2}$ was defined in \S \ref{ssnonstandard}. When $\chi = \delta^s$, this space is denoted, more simply, by $\mathcal S^-_{L(\Ad,\frac{1}{2}-s)} (N,\psi\backslash \SL_2/N,\psi)$. That space was defined explicitly in terms of its germs, and the difficulty in this proposition is to identify it with the twisted push-forward of the space $\mathcal S^-_{L(\Sym^2, 1)} (N,\psi\backslash G/N,\psi)$ which was defined indirectly as the image of the ``$\Sym^2$ space'' $\mathcal S(\bar X/\SL_2)^\circ$ under the unfolding map $\bar{\U}$. The proof is quite long, and we break it down to several intermediate statements.

\begin{proof}[Beginning of the proof: reduction to $\mathcal S(\bar Z,\mathbb L_{-1})$] 

This is an exercise in identifying the images of various subspaces of $\mathcal S(\bar X_\ad)$ under pushforward and unfolding maps. 

First, we recall the definition of the space $\mathcal S(\bar X/\SL_2)^\circ$ (\S \ref{ssSym2}) as the push-forward of a space $\mathcal S(\bar X_\ad)^\circ$, which in turn was defined in terms of its Mellin transform, as Paley--Wiener sections of a sheaf $\mathscr E_\ad^\circ$.

Another space of importance is the Schwartz space of the complement of the zero section in $\bar X_\ad$, $\mathcal S(X_\ad)$, which corresponds to Paley--Wiener sections of a bundle $\mathscr D_\ad$, see \eqref{SXad}. Under the push-forward map $\mathcal S(\bar X_\ad)\to \Meas(X\sslash \SL_2)$, the image of $\mathcal S(X_\ad)$ is equal to $\mathcal S(X/\SL_2)$ (because of the isomorphism $[X_\ad/\PGL_2]\simeq [X/\SL_2]$).
 
If we compare the sheaves $\mathscr D_\ad$ and $\mathscr E_\ad^\circ$ describing the Mellin transforms of those subspaces, we will see that they are equal away from the poles of $\prod_\eta L(\chi\eta,-\frac{\check\alpha}{2} , \frac{1}{2})$, with $\eta$ ranging over all quadratic characters. More precisely, sections of $\mathscr E_\ad^\circ$ are meromorphic sections of $\mathscr D_\ad$ with at most simple poles at the poles of $\prod_{\eta\ne 1} L(\chi\eta,-\frac{\check\alpha}{2} , \frac{1}{2})$, whose residues (for $\eta\ne 1$) and evaluations (for $\eta = 1$) at the poles of $\prod_\eta L(\chi\eta,-\frac{\check\alpha}{2} , \frac{1}{2})$ (including $\eta=1$) satisfy a certain condition of lying in the induction of a finite-dimensional representation. Both bundles are generated by their global Paley--Wiener sections; this is trivial for $\mathscr D_\ad$, and for $\mathscr E_\ad^\circ$ see Lemma \ref{globalsections}.

The $A_\ad$-equivariance of the unfolding map means that the proposition will follow \emph{a fortiori} if we replace the space $\mathcal S^-_{L(\Sym^2, 1)} (N,\psi\backslash G/N,\psi)$ by the image of $\mathcal S(X/\SL_2)$ under $\bar{\U}$, without, in the result, allowing poles at the poles of the local $L$-functions $L(\chi\eta,-\frac{\check\alpha}{2} , \frac{1}{2})$, with $\eta\ne 1$. That is, it suffices to prove (besides the convergence statement) that $p_\chi$ defines a meromorphic family of morphisms 
\[ \bar{\U}\left(\mathcal S(X/\SL_2) \right) \to  \mathcal S^-_{L(\Ad,\chi^{-1}\delta^{\frac{1}{2}}\circ e^\frac{\check\alpha}{2})} (N,\psi\backslash \SL_2/N,\psi),\]
with at most simple poles at the poles of $L(\chi,-\frac{\check\alpha}{2} , \frac{1}{2})$.

As we have seen, the image of $\mathcal S(X)$ under the unfolding map $\U$ is the space of measures $\mathcal S(\bar Z,\mathbb L_{-1})$, hence we are reduced to studying the image of $\mathcal S(\bar Z,\mathbb L_{-1})$ in the bottom-right entry of the commutative diagram
\begin{equation}\label{ZtoKuz} \xymatrix{ \mathcal S(\bar Z,\mathbb L_{-1}) \ar[r]^{\tilde p_\chi} \ar[d] & \Meas(A\backslash \bar Z, \mathcal L_\psi) \ar[d] \\
   \mathcal S^-_{L(\Sym^2, 1)} (N,\psi\backslash G/N,\psi) \ar[r]^{p_\chi} & \Meas(N,\psi\backslash \SL_2/N,\psi),}
\end{equation}
where 
$$ \tilde p_\chi(f) = \tilde\pi_!(\delta^{-\frac{1}{2}}\chi^{-1} f)$$
for the quotient map $\tilde\pi: \bar Z\to A\backslash \bar Z$. The vertical arrows are our standard twisted push-forwards for the Kuznetsov quotient, \S \ref{sstwistedpf}; thus, the symbol $\Meas(N,\psi\backslash \SL_2/N,\psi)$ really stands for a space of measures on $N\backslash \SL_2\sslash N$, after applying our trivialization of the Kuznetsov push-forwards as explained in \S \ref{sstwistedpf}. 

Notice that $\tilde p_\chi$ is absolutely convergent for every $\chi$, since the $A$-orbits on $\bar Z$ are closed. (The reader should not confuse this with the factor $A_\ad$ in the isomorphism $Z = A_\ad\times V^*\times V^*$, see Remark \ref{AonZ}.) Therefore, our goal is to show that the right vertical arrow is convergent when $|\chi|=\delta^\sigma$ with $\sigma \ll 0$, and can be continued to a meromorphic family of morphisms, with at most simple poles at the poles of $L(\chi,-\frac{\check\alpha}{2} , \frac{1}{2})$, and image (away from these poles) equal to the space 
$\mathcal S^-_{L(\Ad,\chi^{-1}\delta^{\frac{1}{2}}\circ e^\frac{\check\alpha}{2})} (N,\psi\backslash \SL_2/N,\psi).$
\end{proof}

\begin{proof}[Second step: the image of $\tilde p_\chi$]

Our next step is to identify the image of $\mathcal S(\bar Z,\mathbb L_{-1}) $ under $\tilde p_\chi$. 

By \eqref{barZ}, we have $\bar Z = \Ga \times^{A_\ad^\adiag} Z$. The action of $A^\adiag$ on $Z$ descends to $A\backslash Z$, and we have an isomorphism 
\begin{equation}\label{ZmodA} A\backslash \bar Z =  \Ga \times^{A_\ad^\adiag} A\backslash Z.
\end{equation}

Recall that the line bundle $\mathbb L_{-1}$ is defined as a tensor product of $\mathcal L_\psi\otimes \mathscr L_{\delta^{-1}}$, where $\mathcal L_\psi$ is the Whittaker line bundle (which extends to the trivial line bundle over $\bar Z\smallsetminus Z$), and $\mathscr L_{\delta^{-1}}$ denotes the line bundle whose sections are smooth functions on $Z$ of the form $|\epsilon(z)|^{-1} \Phi(z)$ close to $\bar Z\smallsetminus Z$, for a local coordinate $\epsilon$ of this divisor, where $\Phi$ is a smooth function on $\bar Z$. Clearly, both line bundles are pullbacks of line bundles on $A\backslash \bar Z$, which will be denoted by the same symbols.

The isomorphism $\bar Z = \Ga\times^{A_\ad^\adiag} Z$ implies that 
for the \emph{unnormalized} action of $A_\ad^\adiag$ on spaces of measures we have
$$ \mathcal S(\bar Z, \mathscr L_{\delta^{-1}}) = (|\bullet|^{-1}\mathcal S(\Ga)) \hat\otimes_{\mathcal S(A_\ad^\adiag)} \mathcal S(Z).$$
For the twisted push-forward $\tilde p_\chi$, we first multiply these measures by $\delta^{-\frac{1}{2}}\chi^{-1}$, before pushing forward to $A\backslash Z$; this is equivalent to replacing the factor $(|\bullet|^{-1}\mathcal S(\Ga))$ by $(\delta^{-\frac{3}{2}}\chi^{-1}\mathcal S(\Ga))$, where we again identify $A_\ad\simeq\Gm$ by the positive root character. It follows that, if $\tilde p_\chi$ 
was applied to the space $\mathcal S(\bar Z, \mathscr L_{\delta^{-1}})$, its image would be the space $\mathcal S(A\backslash \bar Z, \mathscr L_{\delta^{-\frac{3}{2}}\chi^{-1}})$, where the bundle $\mathscr L_{\chi}$ generalizes $\mathscr L_s$ in the obvious way, recovering $\mathscr L_s$ for $\chi=\delta^s$. 

Generalizing, similarly, the notation for $\mathbb L_s$ so that $\mathbb L_\chi = \mathcal L_\psi \otimes \mathscr L_\chi$, it follows that the image of $\mathcal S(\bar Z, \mathbb L_{-1}) $ under $\tilde p_\chi$ is the space
$$M_\chi:=\mathcal S(A\backslash \bar Z, \mathbb L_{\delta^{-\frac{3}{2}}\chi^{-s}}).$$
We are left with computing the image of $M_\chi$ under push-forward to $A\backslash \bar Z\sslash \SL(V)^\diag = N\backslash \SL(V)\sslash N$. 

\end{proof}

Recall that $N\backslash \SL(V)\sslash N$ is a one-dimensional affine space, and we have fixed a coordinate $\zeta$. We compactify it to $\mathbb P^1$, and then we have a rational map $\bar Z \to \mathbb P^1$. This map is defined away from the intersection of the divisor $\zeta =0$ with the divisor $\bar Z\smallsetminus Z$. The complement of this intersection is the union:
$$ Z \cup \bar Z^{\rm disj},$$
where the exponent ``disj'' denotes the locus over the set $\zeta\ne 0$, that is, over the set of pairs $(v,w)\in V^*\times V^*$ which are not colinear. We have a subspace
\begin{equation}\label{sequenceWs} \mathcal S(A \backslash Z,\mathcal L_\psi) + \mathcal S(A\backslash \bar Z^{\rm disj}, \mathbb L_{\delta^{-\frac{3}{2}}\chi^{-1}}) \hookrightarrow M_\chi,
\end{equation}
which we will denote by $M_\chi^0$.

The third step in the proof of Proposition \ref{propdescent} will be to study push-forwards of the elements in $M_\chi^0$.

\begin{proof}[Third step: the push-forward of $M_\chi^0$]

We will see that the right vertical arrow of \eqref{ZtoKuz} is absolutely convergent for every $\chi$ on the subspace $M_\chi^0$, and that its image is the space $\mathcal S^-_{L(\Ad,\chi^{-1}\delta^{\frac{1}{2}}\circ e^\frac{\check\alpha}{2})} (N,\psi\backslash \SL_2/N,\psi),$ 

The elements of $\mathcal S(A \backslash Z,\mathcal L_\psi)$ are usual Schwartz Whittaker measures, so their push-forward will be the space $\mathcal S(N,\psi\backslash \SL(V)/N,\psi)$. 

Over $A\backslash \bar Z^{\rm disj}$, we have an $\SL(V)=\SL(V)^\diag$-equivariant isomorphism: 
$$A\backslash \bar Z^{\rm disj} \simeq (\mathbb P^1\smallsetminus\{0\}) \times \SL(V)$$
(compatible with the map to $\mathbb P^1\supset N\backslash \SL(V)\sslash N$), and from this it is easy to see that 
the Whittaker line bundle $\mathcal L_\psi$ admits an $\SL(V)^\diag$-equivariant trivialization. We also claim: 
\begin{quote}
 The line bundle $\mathscr L_{\delta^{-\frac{3}{2}}\chi^{-1}}$, restricted to $\bar Z^{\rm disj}$, is the pullback of the line bundle $\mathfrak L_{\delta^{-\frac{3}{2}}\chi^{-1}}$ over $\mathbb P^1\smallsetminus\{0\}$, whose sections are smooth functions of the variable $\zeta^{-1}$, multiplied by the character $|\zeta|^{\frac{3}{2}} \chi(e^\frac{\check\alpha}{2}(\zeta))$. 
\end{quote}

To see this, use the isomorphism \eqref{ZmodA} to write $A\backslash \bar Z^{\rm disj}$ as $\Ga \times^{A_\ad^\adiag} (V^*\times V^*)^{\rm disj}/\{\pm 1\}$, where $(V^*\times V^*)^{\rm disj}$ refers to pairs of vectors that are not colinear. The group $A_\ad^\adiag\simeq A_\ad$ acts, here, by the positive root character on $\Ga$, and by the \emph{inverse} of the action of $A_\ad$, as we have defined it, on $(V^*\times V^*)^{\rm disj}$, that is, by the action which descends from the diagonal action of $A$ on the vector space $V\times V$ \emph{through the negative half-root character}. The group $\SL(V)^\diag$ acts freely on $(V^*\times V^*)^{\rm disj}$, and the quotient is $\Gm\subset \mathbb P^1$, with $A_\ad\simeq A_\ad^\adiag$ acting on it through the \emph{negative root character}. Thus, 
$$ A\backslash \bar Z^{\rm disj} /\SL(V) \simeq \Ga \times^{A_\ad^\adiag}\Gm \simeq  \mathbb P^1\smallsetminus\{0\},$$
with $A_\ad\simeq A_\ad^\adiag$ acting via the positive root character on $\Ga$ and via the negative root character on  $\mathbb P^1\smallsetminus\{0\}$. It is clear, now, from the definition of $\mathscr L_{\delta^{-\frac{3}{2}}\chi^{-1}}$ that it is pulled back from the line bundle on $\mathbb P^1\smallsetminus\{0\}$ whose sections, in the local coordinate $\zeta^{-1}$ at $\infty$, are of the form described in the claim above.

The map $\bar Z^{\rm disj}\to \mathbb P^1\smallsetminus\{0\}$ is smooth, therefore the image of $\mathcal S(A\backslash \bar Z^{\rm disj}, \mathbb L_{\delta^{-\frac{3}{2}}\chi^{-1}})$ is equal to $\mathcal S(\mathbb P^1\smallsetminus\{0\}, \mathcal L_{\delta^{-\frac{3}{2}}\chi^{-1}})$. 
Explicitly, these are smooth measures on $F^\times$ which are of rapid decay towards zero, and of the form 
$$C(\zeta^{-1}) |\zeta|^{\frac{3}{2}}\chi(e^\frac{\check\alpha}{2}(\zeta)) d(\zeta^{-1}) = C(\zeta^{-1}) |\zeta|^{\frac{1}{2}}\chi(e^\frac{\check\alpha}{2}(\zeta)) d^\times\zeta$$ 
close to infinity --- the asymptotic behavior of elements of the space that was denoted by $\mathcal S^-_{L(\Ad,\chi^{-1}\delta^{\frac{1}{2}}\circ e^\frac{\check\alpha}{2})} (N,\psi\backslash \SL_2/N,\psi)$ in \S \ref{ssnonstandard}.

We conclude that the image of the subspace $M_\chi^0$ of \eqref{sequenceWs} under push-forward to $N\backslash \SL_2\sslash N$ coincides with the space \[\mathcal S^-_{L(\Ad,\chi^{-1}\delta^{\frac{1}{2}}\circ e^\frac{\check\alpha}{2})} (N,\psi\backslash \SL_2/N,\psi),\] and the push-forward map is defined on this subspace for every $\chi$. It is easy to see from our proof that, as $\chi$ varies, the maps
\begin{equation}\label{fromMchizero}M_\chi^0\to \mathcal S^-_{L(\Ad,\chi^{-1}\delta^{\frac{1}{2}}\circ e^\frac{\check\alpha}{2})} (N,\psi\backslash \SL_2/N,\psi)\end{equation}
vary polynomially (in the obvious sense), in the non-Archimedean case, and holomorphically, in the Archimedean case. (One can even see that, in the Archimedean case, the entire sections that one obtains are of Paley--Wiener type, i.e., of moderate growth in bounded vertical strips.)
\end{proof}

Finally, let $\overline{M_\chi}$ denote the quotient $M_\chi/M_\chi^0$. We claim: 

\begin{equation}\label{barMchi}
\parbox{\dimexpr\linewidth-8em}{ For every $\chi$ that is not a pole of $L(\chi,-\frac{\check\alpha}{2} \frac{1}{2})$, the coinvariant space $(\overline{M_\chi})_{\SL(V)}$ is zero. For $\chi_0$ a pole of $L(\chi,-\frac{\check\alpha}{2} \frac{1}{2})$, for any $h\in \widehat{\mathcal S(A_\ad)}$ with Mellin transform $\check h(\chi)$ vanishing at $\chi_0^{-1}$, the coinvariant space $(\overline{M_{\chi_0}})_{\SL(V)}$ is annihilated by $h$.}
\end{equation}

This will imply that the entire family of maps \eqref{fromMchizero} extends meromorphically to the spaces $M_\chi$, with at most simple poles at the poles of $L(\chi,-\frac{\check\alpha}{2} , \frac{1}{2})$, completing the proof of Proposition \ref{propdescent}.

\begin{proof}[Last step: proof of \eqref{barMchi}]

Let $\bar Z^{\dagger}$ be the complement of $Z\cup \bar Z^{\rm disj}$. As we have seen in \eqref{fnsdensmeas}, dividing by the appropriate measure $dz$ on $Z$, elements of $\mathcal S(\bar Z,\mathbb L_{-1})$ become elements of $\mathcal F(\bar Z, \mathbb L_1)$, i.e., Schwartz sections of the line bundle $\mathbb L_1$, and this operation is equivariant for the normalized action of $A$. For a while, we will work with the space $\mathcal F(\bar Z, \mathbb L_1)$, in order to describe restrictions to those sections on $\bar Z^{\dagger}$ (or an infinitesimal neighborhood of it, in the Archimedean case). More precisely, we wish to describe the \emph{stalk} of $\mathcal F(\bar Z, \mathbb L_1)$ at $\bar Z^\dagger$ which, by definition, is the the quotient $\mathcal F(\bar Z,\mathbb L_1)/\mathcal F(\bar Z\smallsetminus \bar Z^\dagger, \mathbb L_1)$. In the non-Archimedean case, this coincides with the space of restrictions of elements of $\mathcal F(\bar Z,\mathbb L_1)$ to $\bar Z^\dagger$, while in the Archimedean case it is determined by the restrictions of all derivatives of the elements of $\mathcal F(\bar Z,\mathbb L_1)$ to this subset.

Remembering that 
$$\bar Z \simeq \Ga \times (V^*\times V^*),$$
the space $\bar Z^\dagger$ is the $G':=(A\times\SL(V)^\diag)/\{\pm 1\}$-invariant subset 
$$\{0\}\times \{(v,w)\in V^*\times V^*| v\mbox{ and } w \mbox{ are colinear}\}\simeq  \Gm\times V^*,$$ 
where both $A$ and $\SL(V)^\diag$ act trivially on the $\Gm$-factor, and by our usual conventions on $V^*$, so that the stabilizer of a point is the subgroup $B_\ad^\diag\subset G'$, where $B^\diag$ denotes the embedding $b\mapsto (a(b),b)$ of a Borel subgroup of $\SL(V)$ (with $a(b)$ the image of $b$ under the defining quotient $B\twoheadrightarrow A$), and $B_\ad=B/\{\pm 1\}$.

The stabilizer subgroup $B_\ad^\diag$ of a point acts (unnormalized action) by the character $\delta$ on the fiber of the complex line bundle $\mathbb L_1$ over that point, and by the positive root character on the fiber of the (two-dimensional over $F$) normal bundle of $Z^\dagger$ over that point. Let $\CC_1$ be the complex, $G'$-equivariant line bundle over $V^*$, where the stabilizer $B_\ad^\diag$ of some point acts by the character $\delta$, and $\CC_{n\alpha}$ the line bundle where it acts by the character $e^{n\alpha}$ (in the Archimedean case).

Thus, the stalk of $\mathcal F(\bar Z, \mathbb L_1)$ at $\bar Z^\dagger$  be identified, in the non-Archimedean case, with the space of Schwartz sections
$$\mathcal F (\Gm) \hat\otimes \mathcal F\left(V^*, \CC_1 \right),$$
and recall that the traslation action of $G'=(A\times \SL(V))/\{\pm 1\}$ on this space has been twisted, by our normalization \eqref{actionnorm-unfolding-fns}, by the character $\delta^\frac{1}{2}$. 

In the Archimedean case, the stalk has a separable decreasing filtration, indexed by $n\in \mathbb N$, by sections whose $(n-1)$-st derivatives vanish over $Z^\dagger$. The $n$-th graded quotient can be identified with Schwartz sections of $\mathbb L_1$ tensored by the $n$-th symmetric power of the conormal bundle (considered as an $\RR$-vector space), i.e., with
$$\mathcal F (\Gm)\hat\otimes \mathcal F\left(V^*, \CC_1\otimes_\RR \Sym_\RR^n(F_\alpha^2)\right),$$
where $F_\alpha^2$ stands for a two-dimensional $F$-vector space space with a \emph{scalar} action of the stabilizer $B_\ad^\diag$ (really, of $B_\ad^\diag$) by the positive root character. This space is isomorphic to 
\begin{equation}\label{gradedpiece}   \mathcal F (\Gm)\hat\otimes \mathcal F\left(V^*, \CC_1 \otimes \CC_{n\alpha}\right) \otimes_\RR \Sym_\RR^n(F^2),\end{equation}
now with trivial $G'$-action on all but the middle factor. Again, our definition of the action of $G'$ on sections over $V^*$ includes the twist by $\delta^\frac{1}{2}$, by the normalization \eqref{actionnorm-unfolding-fns}).

The map $\mathcal S(\bar Z, \mathbb L_{-1}) \simeq \mathcal F(\bar Z, \mathbb L_1) \to M_\chi$ descends to a map from the stalk of $\mathcal F(\bar Z, \mathbb L_1)$ over $\bar Z^\dagger$ to $\overline{M_\chi}$ and, passing to $\SL(V)^\diag$-coinvariants, we get a map:
$$ \left(\mathcal F(\bar Z,\mathbb L_1)/\mathcal F(\bar Z\smallsetminus \bar Z^\dagger, \mathbb L_1) \right)_{\SL(V)} \to \left(\overline{M_\chi}\right)_{\SL(V)}$$
which is $(A_\ad,\chi)$-equivariant with respect to the normalized action on the left. 

We analyze the corresponding coinvariant spaces of the graded pieces \eqref{gradedpiece} (including $n=0$, which includes the non-Archimedean case). The group $G'=(A\times \SL(V))/\{\pm 1\}$ only acts on the factor in the middle, which can be identified with 
$\mathcal F(V^*,\CC_1)\otimes \CC_{n\alpha}$, where now the whole group $G'$ acts on $\CC_{n\alpha}$ (in the non-Archimedean case) via the character $e^{n\alpha}$ of its quotient $A_\ad$. Under the \emph{unnormalized} actions of $G'$, the space $\mathcal F(V^*,\CC_1)$ is isomorphic to the space of Schwartz measures $\mathcal S(V^*)$, hence its $\SL(V)$-coinvariants are simply a complex line with trivial $A_\ad$-action. Under the \emph{normalized} action \eqref{actionnorm-unfolding-fns}, this means that 
$$ \mathcal F(V^*,\CC_1)_{\SL(V)}\simeq \CC_{\delta^\frac{1}{2}}$$
as an $A_\ad$-module. Hence, the $\SL(V)$-coinvariant space of \eqref{gradedpiece} is an $A_\ad$-eigenspace with eigencharacter $\chi=\delta^\frac{1}{2}\cdot e^{n\alpha}$. As $n$-varies, these are \emph{precisely} the poles of $L(\chi,-\frac{\check\alpha}{2}, \frac{1}{2})$. 

This prove \eqref{barMchi}, concluding the proof of Proposition \ref{propdescent}.
\end{proof}

As a special case, we can now prove Part \eqref{three} (hence also Part \eqref{one}) of Theorem \ref{thmRudnick}, which I recall here:

\begin{theorem}\label{Rudnickpt3}
  The equivariant Fourier transform $\mathcal T_{\SL_2}:=\mathscr F_{\Id,1}$ is an isomorphism 
 \begin{equation}
  \mathcal T_{\SL_2}: \mathcal S_{L(\Ad,1)}^-(N,\psi\backslash \SL_2/N,\psi) \xrightarrow\sim \mathcal S(\frac{\SL_2}{\SL_2}),
 \end{equation}
 when both sides are understood as measures on the affine line, with our usual coordinate $\zeta$ on the left hand side, and the trace coordinate $t$ on the right.
\end{theorem}

\label{Pfpart3}\label{proofRudnick}
\begin{proof}

Consider the diagram
\begin{equation}\label{quotdiagram} \xymatrix{ \mathcal S(\bar X/\SL(V))^\circ \ar[r]^{\bar\U} \ar[d] & \mathcal S^-_{L(\Sym^2, 1)} (N,\psi\backslash G/N,\psi)\ar[d]^{p_s} \\ 
\mathcal S(\bar X/\SL(V))^\circ_{(A_\ad,\delta^s)} \ar@{-->}[r] & 
\mathcal S^-_{L(\Ad,\frac{1}{2}-s)} (N,\psi\backslash \SL_2/N,\psi).
}\end{equation}

At the point $s= -\frac{1}{2}$, where $p_s$ is simply the push-forward from $N\backslash G\sslash N$ to $N\backslash \SL_2\sslash N$, we also have a push-forward map
$$ \mathcal S(\bar X/\SL(V))^\circ \to \mathcal S(\bar X/\SL(V))^\circ_{(A_\ad,\delta^s)} \to \Meas(A\bbslash \bar X\sslash \SL(V)) = \Meas\left(\Dfrac{\SL(V)}{\SL(V)}\right).$$

I claim that its image is the same as the image of $\mathcal S(X/\SL(V))$. Indeed, thinking of $\mathcal S(\bar X/\SL(V))^\circ$ as a quotient of the space $\mathcal S(\bar X_\ad)^\circ$ which, under Mellin transform, is identified with Paley--Wiener sections of the sheaf $\mathscr E_\ad^\circ$ over $\widehat{A_\ad}_\CC$ (see \S \ref{ssSym2}), the meromorphic family of maps to the spaces $\mathcal S^-_{L(\Ad,\frac{1}{2}-s)} (N,\psi\backslash \SL_2/N,\psi)$ corresponds to a meromorphic family of maps 
$$ \mathscr E_{\ad,s}^\circ \to \mathcal S^-_{L(\Ad,\frac{1}{2}-s)} (N,\psi\backslash \SL_2/N,\psi),$$
where $\mathscr E_{\ad,s}^\circ$ denotes the fiber of $\mathscr E_\ad^\circ$ over $\delta^s \in \widehat{A_\ad}_\CC$.

The image of $\mathcal S(\bar X/\SL(V))^\circ$ in $\Meas\left(\Dfrac{\SL(V)}{\SL(V)}\right)$ will be the image of $ \mathscr E_{\ad,-\frac{1}{2}}^\circ$ (since the sheaf is generated by its Paley--Wiener sections, Lemma \ref{globalsections}).
But $\delta^{-\frac{1}{2}}$ is not a pole of the $L$-function $L(\chi\omega, -\frac{\check\alpha}{2}, \frac{1}{2})$, for any quadratic character $\omega$, so, by definition, the bundle $\mathscr E_\ad^\circ$ coincides, around this character, with the bundle $\mathscr D_\ad$ describing the Mellin transform of elements of $\mathcal S(X_\ad)$. Thus, the image of $\mathcal S(\bar X/\SL(V))^\circ$ under push-forward to $A\bbslash \bar X\sslash \SL(V) = \Dfrac{\SL(V)}{\SL(V)}$ is the same as the image of $\mathcal S(X_\ad)$, which is also the same as the image of $\mathcal S(X)$ (by the isomorphism of stacks $[X/\SL_2] = [X_\ad/\PGL_2]$); that is, the image is the space $\mathcal S\left(\frac{\SL(V)}{\SL(V)}\right)$ of test measures for the stable trace formula of $\SL(V)$:
\begin{equation}\label{circtogroup}
 \mathcal S(\bar X/\SL(V))^\circ \twoheadrightarrow \mathcal S\left(\frac{\SL(V)}{\SL(V)}\right).
\end{equation}

The map $\bar\U$ is given, according to \eqref{Ubar-measures}, by the convolution operator $\left( \psi^{-1}({\bullet}) d^\times \bullet \right)\star_{\check\lambda}$; hence, its inverse will be the Fourier convolution $\mathscr F_{-\check\lambda, 1} = \left( \psi({\bullet}) |\bullet| d^\times \bullet \right)\star_{-\check\lambda}$. For $f\in \mathcal S^-_{L(\Sym^2, 1)} (N,\psi\backslash G/N,\psi)$, we can now compute the push-forward of $\bar\U^{-1} f$ under the surjection \eqref{circtogroup}, and it is immediately seen to factor through a map
$$ \mathcal S^-_{L(\Ad,1)} (N,\psi\backslash \SL_2/N,\psi) \twoheadrightarrow \mathcal S\left(\frac{\SL(V)}{\SL(V)}\right)$$
given by the same Fourier convolution $\mathscr F_{-\bar{\check\lambda}, 1}$, where $-\bar{\check\lambda}$ is the image of $-\check\lambda$ into the torus $A \subset \Aut(N\backslash \SL_2\sslash N)$, which coincides with the cocharacter $e^{\check\alpha}:a\mapsto \zeta=a$. Recall from \eqref{coordinatechange} that the output of this convolution operator, ``evaluated'' at a point $\zeta$, corresponds to the ``evaluation'' of a measure of $\mathcal S\left(\frac{\SL(V)}{\SL(V)}\right)$ at the point corresponding to trace $t=\zeta$.
\end{proof} 

The same argument gives us a meaningful statement about a transfer operator $\mathcal T_\chi$ for every character $\chi$ of $A_\ad$: Define twisted push-forward maps
$$ p'_\chi: \mathcal S(\frac{\SL_2}{N}) \to \Meas(\mathbbm A^1),$$
where $\mathbbm A^1 = \Dfrac{\SL_2}{B} = \Dfrac{\SL_2}{\SL_2}$ has coordinate $t=$the trace, by 
$$ p'_\chi (\varphi) (t) = t_!\left(\varphi(c,t) |c|^{-\frac{1}{2}}\chi^{-1}(c) \right),$$
where $c$ is the same coordinate on $\Dfrac{\SL_2}{N}$ as before, and $\chi$ is identified with a character of $\Gm$ through the positive half-coroot cocharacter of $A_\ad$.

The map $p'_\chi$ factors through $(B_\ad,\chi\delta^\frac{1}{2})$-coinvariants for the \emph{unnormalized} conjugation action of $B_\ad$ on $\mathcal S(\frac{\SL_2}{N})$, but of course the choice of base points with $c=1$ is important in realizing this coinvariant space as scalar-valued measures in the trace variable $t$. Denote the image of $p'_\chi$ by $\mathcal S\left(\frac{\SL_2}{B_\ad,\chi\delta^{\frac{1}{2}}}\right)$.

\begin{proposition}\label{descent-general}
 The operator
 $$ |\zeta|^{-\frac{1}{2}} \chi(e^{-\frac{\check\alpha}{2}}(\zeta)) \mathscr F_{\Id, \chi\circ e^{\frac{\check\alpha}{2}},\frac{3}{2}} $$
 defines an isomorphism 
 \begin{equation}\label{chicoinv} \mathcal S^-_{L(\Ad,\chi^{-1}\delta^{\frac{1}{2}}\circ e^\frac{\check\alpha}{2})} (N,\psi\backslash \SL_2/N,\psi) \xrightarrow\sim \mathcal S\left(\frac{\SL_2}{B_\ad,\chi\delta^{\frac{1}{2}}}\right)
 \end{equation}
 away from the poles of the local $L$-functions $L(\chi\eta, -\frac{\check\alpha}{2}, \frac{1}{2})$, where $\eta$ ranges over all quadratic characters. Here, $\Id$ denotes the identity cocharacter of the multiplicative group, acting on the one-dimensional space with coordinate $\zeta = t$.
\end{proposition}

\begin{proof}
By the same argument as before,  away from the poles of the local $L$-functions $L(\chi\eta, -\frac{\check\alpha}{2}, \frac{1}{2})$ we may replace $\mathcal S^-_{L(\Ad,\chi^{-1}\delta^{\frac{1}{2}}\circ e^\frac{\check\alpha}{2})} (N,\psi\backslash \SL_2/N,\psi)$ by the image of $\mathcal S(X/\SL_2)$ under the unfolding map $\bar{\U}$.

For notational simplicity, let us work with $\chi = \delta^s$, denoting $p'_\chi$ by $p'_s$ --- the general case is identical. The map \eqref{chicoinv}, in this case, is
 $$ \mathcal S^-_{L(\Ad,\frac{1}{2}-s)} (N,\psi\backslash \SL_2/N,\psi) \xrightarrow\sim \mathcal S\left(\frac{\SL_2}{B_\ad,\delta^{\frac{1}{2}+s}}\right).$$

Let $f\in \mathcal S^-_{L(\Sym^2, 1)} (N,\psi\backslash G/N,\psi)$, and set $\varphi(c,t) = \bar{\U}^{-1} f (c,t)$ and $\tilde \varphi(a, \zeta) = \varphi(\zeta a , \zeta)$. We will apply again the inverse of \eqref{Ubar-measures}, which states that 
$$\tilde\varphi(a, \zeta) = \mathscr F_{-\check\lambda, 1} f(a,\zeta) = \left( \psi({\bullet}) |\bullet| d^\times \bullet \right)\star_{-\check\lambda} f(a,\zeta).$$

We compute the push-forward of $\varphi(c,t) |c|^{-\frac{1}{2}-s} $ to the variable $t=\zeta$; in the calculations that follow, $c$ and $a$ are dummy variables that are being integrated over when we push forward to $t=\zeta$:
\begin{align*} t_!\left(\varphi(c,t) |c|^{-\frac{1}{2}-s} \right) & = \zeta_!\left(\tilde \varphi(a, \zeta) |\zeta a|^{-\frac{1}{2}-s}\right) \\
& = |\zeta|^{-\frac{1}{2}-s} \zeta_! \left( |a|^{-\frac{1}{2}-s} \mathscr F_{-\check\lambda, 1} f(a, \zeta)\right) \\
& = |\zeta|^{-\frac{1}{2}-s} \zeta_! \left( |a|^{-\frac{1}{2}-s} \int_{F^\times} f(a z, \zeta z^{-1}) \psi(z) |z| d^\times z \right) \\
& = |\zeta|^{-\frac{1}{2}-s} \zeta_! \left( \int_{F^\times} |az|^{-\frac{1}{2}-s} f(a z, \zeta z^{-1}) \psi(z) |z|^{\frac{3}{2}+s} d^\times z \right) \\
& = |\zeta|^{-\frac{1}{2}-s} \int_{F^\times} p_s f (\zeta z^{-1}) \psi(z) |z|^{\frac{3}{2}+s} d^\times z \\
& = |\zeta|^{-\frac{1}{2}-s} \mathscr F_{\Id, \frac{3}{2}+s} \left(p_s f\right).
  \end{align*}

\end{proof}
 
\subsection{Basic vector} 

Finally, we verify the statement of Theorem \ref{thmSym2} on the basic vector. Here $F$ will be a non-Archimedean field of residual degree $q$, and the symplectic space $V$ is defined over its ring of integers $\mathfrak o$, with the symplectic form integral and residually non-vanishing. As a result, all spaces $X, Z$ etc.\ are defined over $\mathfrak o$. All choices made in the previous sections should now be \emph{integral} and \emph{residually non-vanishing}; for example: the point on $V$ whose stabilizer we denoted by $N$, and the isomorphism $N\simeq \Ga$. We also assume that $F$ is unramified over $\mathbb Q_p$ or $\mathbb F_p((t))$, and recall that in this case we take the additive character $\psi$ to have conductor equal to $\mathfrak o$; the corresponding self-dual measure gives mass $1$ to $\mathfrak o$. 

Fix an $\SL_2^2$-invariant measure $dx$ on $\bar X$ and let $dz$ be its dual measure on $\bar Z$, as in \eqref{unfoldingfunctionsmeasures}. 

\begin{lemma}\label{lemmameasures}
 If the measure of $X(\mathfrak o)$ under $dx$ is $1$, then the measure of $\bar Z(\mathfrak o)$ under $dz$ is $1$.
\end{lemma}

\begin{proof}
 Indeed, $X$ and $\bar Z$ are fibered over $\SL(V)$, with the former being an affine bundle and the latter being the vector bundle dual to the structure group of the former. Dual measures on the fibers (with respect to the character $\psi$) assign the same mass to the fibers of $X(\mathfrak o)$ and of $\bar Z(\mathfrak o)$, hence the claim.
\end{proof}

In \S \ref{ssbasicvectorsX} we defined basic vectors $f_{\bar X}$ and $f_{\bar X}^\circ$ for the spaces $\mathcal S(\bar X/\SL_2)$ and $\mathcal S(\bar X/\SL_2)^\circ$; the former was the image of the measure $1_{\bar X(\mathfrak o)}dx$ with $dx(\bar X(\mathfrak o))=1$. 
On the other hand, on the space $\mathcal S^-_{L(\Sym^2,1)} (N,\psi\backslash G/N,\psi)$ we defined in \S \ref{ssnonstandard} a ``basic vector'' $f_{L(\Sym^2,1)}$, as the product of the generating Whittaker function of the $L$-function $L(\Sym^2,1)$ by a Haar measure on $Z$ with $Z(\mathfrak o)=1$.

\begin{theorem}\label{thmbasicvector}
We have
\begin{equation}\label{unfoldingbasic} \bar{\U} f^\circ_{\bar X} = (1-q^{-1})(1-q^{-2}) f_{L(\Sym^2,1)}. 
\end{equation}
\end{theorem}

\begin{proof}
The statement follows from \eqref{deff00} and the analogous statement about $f_{\bar X}$:
\begin{equation}\label{unfoldingbasicbarX} \bar{\U} f_{\bar X} = (1-q^{-1})(1-q^{-2}) f_{L(\Sym^2\oplus e^{\frac{\check\alpha}{2}},1)}. 
\end{equation}
Indeed, by the $A_\ad$-equivariance of the unfolding map $\bar\U$, and by \eqref{deff00}, it suffices to apply the operator $h_{L(\frac{\check \alpha}{2},\frac{1}{2})^{-1}}$ to the image of $f_{\bar X}$ in order to arrive at the image of $f^\circ_{\bar X}$; but this operator maps $f_{L(\Sym^2\oplus e^{\frac{\check\alpha}{2}},1)}$ to $f_{L(\Sym^2,1)}$. (Notice that the action of $h_{L(\frac{\check \alpha}{2},\frac{1}{2})^{-1}}$ in \eqref{deff00} is the \emph{normalized} one, which corresponds to the unnormalized action of the analogous element $h_{L(\frac{\check \alpha}{2},1)^{-1}}$.)

Consider the tensor product representation 
\begin{equation}\label{tensorprod} \otimes: \Gm\times \SL_2 \times \SL_2 \twoheadrightarrow \check{\tilde G} \hookrightarrow \GL_4,\end{equation}
where the first arrow is dual to the quotient $\tilde G \to A_\ad \times \PGL_2^2$, with $A_\ad$ identified as $\Gm$ via the positive root character.

By Rankin--Selberg theory, the image of the \emph{function }$1_{\bar X(\mathfrak o)}$ under unfolding is the Whittaker function which is a generating series for local tensor product $L$-value $L(\otimes, 0)$:
\begin{equation}\label{UPhi0} \U(1_{\bar X(\mathfrak o)})  =   \sum_{i\ge 0}  \tilde h_{\Sym^i(\otimes)} \star F_0,
\end{equation}
where $F_0 \in \mathcal F(Z,\mathcal L_\psi)$ is the Whittaker function which is supported on $Z(\mathfrak o) = N^2\cdot \tilde G(\mathfrak o)$ and is equal to $1$ on $\tilde G(\mathfrak o)$, and our notation for elements in the Hecke algebra (here denoted by a tilde, because of $\tilde G$) is as in \S \ref{ssnonstandard}. Here the convolution is as in \eqref{ssnonstandard}, unnormalized. Equivalently, 
when the measure $dz = \delta(z) d^\times z$ on $\bar Z$ is normalized to have total mass $1$ (this is not our standard normalization! see below) on $Z(\mathfrak o)$, we have
$$\int_{Z} \U(1_{\bar X(\mathfrak o)})(z) W_\pi(z) dz = L(\tilde\pi,\otimes,1),$$
where $W_{\tilde\pi}$ is the unramified Whittaker function of an unramified representation $\tilde\pi$, normalized to have value $1$ at $1$; indeed, this is the statement of \cite[Proposition 15.9]{Jacquet-Auto2}.\footnote{As a check for the normalization, the intersection of $\bar X(\mathfrak o)$ with the preimage of $V^*\times V^* (\mathfrak o)$ is equal to $X(\mathfrak o)$; hence, over $V^*\times V^* (\mathfrak o)$ the function $1_{\bar X(\mathfrak o)}$ coincides with the characteristic function of $X(\mathfrak o)$, and Fourier transform on the fibers takes it to the Whittaker function which, as a function on $\tilde Z$, is equal to $1$ on a point $(x, \ell)$ with $x \in X(\mathfrak o)$ having image $(v_1,v_2) \in V^*\times V^*(\mathfrak o)$, and $\ell: N\to \Ga$ an integral isomorphism, where $N$ is the stabilizer of $v_1$.}

We can multiply \eqref{UPhi0} by the dual measures $dx$ and $dz$ but notice that, if $dx(\bar X(\mathfrak o))=1$, then $dx(X(\mathfrak o)) = (1-q^{-2})$, hence by Lemma \ref{lemmameasures} $dz(\bar Z(\mathfrak o))=(1-q^{-2})$, and $dz(Z(\mathfrak o))=(1-q^{-1})(1-q^{-2})$, hence the factor in \eqref{unfoldingbasicbarX}. 

To arrive at \eqref{unfoldingbasic}, we compute the image (push-forward) of the measure $\U(1_{\bar X(\mathfrak o)}) dz$ in $\mathcal S(N,\psi\backslash G/N,\psi)$.
 By \eqref{UPhi0} and the volume calculation that we just did, it will coincide with 
\begin{equation}\label{itwillcoincide} (1-q^{-1})(1-q^{-2}) p_! \left(\sum_{i\ge 0} q^{-i} m_!(\tilde h_{\Sym^i(\otimes)})\right),
\end{equation}
where $p_!$ is the twisted push-forward of \S \ref{sstwistedpf}, and $m_!$ denotes push-forward (essentially, convolution) with respect to the map
$$m: \tilde G\to G$$
descending from the map 
$$ A\times \SL(V)^2 \ni (a,g_1, g_2)\mapsto (a,g_1 g_2^{-1}) \in A \times \SL(V).$$

\begin{lemma}\label{convolution}
The following diagram commutes:
$$\xymatrix{
\mathcal H(\tilde G, \tilde K) \ar[r]^{m_!} \ar[d] & \mathcal H(G,K) \ar[d] \\
\CC[\check{\tilde G}]^{\check{\tilde G}} \ar[r] & \CC[\check G]^{\check G},
}$$
where the vertical arrows denote the Satake isomorphism, and the bottom horizontal arrow is induced by the map of dual groups:
$$ m^*: \check G = \Gm\times \PGL_2 \ni (\chi, x)\hookrightarrow [\chi, \tilde x, \tilde x] \in \check{\tilde G}.$$
\end{lemma}

\begin{proof}[Proof of the lemma]
The statement easily reduces to the corresponding statement for the push-forward under $\tilde G \to G \to \SL_2$, simply by ``slicing'' a Hecke element along preimages of $A_\ad(\mathfrak o)$-cosets in $A_\ad$. Thus, consider the push-forward map 
$$ A\backslash \tilde G \simeq \SL_2^2/\{\pm 1\} \simeq \SO_4 \xrightarrow{m} \SL_2,$$
where we have denoted again by $m$ the action map on the identity element of $\SL_2$. We want to show that it induces the morphism dual to the morphism of dual groups 
$$ \PGL_2\overset{\diag}\hookrightarrow \SL_2^2/\{\pm 1\}$$
on the Hecke algebra. 

There is nothing special about $\SL_2$ here, besides the fact that the Chevalley involution is inner. The general statement is that, for a group $H$ with center $Z$, the action map
$$ \tilde H:= H\times^Z H \xrightarrow{m} H$$
induced from $(g_1, g_2)\mapsto g_1^{-1} g_2$  induces the map of Hecke algebras dual to 
$$ \check H \ni \check g\mapsto (\check g^c, \check g) \in \check{\tilde H},$$
where $c$ is a Chevalley involution (corresponding to the map $h\mapsto h^\vee(g)= h(g^{-1})$ on Hecke algebras).

The statement holds, because the Satake isomorphism is an algebra isomorphism, for the action maps
$$ H\times H\to H$$
and 
$$ H_\ad \times H_\ad \to H_\ad,$$
where $H_\ad = H/Z$. There is a canonical map of Hecke algebras, from the Hecke algebra of $H\times H$ to that of $\tilde H$ to that of $H_\ad\times H_\ad$ (push-forward followed, if necessary, by convolution by the probability measure of the hyperspecial subgroup of the target), hence we get a commutative diagram, using the Satake isomorphism (and assuming that $H$ is split, for notational simplicity):

$$\xymatrix{ \CC[\check{\tilde H}]^\inv \ar[r]^{m_!}\ar[d] & \CC[\check H]^\inv \ar[d] \\
\CC[\check H_{\rm sc}\times H_{\rm sc}]^\inv \ar[r] & \CC[\check H_{\rm sc}]^\inv,}$$
where $\check H_{\rm sc}$ denotes the simply connected cover of the dual group (= the dual group of $H_\ad$), and the exponent ``inv'' denotes invariants under conjugation. 

The right vertical arrow is injective, hence Lemma \ref{convolution} follows from the corresponding statement for $H_\ad$. 

\end{proof}

We continue with the proof of Theorem \ref{thmbasicvector}. 

The pullback of the tensor product representation of $\check{\tilde G}$ under $m^*$ is equal to the sum 
$$ \left(e^{\frac{\check\alpha}{2}} \otimes \Ad\right) \oplus e^{\frac{\check\alpha}{2}}$$
of representations of $\check G = \check A_\ad \times \PGL_2$. When we identify $\check A_\ad$ with $\Gm$ by the character $e^{\frac{\check\alpha}{2}}$, the second summand becomes the identity representation of $\Gm$, that we will denote by $\Id$, while the first factor becomes the representation $\Sym^2$ of $\GL_2$, factoring through the quotient $\Gm\times \PGL_2$ (with the map to $\Gm$ being the determinant). 
Thus, \eqref{itwillcoincide} is equal to 
\begin{multline*} (1-q^{-1})(1-q^{-2})p_!  \left(\sum_{i\ge 0} q^{-i} h_{\Sym^i(\Id\oplus \Sym^2)} \right) \\
 = (1-q^{-1})(1-q^{-2})p_!  \left( \left(\sum_{i\ge 0} q^{-i} h_{\Sym^i(\Id)}\right) \star \left(\sum_{i\ge 0} q^{-i} h_{\Sym^i(\Sym^2)}\right)\right) \\
 = (1-q^{-1})(1-q^{-2})   \left(\sum_{i\ge 0} q^{-i} h_{\Sym^i(\Id)}\right)\star p_! \left(\sum_{i\ge 0} q^{-i} h_{\Sym^i(\Sym^2)}\right),
\end{multline*}
the last step because the action of $A_\ad\simeq \Gm$ commutes with twisted push-forward.

This is the push-forward of $\U(1_{\bar X(\mathfrak o)})dz$, i.e., the element $\bar\U(f_{\bar X})$, and according to \eqref{deff00}, $\bar\U(f^\circ_X)$ will be obtained by applying to it, under the \emph{normalized} action, the element of $\mathcal S(\widehat{A_\ad})$ whose Mellin transform is $L(\chi, \frac{\check\alpha}{2}, \frac{1}{2})^{-1}$. This cancels the factor $ \left(\sum_{i\ge 0} q^{-i} h_{\Sym^i(\Id)}\right)$, and we arrive at the statement of the theorem.

\end{proof}

\begin{corollary}\label{corbasicvectorSym2}
 The statement of Theorem \ref{thmSym2} on basic vectors holds: the vector $f_{L(\Sym^2,\frac{1}{2})}$ is contained in $\mathcal D_{L(\Sym^2,\frac{1}{2})}(N,\psi\backslash G/N,\psi)$, and 
 $$\mathcal H_{\Sym^2} (h\cdot f_{L(\Sym^2,\frac{1}{2})}) = h\cdot f_{L((\Sym^2)^\vee,\frac{1}{2})},$$ 
 for any element $h$ of the unramified Hecke algebra of $G$.
\end{corollary}

\begin{proof}
 Notice that when passing from measures to half-densities as per Remark \ref{denstomeas}, because of a factor of $\delta^{-\frac{1}{2}}$ the measure $f_{L(\Sym^2,1)}$ is mapped to the half-density that we denote by $f_{L((\Sym^2)^\vee,\frac{1}{2})}$. 
 It follows from Theorem \ref{thmbasicvector} that $f_{L(\Sym^2,\frac{1}{2})}$ is contained in $\mathcal D_{L(\Sym^2,\frac{1}{2})}(N,\psi\backslash G/N,\psi)$. Moreover, by Theorems \ref{thmbasicvector} and \ref{thmbasicvectorX}, it is mapped by $\mathcal H_{\Sym^2}$ to $f_{L((\Sym^2)^\vee,\frac{1}{2})}$. Since  $\mathcal H_{\Sym^2}$ descends from a $\tilde G$-equivariant transform, the same is true when we act by the unramified Hecke algebra of $\tilde G$, which by Lemma \ref{convolution} descends to the action of the unramified Hecke algebra of $G$.
\end{proof}

Finally, we can now prove the remaining assertion \eqref{four} (the fundamental lemma) of Theorem \ref{thmRudnick}, which we recall:

\label{Pfpart4}

\begin{theorem}\label{Rudnickpt4} At non-Archimedean places, unramified over the base field, the transfer operator $\mathcal T_{\SL_2}$ of Theorem \ref{Rudnickpt3} satisfies the fundamental lemma for the Hecke algebra up to a factor of $\zeta(2)=(1-q^{-2})^{-1}$, namely: for all $h\in \mathcal H(\SL_2,K)\subset S(\SL_2)$, it takes the element
 $$ h\cdot f_{L(\Ad, 1)} \in \mathcal S_{L(\Ad,1)}^-(N,\psi\backslash \SL_2/N,\psi)$$ to the image of $\zeta(2) h$ in $\mathcal S(\frac{\SL_2}{\SL_2})$.
\end{theorem}

\begin{proof}
By Statement \eqref{three} of Theorem \ref{thmRudnick}, proven at the end of \S \ref{Pfpart3}, the operator $\bar{\U}$ descends to the operator $\mathcal T^{-1}= \mathscr F_{\Id, 1}^{-1}$ in the coordinates of the theorem; this is the bottom arrow of diagram \eqref{quotdiagram}, for $s=-\frac{1}{2}$. 

Now we descend Theorem \ref{thmbasicvector}, \eqref{deff00}, and \eqref{barXtoX} to the bottom row of diagram \eqref{quotdiagram}, for $s=-\frac{1}{2}$; that corresponds to evaluating Satake transforms for $A_\ad$ at the character $\delta^{-\frac{1}{2}}$, and to push-forwards of measures to the spaces $A_\ad\backslash X\sslash\SL_2= \Dfrac{\SL_2}{\SL_2}$ and $(N,\psi)\backslash \SL_2/(N,\psi)$. Notice that 
the push-forward of $f_X$ to $\Dfrac{\SL_2}{\SL_2}$ is equal to the push-forward of the identity element of the  unramified Hecke algebra of $\SL_2$ under $\SL_2 \to \Dfrac{\SL_2}{\SL_2}$ --- let us denote it by $f_{\frac{\SL_2}{\SL_2}}$. By \eqref{barXtoX} (or, directly from the definitions), this will be the same as the push-forward of $f_{\bar X}$ which, by \eqref{deff00}, is $(1-q^{-1})^{-1}$ times the push-forward of $f_{\bar X}^\circ$. By Theorem \ref{thmbasicvector}, this will map to the element $(1-q^{-2}) f_{L(\Ad, 1)}$ in $\mathcal S^-_{L(\Ad,1)}(N,\psi\backslash \SL_2/N,\psi)$, hence
$$ \mathcal T^{-1} f_{\frac{\SL_2}{\SL_2}} = (1-q^{-2}) f_{L(\Ad, 1)}.$$

\end{proof}

\begin{remark}
 As a reality check, let us compute the limit (which stabilizes) 
 $$\lim_{|\zeta|\to\infty} \frac{\mathcal T^{-1}f_{\frac{\SL_2}{\SL_2}}}{d^\times \zeta}.$$
 
 On one hand, a straightforward calculation using the formula $\mathcal T^{-1}= \mathscr F_{\Id, 1}^{-1}$ shows that this limit is equal to the total mass of $f_{\frac{\SL_2}{\SL_2}}$, which is $1$. On the other hand, as we saw in \eqref{asymptoticsbasicadjoint}, this is also the limit for $(1-q^{-2}) f_{L(\Ad, 1)}$.
\end{remark}

\subsection{Relative characters} \label{ssHankelrelchars}

Let $\iota$ be the involution on $\tilde G= (A\times \SL_2)^2/\{\pm 1\}^\diag$ that is induced from the inversion map on $A$.  Recall that we have a canonical identification $Z\simeq A_\ad \times V^*\times V^*$; The inversion map on $A_\ad$ gives rise to an endomorphism $j$ of $Z$ which is $(\tilde G,\iota)$-equivariant, i.e., it twists the action of $\tilde G$ by $\iota$. 

There is also an endomorphism of the line bundle $\mathcal L_{\psi^{-1}}$ which is compatible with $j$. It is described as follows: Recall that $\mathcal L_{\psi^{-1}}$ was obtained by reduction (via the character $\psi^{-1}$) of the $\Ga$-bundle
$\tilde Z\to Z$, see \S \ref{ssunfolding}, whose points can be identified with triples $(v,g, z) \in V^*\times G \times \Ga$. Thus, it is enough to define a lift $\tilde j$ of $j$ to $\tilde Z$. The obvious map $(v, g, z)\mapsto (v,g, z^{-1})$ will not work, because it is not $(\tilde G,\iota)$-equivariant. Instead, we define 
$$\tilde j(v, g, z) = (z^{-1}v, g, z^{-1}).$$
This induces an identification $j^* \mathcal L_{\psi^{-1}} \simeq \mathcal L_{\psi^{-1}}$, but because of our normalization \eqref{actionnorm-unfolding-fns} of the action on functions or sections, this does not quite induce a $(\tilde G,\iota)$-equivariant map on sections. We will instead work with half-densities, denoting by $\mathcal D^\infty(Z,\mathcal L_{\psi^{-1}})$ the product of the space $C^\infty(Z,\mathcal L_{\psi^{-1}})$ of smooth sections by $(dz)^\frac{1}{2} = (\delta(z) d^\times z)^\frac{1}{2}$, see \eqref{fnsdensmeas}. (In particular, $\mathcal D^\infty$ contains the space of Schwartz half-densities, that we denote simply by $\mathcal D$.)

Thus, we get a $(\tilde G,\iota)$-equivariant pullback map of $\mathcal L_{\psi^{-1}}$-valued half-densities:
$$ j^*:\mathcal D^\infty(Z,\mathcal L_{\psi^{-1}}) \to \mathcal D^\infty(Z,\mathcal L_{\psi^{-1}}).$$

Explicitly, in terms of the isomorphisms \eqref{unfoldingquotients}, we have just identified $\mathcal D^\infty(Z,\mathcal L_{\psi^{-1}})$ with Whittaker functions $C^\infty(\tilde N,\tilde\psi^{-1}\backslash \tilde G)$ times the product of $\delta^\frac{1}{2}$ (considered as a function on $G$) by a Haar half-density, and applied the automorphism $\iota$ to $\tilde G$ (which does not change the character $\tilde \psi$ of $\tilde N$); here, $\tilde N = N\times N$ is the stabilizer of a point on the diagonal $V^*\hookrightarrow V^*\times V^*$, and $\tilde\psi^{-1}$ is the character by which it acts on the fiber of $\mathcal L_{\psi^{-1}}$ --- it is of the form $\psi^{-1}\boxtimes\psi$ for the identification $N\simeq\Ga$ induced by the symplectic structure.

Let $\tau$ be a generic irreducible representation of $\tilde G$. The composition $\tau\circ \iota$ is isomorphic to the contragredient $\tilde\tau$ of $\tau$. The (half-density-valued) Whittaker model of $\tau$ is the subspace of $\mathcal D^\infty(Z,\mathcal L_{\psi^{-1}})$ that is isomorphic to $\tau$; it will be denoted by $\mathcal W(\tau)$.
Thus, $j^* \mathcal W(\tau) = \mathcal W(\tilde \tau)$. 

There is a $\tilde G$-invariant pairing
$$ \mathcal D(\bar X) \otimes \mathcal W(\tau) \to \CC,$$
defined by the convergent integral 
$$ \varphi\otimes W_\tau\mapsto \left< \varphi, W_\tau\right>:= \int_{Z} \U(\varphi) W_\tau$$
when the central character of $\tau$ is such that elements of $\mathcal W(\tau)$ vanish sufficiently rapidly on $\bar Z\smallsetminus Z$, and by meromorphic continuation otherwise. 

By the $(\tilde G,\iota)$-equivariance of the symplectic Fourier transform $\mathfrak F$ on $\mathcal D(\bar X)$, the pairing
$$ \varphi\otimes W_\tau \mapsto \left<\mathfrak F\varphi, j^*W_\tau\right>$$
is also a $\tilde G$-invariant pairing between the same spaces, varying meromorphically in $\tau$. By a multiplicity-one property, it has to be a meromorphic multiple of the former, i.e., there is a meromorphic scalar
$$ \gamma(\tau, \otimes, \frac{1}{2},\psi)$$
such that 
$$ \left<\mathfrak F\varphi, j^*W_\tau\right> = \gamma(\tau, \otimes, \frac{1}{2},\psi) \left< \varphi, W_\tau\right>.$$
This is the \emph{Rankin--Selberg} gamma factor, denoted by $\epsilon'(\frac{1}{2},\tau,\psi)$ in \cite[(14.8.6)]{Jacquet-Auto2}.

\begin{remark}
 To compare to \cite{Jacquet-Auto2}, we need to identify the group $\tilde G$ with the group $\GL_2\times_{\Gm, \det} \GL_2$ via the map that descends from 
 $$ A\times \SL_2^2\ni (a,g_1, g_2)\mapsto (e^\frac{\alpha}{2}(a) g_1, e^\frac{\alpha}{2}(a) g_2).$$
 The integral $\Psi(\frac{1}{2}, W_1, W_2,\Phi)$ of Jacquet, then, is\footnote{The Schwartz function $\Phi$ of Jacquet does not live in the space $\mathcal F(\bar X)$, but only on the fiber of $\bar X\to \SL_2$ over the identity element; however, we can convolve Jacquet's function, viewed as a generalized function, by a Schwartz measure on $\tilde G$ to arrive at an element of $\mathcal F(\bar X)$.} our pairing $\left< \Phi, W_\tau\right>$, except that Jacquet defines the unfolding map by using the character $\psi$ instead of $\psi^{-1}$ in \eqref{unfolding-formula}, hence his function $W_1(g_1)W_2(\eta g_2)$ lives in $ C^\infty(\tilde N,\tilde\psi\backslash \tilde G),$ instead of our $ C^\infty(\tilde N,\tilde\psi\backslash \tilde G)$ --- this does not affect the functional equation.
The involution $\iota$ on $\tilde G$ reads $(g_1, g_2)\mapsto (\frac{g_1}{\det(g_1)} , \frac{g_2}{\det(g_2)})$ on $\GL_2\times_{\Gm,\det}\GL_2$. The integral $\tilde\Psi(\frac{1}{2}, W_1, W_2,\hat\Phi)$ of Jacquet, though, does not fully correspond to our $\left<\mathfrak F\Phi, j^*W_\tau\right>$, because it arises from applying the involution 
\[W_1(g_1)\cdot W_2(\eta g_2) \mapsto W_1(g_1) \omega^{-1}(\det g_1) \cdot W_2(\eta g_2) \omega_2^{-1}(\det (\eta g_2))\] to Whittaker functions, where $\omega = (\omega, \omega_2)$ are the central characters for the two factors of $\tau$, and $\eta = \begin{pmatrix} -1 \\ & 1 \end{pmatrix}$. In contrast, our involution would read $W(g) \mapsto W(g) \omega^{-1}(g)$, for $g=(g_1, g_2)$. The two differ by a factor of $\omega_2(-1)$, which is why this factor appears in Jacquet's functional equation, but not ours.
\end{remark}

When the restriction of $\tau$ to the subgroup $\SL_2^2\subset \tilde G = \GL_2\times_{\Gm, \det} \GL_2$ contains an irreducible representation of the form $\sigma \otimes \sigma$ (notice that $\tilde\sigma\simeq \sigma$ for $\SL_2$), then $\tau$ is the restriction of an irreducible representation $\tau_1\boxtimes \tau_2$ of $\GL_2^2$, with $\tau_1 \simeq \widetilde{\tau_2} \otimes (\chi\circ \det)$, for some character $\chi$ of $F^\times$. For such a $\tau$, consider the irreducible representation $\pi=\chi\boxtimes \sigma$ of $G=A_\ad\times \SL_2$. One defines the symmetric-square $\gamma$-factor as 
\begin{equation} \gamma(\chi\boxtimes \sigma, \Sym^2, \frac{1}{2},\psi) :=  \frac{\gamma(\tau, \otimes, \frac{1}{2},\psi)}{\gamma(\chi,\frac{\check\alpha}{2}, \frac{1}{2},\psi)}.\end{equation}
It clearly depends only on the $L$-packet of $\chi\boxtimes \sigma$.

For the following theorem, we consider half-density-valued relative characters for the Kuznetsov formula, that we will also denote by $J_\pi$, as we  denoted the corresponding generalized functions in \S \ref{ssrelchars-Kuz}. In terms of the coordinates $(a,\zeta)$ that we have been using (with $a$ the value of the positive root character on $A_\ad$, and $\zeta$ our usual coordinate for $N\backslash \SL_2\sslash N$), the (generalized) half-density $J_\pi$ is the product of the generalized function $J_\pi$ by the half-density $|\zeta|\cdot (d^\times a d^\times \zeta)^\frac{1}{2}$.

\begin{theorem}
 The Hankel transform 
 $$ \mathcal H_{\Sym^2}: \mathcal D^-_{L(\Sym^2, \frac{1}{2})}(N,\psi\backslash G/N,\psi) \xrightarrow\sim \mathcal D^-_{L((\Sym^2)^\vee, \frac{1}{2})}(N,\psi\backslash G/N,\psi)$$
 satisfies 
\begin{equation}\label{Hankel-sym2-characters}\mathcal H_{\Sym^2}^* J_\pi = \gamma(\pi, \Sym^2, \frac{1}{2}, \psi) \cdot J_\pi,\end{equation}
 for relative characters $J_\pi$ on $(N,\psi\backslash G/N,\psi)$, understood as an identity of meromorphically varying functionals on $\mathcal D^-_{L(\Sym^2, \frac{1}{2})}(N,\psi\backslash G/N,\psi)$, as $\pi$ varies in any family of irreducible representations of $G$. 
\end{theorem}

\begin{proof}
 By definition, 
 $\mathcal H_{\Sym^2} =  \bar{\U}^\vee \circ \mathcal H_X^\circ \circ \bar{\U}^{-1},$
 where $\bar\U$ is the descent of the unfolding map $\U$ to push-forwards modulo the $\SL_2^\diag$-action, and $\bar\U^\vee = j^*\circ \bar\U$. 
 
 We first study the transform 
 $$\bar{\U}^\vee \circ \mathcal H_X \circ \bar{\U}^{-1},$$
 which is equal to the lower row of the following commutative diagram (using the notation introduced after Lemma \ref{unfoldingextends}), followed by the inversion map $j^*$:
 $$ \xymatrix{
 \mathcal D^{--}(Z,\mathcal L_\psi)  \ar[r]^{\U^{-1}}\ar[d] & \mathcal D(\bar X) \ar[r]^{\mathfrak F}\ar[d] & 
 \mathcal D(\bar X) \ar[r]^{\U}\ar[d] &  \mathcal D^{--}(Z,\mathcal L_\psi)\ar[d] \\ 
 \mathcal D^{--}(N,\psi\backslash G/N,\psi)  \ar[r]^{\bar\U^{-1}} & \mathcal D(\bar X/\SL_2) \ar[r]^{\mathcal H_X} & 
 \mathcal D(\bar X/\SL_2) \ar[r]^{\bar\U} &  \mathcal D^{--}(N,\psi\backslash G/N,\psi). 
 }$$
 
 Write $\pi$ as $\chi\boxtimes\sigma$, and let $\tau_1$ be an irreducible representation of $\GL_2$ whose restriction to $\SL_2$ contains $\sigma$ as the unique $(N,\psi^{-1})$-generic subrepresentation; we denote by $\omega$ its central character. Let $\mathcal W(\tau_1)\subset \mathcal D^\infty(N,\psi^{-1}\backslash \GL_2)$ denote the Whittaker model of $\tau_1$, and $\mathcal W'(\tau_1)$ the Whittaker model defined with the inverse character $\psi$. 
 Fix any dual pairing $\mathcal W'(\tau_1\otimes (\omega^{-1}\circ\det)) \simeq \widetilde{\mathcal W(\tau_1)}$; a Kuznetsov relative character for the representation $\tau_1$ of $\GL_2$ can be thought of as the $\GL_2^\diag$-invariant generalized Whittaker half-density
 $$ \mathbb W_{\tau_1}:= \sum_i W_i \boxtimes W_i^\vee \in \left(\mathcal D(N,\psi\backslash \GL_2 \times N,\psi^{-1}\backslash \GL_2)\right)^*,$$
 where $(W_i, W_i^\vee)$ runs over a dual basis of $\mathcal W(\tau_1)\boxtimes \mathcal W'(\tau_1\otimes (\omega^{-1}\circ\det))$. 
 
 We identify the space $Z$ with $(N\backslash \GL_2)\times_{\Gm,\det} (N\backslash \GL_2)$, and we freely restrict half-densities to subspaces, by choosing the necessary Haar half-densities; these choices will not matter for the functional equation. Consider the restriction of the half-density $\mathbb W_{\tau_1}\cdot (\chi\circ\det)$ to $Z$; it is the pullback of the Kuznetsov relative character $J_\pi$ for the representation $\pi = \chi\boxtimes \sigma$ of $G$, 
  up to a scalar that is independent of $\chi$.

 Let $\tau$ be the representation $\tau_1\boxtimes \widetilde{\tau_1} \otimes (\chi\circ\det)$ of $\GL_2^2$. 
  By the definition of the Rankin--Selberg $\gamma$-factors that we saw above, for every $\varphi\in \mathcal D(\bar X)$ we have
  $$\int_Z \U^\vee \mathfrak F\varphi \cdot  \mathbb W_{\tau_1} \cdot (\chi\circ\det)   =  \gamma(\tau, \otimes, \frac{1}{2},\psi) \int_Z \U \varphi \cdot  \mathbb W_{\tau_1} \cdot (\chi\circ\det).$$
  
  If $f$ is the image of $\U \varphi $ in $\mathcal D^-_{L(\Sym^2, \frac{1}{2})}(N,\psi\backslash G/N,\psi)$, this can also be written:
  $$\int_{N\backslash G\sslash N} \bar\U^\vee \mathcal H_X \bar{\U}^{-1} f \cdot  J_\pi =  \gamma(\tau, \otimes, \frac{1}{2},\psi) \int_{N\backslash G\sslash N}  f \cdot  J_\pi .$$
  
  Thus, the adjoint  of $ \bar{\U}^\vee \circ \mathcal H_X \circ \bar{\U}^{-1}$ acts on $J_\pi$ by the scalar $ \gamma(\tau, \otimes, \frac{1}{2},\psi)$. On the other hand, by \eqref{relationbetweenH}, 
  $\mathcal H_X = \mathscr F_{\frac{\check\alpha}{2}, \frac{1}{2}} \circ \mathcal H_X^\circ$, so by \eqref{FE} the adjoint of $ \mathcal H_{\Sym^2} = \bar{\U}^\vee \circ \mathcal H_X^\circ \circ \bar{\U}^{-1}$ acts on $J_\pi$ by the scalar 
  $$\frac{\gamma(\tau, \otimes, \frac{1}{2},\psi)}{\gamma(\chi,\frac{\check\alpha}{2}, \frac{1}{2},\psi)} =  \gamma(\pi, \Sym^2, \frac{1}{2}, \psi).$$
\end{proof}

\begin{remark}
 A careful choice of the duality between $\mathcal W'(\tau_1\otimes (\omega^{-1}\circ\det))$ and ${\mathcal W(\tau_1)}$ could lead to a direct proof of Statement \eqref{two} of Theorem \ref{thmRudnick}.
\end{remark}

\subsection{Comparison with the boundary degeneration and the standard $L$-function} \label{ssstandard}
 
 Here $F$ is a non-Archimedean field. The asymptotics morphism \eqref{asymptotics}
$$ e_\emptyset^*: \mathcal S(Z,\mathcal L_\psi) \to \mathcal S^+(Z)$$
extends to the image of the space $\mathcal S(\bar X)$ under the unfolding map $\U$, which we have denoted (after Lemma \ref{unfoldingextends}) by $\mathcal S^{--}(Z,\mathcal L_\psi)$.  
The image of $\mathcal S^{--}(Z,\mathcal L_\psi)$ under $e_\emptyset^*$ will be denoted by $\mathcal S^{\pm\pm} (Z)$. There is a commutative diagram
$$ \xymatrix{
\mathcal S^{--}(Z,\mathcal L_\psi) \ar[d]\ar[rr]^{e_\emptyset^*} && \mathcal S^{\pm\pm}(Z)\ar[d] \\
\mathcal S^{--}(N,\psi\backslash G/N,\psi) \ar[rr] &&\Meas(N\backslash G\sslash N),
}$$
where the vertical arrows are the natural push-forward maps, and the bottom horizontal arrow exists by an easy extension of Theorem \ref{density}; let us denote it, too, by $e_\emptyset^*$. We are interested in the subspace 
\[\mathcal S^-_{L(\Sym^2, 1)}(N,\psi\backslash G/N,\psi)\subset \mathcal S^{--}(N,\psi\backslash G/N,\psi).\] Denote by $\mathcal S^\pm_{L(\Sym^2, 1)}(N\backslash G/N)$ its image under $e_\emptyset^*$. 

Let $A_G \simeq A_\ad \times A$ denote the Cartan of $G$; by the conventions of \S \ref{sstwistedpf}, it is identified with an open subset of $N\backslash G\sslash N$. We also have corresponding spaces of half-densities, with a map
$$ \xymatrix{
\mathcal D^-_{L(\Sym^2, \frac{1}{2})}(N,\psi\backslash G/N,\psi)\ar[rr]^{e_\emptyset^*} && \mathcal D^\pm_{L(\Sym^2, \frac{1}{2})}(N\backslash G/N),
}$$
by dividing by the half-density $(|a|d^\times a \cdot |\zeta|^2 d^\times \zeta)^\frac{1}{2}$ in our usual coordinates $(a,\zeta)$ for $A_\ad\times A$ (see Remark \ref{denstomeas}).
Applying the inversion map $j$ on the $A_\ad$-coordinate of $G$, we get a similar space $\mathcal D^\pm_{L((\Sym^2)^\vee, \frac{1}{2})}(N\backslash G/N)$ of half-densities on $N\backslash G\sslash N$, which is the asymptotic image of $\mathcal D^-_{L((\Sym^2)^\vee, \frac{1}{2})}(N,\psi\backslash G/N,\psi)$.

\begin{theorem}\label{Sym2degen}
There is a unique $A_G$-equivariant operator $\mathcal H_{\Sym^2,\emptyset}$ making the following diagram commute:
$$ \xymatrix{
\mathcal D^-_{L(\Sym^2, \frac{1}{2})}(N,\psi\backslash G/N,\psi) \ar[rr]^{e_\emptyset^*}\ar[d]^{\mathcal H_{\Sym^2}} && \mathcal D^\pm_{L(\Sym^2, \frac{1}{2})}(N\backslash G/N) \ar[d]^{\mathcal H_{\Sym^2,\emptyset}} \\
\mathcal D^-_{L((\Sym^2)^\vee, \frac{1}{2})}(N,\psi\backslash G/N,\psi) \ar[rr]^{e_\emptyset^*} && \mathcal D^\pm_{L((\Sym^2)^\vee, \frac{1}{2})}(N\backslash G/N)}.
$$
It is given by $\mathcal H_{\Sym^2,\emptyset} = \mathscr F_{-\check\lambda_+,\frac{1}{2}} \circ \mathscr F_{-\check\lambda_0,\frac{1}{2}} \circ \mathscr F_{-\check\lambda_-,\frac{1}{2}}$, understood as regularized Fourier convolutions (see \S \ref{sssFourierconv}).
\end{theorem}

The proof is very similar to that of Theorem \ref{groupdegen}, and will be omitted. It is, in fact, easier, since here we do not need to compare relative characters on different spaces, but only on $(N,\psi)\backslash G/(N,\psi)$ --- thus, the main input is simply \eqref{Hankel-sym2-characters} on the action of $\mathcal H_{\Sym^2}$ on relative characters, and the stated $\mathcal H_{\Sym^2,\emptyset}$ is simply the only $A_G$-equivariant transform that can act on relative characters by the same gamma factors. Again, at no point will we need to use the explicit formula \eqref{HSymformula} for $\mathcal H_{\Sym^2}$. Comparing this formula with the theorem above, we discover again a very gratifying reality: $\mathcal H_{\Sym^2}$ is simply a \emph{deformation} of the abelian transform $\mathcal H_{\Sym^2,\emptyset}$! The nature of this deformation, presently, escapes my understanding.

It is worth comparing with the calculation of the Hankel transform for the \emph{standard} $L$-function in a paper of Jacquet \cite{Jacquet}, as presented in \cite[\S 8]{SaHanoi}. The group here is $G'=\GL_2$ (or, more generally, $\GL_n$, but here we will restrict ourselves to $\GL_2$), and the analog of the operator $\mathcal H_{\Sym^2}$ is an operator 
$$\xymatrix{ \mathcal D^-_{L(\Std,\frac{1}{2})}(N,\psi\backslash G'/N,\psi) \ar[rr]^{\mathcal H_{\Std}} && \mathcal D^-_{L(\Std^\vee,\frac{1}{2})}(N,\psi\backslash G'/N,\psi)},$$
given by the formula 
\begin{equation}
  \mathcal H_{\Std} = \mathcal F_{-\check\epsilon_1,\frac{1}{2}} \circ \psi(-e^{-\alpha_1}) \circ \mathcal F_{-\check\epsilon_2,\frac{1}{2}},
\end{equation}
where $\check\epsilon_1, \check\epsilon_2$ are the weights of the standard representation of the dual group, and the factor $\psi(-e^{-\alpha_1}) $ in the middle denotes the operator of multiplication by this function. 

The space $\mathcal D^-_{L(\Std,\frac{1}{2})}(N,\psi\backslash G'/N,\psi)$, here, is simply the twisted push-forward (\S \ref{sstwistedpf}) of the space of Schwartz half-densities on $\Mat_2$. The transform $\mathcal H_{\Std}$ descends from Fourier transform on the space of $2\times 2$ matrices, and by Godement--Jacquet theory satisfies the following functional equation for half-density-valued relative characters:
\begin{equation}\label{Hankel-std-characters}\mathcal H_{\Std}^* J_\pi = \gamma(\pi, \Std, \frac{1}{2}, \psi) \cdot J_\pi.\end{equation}
(This is the analog of \eqref{Hankel-sym2-characters}.)

The analog of Theorem \ref{Sym2degen} is:
\begin{theorem}\label{Stddegen}
There is a unique $A_G$-equivariant operator $\mathcal H_{\Std,\emptyset}$ making the following diagram commute:
$$ \xymatrix{
\mathcal D^-_{L(\Std, \frac{1}{2})}(N,\psi\backslash G/N,\psi) \ar[rr]^{e_\emptyset^*}\ar[d]^{\mathcal H_{\Std}} && \mathcal D^\pm_{L(\Std, \frac{1}{2})}(N\backslash G/N) \ar[d]^{\mathcal H_{\Std,\emptyset}} \\
\mathcal D^-_{L((\Std)^\vee, \frac{1}{2})}(N,\psi\backslash G/N,\psi) \ar[rr]^{e_\emptyset^*} && \mathcal D^\pm_{L((\Std)^\vee, \frac{1}{2})}(N\backslash G/N)}.
$$
It is given by $\mathcal H_{\Std,\emptyset} = \mathcal F_{-\check\epsilon_1,\frac{1}{2}} \circ \mathcal F_{-\check\epsilon_2,\frac{1}{2}}$, understood as regularized Fourier convolutions.
\end{theorem}

Thus, again, the Hankel transform for the standard $L$-function on the Kuznetsov formula is simply a deformation of its abelian analog!

\section{Lifting from tori to $\GL_2$} \label{sec:Venkatesh}

Let $T'$ be a one-dimensional torus, and $r:{^LT'}\to \GL_2$ the standard $2$-dimensional representation of its $L$-group. That is, $r$ identifies the dual torus $\check T'$ with $\SO_2\subset \GL_2$, and if $T'$ is split the action of the Galois factor is trivial, while if $T'$ is non-split, the image of ${^LT'}$ is the disconnected subgroup $O_2\subset \GL_2$. 

Let $\bar r : {^LT} \to \check G = \Gm \times \PGL_2$ be the composition of $r$ with the natural projection (the first factor being the image of the determinant map), where ${^LT}={^LT'}/\{\pm 1\}$ is the $L$-group of a torus $T$ mapping to $T'$ with kernel $\{\pm 1\}$. This map of $L$-groups should correspond to a functorial lift from $T$ to $G$. This functorial lift is realized in the project of endoscopy by studying the unstable summands of the trace formula for $\SL_2$, but here we would like to revisit the thesis of Akshay Venkatesh \cite{Venkatesh}, which reproduces this lift by ``beyond endoscopy'' techniques, using the Kuznetsov formula. We will deduce the local comparison independently, without direct reference to Venkatesh's thesis, but we will also comment, after the statement of the main theorem, on the relation with his results. 

Interestingly, it will turn out that the endoscopic and the ``beyond endoscopy'' approaches are closely related to each other; in fact, we will see that the local transfer  directly produces a ``geometric'' map from $\mathcal S(N,\psi\backslash G/N,\psi)$ to the space of ``$\kappa$-orbital integrals'' for $\SL_2$ hence (through endoscopy) to the space $\mathcal S(T)$. 

\subsection{Statement of the results, and relation to Venkatesh's thesis}

 Let $T$ be a one-dimensional torus, with associated quadratic extension $E$ and quadratic character $\eta$. The torus could be split, in which case $E = F\oplus F$. 
We can identify $T$ with the kernel of the norm map $\Res_{E/F}\Gm \to \Gm$, and thus with a subgroup of $\SL_2\simeq \SL_F(\Res_{E/F}\Ga)$. The Weyl group $W$ of $T$ in $\SL_2$ acts on $T$ by inversion, and its coinvariants on the Schwartz space $\mathcal S(T)$ are canonically identified with $\mathcal S(T/W):=$ the push-forward of $\mathcal S(T)$ to the space $T\sslash W = \Dfrac{\SL_2}{\SL_2}$. \emph{To avoid confusion, I stress that the symbol $t$ will denote, as in the rest of the paper, the ``trace'' coordinate on $T\sslash W$, not an element in $t$.}

The main result of this section is 
\begin{theorem}\label{thmVenkatesh}
 There is a linear map
 $$ \mathcal T_T: \mathcal S^-_{L(\Sym^2, 1)} (N,\psi\backslash G/N,\psi) \to \mathcal S(T)$$
 with the following properties:
 \begin{enumerate}
  \item It is $(\Gm,\eta)$-equivariant (for the unnormalized action).
  \item Its image is the subspace $\mathcal S(T)^{\Z/2}$, where $\Z/2$ acts by inversion on $T$.
  \item The pullback of any unitary character $\theta$ of $T$ is equal to the Kuznetsov relative character $J_\Pi$, where $\Pi$ is the endoscopic $L$-packet associated to $\theta$ (see \S \ref{sstransfer-tori}).
  \item It satisfies the fundamental lemma for the Hecke algebra: For $E/F$ unramified non-Archimedean, the basic vector $f_{L(\Sym^2, 1)} \in \mathcal S^-_{L(\Sym^2, 1)} (N,\psi\backslash G/N,\psi)$ is mapped to $\frac{\zeta(1)\zeta(2)}{L(\eta,1)}$ times the basic vector of $\mathcal S(T)$ (the unit of the Hecke algebra). Moreover, for every $h$ in the unramified Hecke algebra of $G$, it takes the image of $h$ to $\frac{\zeta(1)\zeta(2)}{L(\eta,1)}\cdot \bar r^*h$, where $\bar r^*$, under Satake isomorphism, corresponds to pullback under the $L$-embedding: 
  $$\bar r:{^LT} \to {^LG}.$$
  \item Finally, if $da$ is the a Haar measure on $T$ which is pulled back (in the sense that it is equal on subsets mapping injectively) from the measure $|t^2-4|^{-\frac{1}{2}} dt$ in the trace coordinate $t=t(a)$ on $T\sslash W \simeq \Dfrac{\SL_2}{\SL_2}$, the transfer operator $\mathcal T_T$ is given by the formula
 \begin{equation}\label{Venkatesh-final}
 \frac{\mathcal T_T(f)(a)}{da} = (dt)^{-1} \lambda(\eta,\psi) \int_r \int_x f\left( r, \frac{t}{x} \right)  \eta(xrt) \psi(x) dx.
\end{equation}
Here, $f\in \mathcal S^-_{L(\Sym^2, 1)} (N,\psi\backslash G/N,\psi)$ is expressed as a measure in the same two coordinates denoted by $(c,\zeta)$ in \S \ref{sec:sym2}.
\end{enumerate}
\end{theorem}

The proof of the theorem will be completed in \S \ref{sstransfer-tori}.

\begin{remark}\label{remarkfeta}
The fact that this map should be $(\Gm,\eta)$-equivariant follows, of course, from the composition of $\bar r$ with projection to $\Gm$, which is equal to the quadratic Galois character attached to $E/F$. Thus, we could have integrated $f$ against $\eta$ from the beginning, obtaining an element of the space $\mathcal S^-_{L(\Ad,\eta |\bullet|)} (N,\psi\backslash \SL_2/N,\psi)$, according to Proposition \ref{propdescent} (because the unnormalized integral against $\eta$ is the normalized pushforward $p_{\eta\delta^{-\frac{1}{2}}}$ of this proposition). Let $f_\eta$ denote that element. Then, \eqref{Venkatesh-final} can be read as a map 
\begin{equation}
 \mathcal S^-_{L(\Ad,\eta |\bullet|)} (N,\psi\backslash \SL_2/N,\psi) \to \mathcal S(T)^{\mathbb Z/2},
\end{equation}
given by 
 \begin{equation}\label{Venkatesh-final-eta}
 \frac{\mathcal T_T(f)(a)}{da} = (dt)^{-1}\lambda(\eta,\psi) \int_x f_\eta\left( \frac{t}{x} \right)  \eta(xt) \psi(x) dx,
\end{equation}
which, up to the scalar $\lambda(\eta,\psi)$, is the ``usual'' Fourier transform of the measure $u\mapsto f_\eta(u^{-1}) \eta(u)$.   
The significance of working with the more complicated space $\mathcal S^-_{L(\Sym^2, 1)} (N,\psi\backslash G/N,\psi)$ lies in the global application, where the Euler product of the $\Gm$-integrals over all places does not converge, but corresponds to a pole of the $\Sym^2$ $L$-function.
\end{remark}

Now I will give a rough idea, without details, of how the transfer operator of Theorem \ref{thmVenkatesh} is related to Venkatesh's thesis. I refer to the exposition of \cite[\S 3]{Venkatesh} for the special case of the field $\QQ$, unless otherwise stated, because it is based on the classical Kuznetsov formula and is simpler to follow, notationally. 

The basic, although very coarse, idea is that one should take an Euler product of elements of $\mathcal S^-_{L(\Ad, \eta|\bullet|^s)} (N,\psi\backslash \SL_2/N,\psi)$, over all places, where almost every factor is equal to the basic vector, and plug it into the Kuznetsov formula. Classically, the Kuznetsov formula involves the Fourier coefficients $a_n(f)$ of $\GL_2$-automorphic forms $f$ (with fixed central character, here $\eta$), and plugging in these non-standard test functions corresponds to writing down a ``series of Kuznetsov formulas'' associated to the Dirichlet series of $L(f, \Sym, s)$ (which, because of the central character, can also be written as $L(f|_{\SL_2}, \Ad, \eta|\bullet|^s)$).
 As observed in \cite{Venkatesh}, up to a factor that is analytic at $s=1$, this series is $\sum_n a_{n^2}(f)  n^{-s}$, and the corresponding sum of Kunzetsov formulas --- or rather, of the discrete spectrum of their spectral expansion --- would be a sum of the form 
\begin{equation}\label{Venk3} \sum_f h(t_f) \sum_{n=1}^\infty \frac{a_{n^2}(f)}{n^s} \overline{a_m(f)},
\end{equation}
in the notation of that paper. The parameter $m$ is an auxilliary parameter which should be thought of as the result of replacing a basic function by a Hecke operator acting on it.
To compute the residue of this at $s=1$, Venkatesh computes the related expression (3) of his paper, or rather the asymptotics as some variable $X\to \infty$ of an alternate expression (18) that involves a smooth cutoff function $g\in C_c^\infty(\mathbb R^\times_+)$, whose discrete terms read
\begin{equation}\label{Venk18} \sum_f h(t_f) \sum_{n=1}^\infty g(\frac{n}{X}) a_{n^2}(f)  \overline{a_m(f)}.
\end{equation}

\begin{remark}\label{remarkanalytic}
 The global analytic issues that are addressed by introducing the cutoff function $g$ are beyond the scope of the present article, which focuses on local transfer operators. For the discussion that follows I will just focus on the formal structure of Venkatesh's argument, trying to match it to the local transfer operator of Theorem \ref{thmVenkatesh}, and will treat \eqref{Venk18} as an avatar for the ``limit of \eqref{Venk3} as $s\to 1$'', ignoring the appearance of the function $g$. It is conceivable that one could avoid this cutoff function, following the techniques of \cite{SaBE2}, which are based on the idea of deforming spaces of orbital integrals. I will also freely identify test measures for the Kuznetsov quotient with test functions (that is, with orbital integrals); the precise factors needed for a complete comparison with Venkatesh's thesis will not be calculated here.
\end{remark}

The geometric side of the sum above has the form (see equation (19) in that article)
\begin{equation}\label{Venk19}\sum_{n=1}^\infty \sum_c g(\frac{n}{X}) \varphi\left(\frac{4\pi \sqrt{n^2m}}{c}\right) \frac{KS(n^2, m ; c)}{c},
\end{equation}
where $KS$ stands for a certain generalized Kloosterman sum. The functions $\varphi$ and $h$ of the past two formulas correspond to the orbital integrals of the Archimedean test function, and their spectral ``Lebedev--Kontorovitch'' transform in terms of Bessel functions, respectively.

The key in relating Venkatesh's argument to our transfer operator is to identify the variables that appear in the equation above to our local coordinates appearing in \eqref{Venkatesh-final}. For simplicity, we will take $m=1$, which is enough to convey the idea.  It is simplest to start from an adelic setting, produce a classical translation, and then relate it to the manipulations of Venkatesh.

Adelically, we would write the Kuznetsov formula of $\SL_2$, over the base field $\QQ$ as
\begin{equation}\label{KTF-global} \KTF(f_{\eta|\bullet|^s}) = \sum_{\zeta\in \QQ} f_{\eta|\bullet|^s}(\zeta) = \sum_{\zeta\in \QQ} \int_{\mathbb A^\times} f(r,\zeta) \eta(r)|r|^s d^\times r,
\end{equation}
where we are being imprecise about the contribution of the singular orbits corresponding to $\zeta=0$.  The test function $f_{\eta|\bullet|^s}$, here, is obtained from the pushforward of a factorizable $f = \otimes_{p\le \infty} f_p$, with factors in the space $\mathcal S^-_{L(\Sym^2, 1)} (N,\psi\backslash G/N,\psi)$ of the corresponding local field.
The above would only converge for $\Re s\gg 0$, and we would seek to compute the residue at $s=0$ --- see Remark \ref{remarkanalytic} above.  

Our quadratic character $\eta$ would correspond to the Legendre symbol $\left(\frac{D}{\bullet}\right)$ under the correspondence between idele class characters and Dirichlet characters, where $D$ is a square-free integer. 
Writing $\Z_D = \prod_{p|D} \Z_p$, $K_D$ for the product of $\prod_{p \nmid D} \Z_p^\times$ with the subgroup of elements congruent to $1$ in $\Z_D^\times$, our test function $f$ would be taken to be $K_D$-invariant in the variable $r$; in particular, unramified outside of $p |D$. Using the decomposition $\mathbb A^\times/K_D = \QQ^\times \times (\Z/D)^\times \times \RR_+^\times$ (with $(\Z/D)^\times$ understood as a quotient of $\Z_D^\times$), we would rewrite the integral over $\mathbb A^\times$ in \eqref{KTF-global} as follows:
\[ \KTF(f_{\eta|\bullet|^s}) = \sum_{(q,\zeta)\in (\QQ^\times \times \QQ)} \sum_{x\in (\Z/D)^\times}  f_f(q x ,\zeta) \eta_D(x) \cdot \int_{\RR^\times_+} f_\infty(qr, \zeta)  r^s d^\times r,\]
where we have written $f = f_f \otimes f_\infty$ for its restriction to the finite and Archimedean factors of the adeles, $\eta_D$ for the restriction of $\eta$ to $\Z_D^\times$, etc. Writing $f_f^\eta(r, \zeta) = \sum_{x\in (\Z/D)^\times} f_f(rx,\zeta) \eta_D(x)$ where $r$ and $\zeta$ are in the finite ideles (resp.\ adeles), a factorizable test function whose factors depend on $\eta$ only at the primes dividing $D$, and choosing it so that it is only supported on integers at every place, the above sum could be written
\begin{multline} \KTF(f_{\eta|\bullet|^s}) =  \sum_{(n,\zeta)\in (\Z_{>0} \times \QQ)}  f_f^\eta(n,\zeta)  \cdot  f^+_{\infty, \eta|\bullet|^s} (\zeta) \\ + \eta_{\infty}(-1) \sum_{(n,\zeta)\in (\Z_{<0} \times \QQ)}  f_f^\eta(n,\zeta)  \cdot  f^-_{\infty, \eta|\bullet|^s} (\zeta). \label{integraltosum}
\end{multline}
Here, we have written $f^\pm_{\infty, \eta|\bullet|^s}$ for the integral of $f_\infty$ over the positive, resp.\ negative reals. For simplicity, let us assume that $f_\infty$ is supported on positive reals, and drop the exponent.

As can be seen in \cite{KLP}, this expression corresponds to \eqref{Venk19}, for appropriate choices of test functions, up to replacing the cutoff function by the Dirichlet series (as discussed in Remark \ref{remarkanalytic}), and with the substitution $\zeta = \frac{c}{n}$. I point the reader, in particular, to Proposition 3.7 in that paper, which can be translated to $\SL_2$ by using the property that, for a test function $\Phi$ on $\GL_2(\mathbb A_f)$ with central character $\eta$, and for $c\in \Z_{>0} \subset \QQ^\times$ (so that $\eta(c)=1$), we have 
\begin{multline*} \int_{\mathbb A_f^2} \Phi\left( \begin{pmatrix} 1 & x \\ & 1 \end{pmatrix} \begin{pmatrix} & - c^{-2} \\ 1 \end{pmatrix} \begin{pmatrix} 1 & y \\ & 1 \end{pmatrix} \right) \psi^{-1} (n^2 x + y ) dx dy = \\
   \int_{\mathbb A_f^2} \Phi\left( \begin{pmatrix} n^{-1} \\ & n \end{pmatrix} \begin{pmatrix} 1 & x \\ & 1 \end{pmatrix} \begin{pmatrix} &  -n c^{-1} \\ n^{-1} c \end{pmatrix} \begin{pmatrix} 1 & y \\ & 1 \end{pmatrix} \right) \psi^{-1} (x + y ) dx dy.
\end{multline*}

Combining \cite[Proposition 3.7]{KLP} (with parameters $\mathfrak n = 1$, $m_1 = n^2$, $m_2=-1$) with  \eqref{integraltosum}, we see that \eqref{Venk19} can be obtained by taking 
 $f_p(n, \bullet)$, at primes $p$ not dividing $D$,  to be the one corresponding to $\Phi = 1_{\SL_2(\Z_p)}$ (in the equation above). Equivalently, $f_p(n, \bullet)$ corresponds to the orbital integrals of the $\SL_2^2$-Whittaker function of the form $ \Phi_1\otimes \Phi_n$, where $\Phi_a$ is supported on $N(F)\begin{pmatrix} a \\ & a^{-1} \end{pmatrix} \SL_2(\mathbb Z_p)$. (See \S \ref{sstwistedpf} for a discussion of the Kuznetsov formula in terms of a pair of Whittaker functions/measures.) Of course, this is very close to the basic function of $\mathcal S^-_{L(\Sym^2, 1)} (N,\psi\backslash G/N,\psi)$, in the same sense that the Dirichlet series $\sum_n a_{n^2}(f)  n^{-s-1}$ is close to the symmetric square $L$-function at $s+1$ (i.e., their differences do not affect the residue at $s=0$).  At primes dividing $D$, the corresponding test functions of \cite{KLP} depend on $\eta$, as do the functions $f_f^\eta$ in \eqref{integraltosum}. 
 
Having related \eqref{Venk19} to its adelic version \eqref{KTF-global}, let us now confirm that the local transfer operator $\mathcal T_T$ of Theorem \ref{thmVenkatesh} suggests the global manipulations of Venkatesh's paper. In order to do that, we rewrite \eqref{Venkatesh-final}, ignoring scalar factors and treating measures as functions, as an operator
\begin{equation}\label{transform-imprecise} f \mapsto \lim_{s\to 1} \int_r \eta(r) |r|^s\int_x f\left( \frac{x r}{t} , \frac{t}{x} \right)   \psi(x) dx  d^\times r.
\end{equation}

This is not a random change of variables. As we shall see, the inner and the outer integral play a distinct role in the proof\footnote{Hopefully, putting this discussion ahead of the proof will serve as a motivation for the reader!} of Theorem \ref{thmVenkatesh}, and more precisely of Theorem \ref{transfertokappa}, with the inner one corresponding to 
formula \eqref{barUinverse} for the inverse unfolding operator $\bar{\U}^{-1}$, and the other one, a Mellin transform, being (essentially) the composition 
$\mathcal P_\kappa^\circ \circ \mathcal H_X^\circ $ appearing in (\eqref{Tkappa}). It turns out that these two integrals correspond to the two basic steps of Venkatesh's proof; hence, it is not only our final formula, but also its derivation which is very closely mirrored in the manipulations of \cite{Venkatesh}. 

The inner integral of \eqref{transform-imprecise} is a Fourier transform, and suggests applying a Poisson summation formula, globally. Only for the purposes of this subsection, we will be writing $\hat f$ for what is denoted by $\bar{\U}^{-1}$ in the rest of this paper, that is, $\hat f(c,\nu) = \int_x f\left( \frac{x c}{\nu} , \frac{\nu}{x} \right)   \psi(x) dx$. Then, applying the Poisson summation formula formally (without any claim to rigor) would rewrite the Kuznetsov formula of \eqref{KTF-global}  as
\begin{align}\label{KTF-Poisson}  \nonumber \KTF(f_{\eta|\bullet|^s}) &=  \int_{\QQ^\times\backslash\mathbb A^\times} \eta(r)|r|^s \sum_{(n,c)\in (\QQ^\times \times \QQ)}  f(rn,\frac{c}{n})  d^\times r \\
&= \int_{\QQ^\times\backslash\mathbb A^\times} \eta(r)|r|^s \sum_{(\nu,c)\in \QQ^2}  \hat f(rc,\nu)  d^\times r,
\end{align}
where we have used the fact that the Fourier transform of $n\mapsto f(rn, \frac{c}{n})$, \emph{with dual variable denoted $\frac{\nu}{c}$} is 
\[ \int_n f(rn, \frac{c}{n}) \psi(n \frac{\nu}{c}) dn = |\frac{c}{\nu}| \int_x f(\frac{rc x }{\nu}, \frac{\nu}{x}) \psi(x) dx = |\frac{c}{\nu}| \hat f(rc, \nu).\]

The first main step of Venkatesh is, indeed, to apply the Poisson summation formula in the variable $n$, with $c$ fixed. 
This is done in two moves, first over classes of $n$ modulo $c$ in order to take the Fourier transform of the Archimedean component, and then as a summation over classes $x\in \Z/c$, arriving at equation (20) in his paper. It is easily checked that the two steps correspond, indeed, to an adelic Poisson summation formula in the variable $n$, with dual variable corresponding to $\frac{\nu}{c}$. Thus, the last expression of \eqref{KTF-Poisson} corresponds to  \cite[(20)]{Venkatesh}. 

The outer integral in \eqref{transform-imprecise} is already embedded in \eqref{KTF-Poisson}. Classically, it corresponds to a Dirichlet series in the parameter $c$, with the parameter $\nu$ fixed. Studying this Dirichlet series is indeed the second main step  of Venkatesh. This Dirichlet series is defined in \cite[\S 4.5.2]{Venkatesh}, and denoted $Z(s)$. (The reader can also consult the actual thesis, \cite[Theorem 3]{Venkatesh-thesis}, where it is written more simply for the case of $\mathbb Q$.) 

As Venkatesh observes, this Dirichlet series exhibits a lot of cancellation and does not have a pole at $s=0$, unless $\nu$ in the image of the map $T(\mathbb Q)\to T\sslash W (\mathbb Q)$. Locally, however, there is no pole, and the evaluation at $s=0$ corresponds to the integral of $\hat f$ against $\eta$ in the variable $c$. (This point is lost in the translation between measures and functions: this integral looks more like a Tate zeta integral at $s=1$.) Nonetheless, our derivation of \eqref{Venkatesh-final}, the outer integral of \eqref{transform-imprecise} is written as a composition of two more complicated operators,
$\mathcal P_\kappa^\circ \circ \mathcal H_X^\circ $, 
see Theorem \ref{transfertokappa}. One of the operators is essentially a Fourier transform of the function $c\mapsto \eta(c) \hat f(c,\nu)$, and the other is a residue of its Tate zeta integral at $s=0$ (corresponding to $s=-1$ in our current coordinates). The need to write a simple integral in this complicated way arises from the fact that the operator $\mathcal T_T$ is obtained as a ``boundary value'' of certain orbital integrals on the Rankin--Selberg variety --- the boundary value itself interpreted as a $\kappa$-orbital integral for $\SL_2$, and then related to the torus via the endoscopic comparison of Labesse and Langlands.

Having, hopefully, given enough details to the interested reader to relate this paper with the work of Venkatesh, we return to the rigorous discussion in the local setting, in order to prove Theorem \ref{thmVenkatesh}.

\subsection{Germs at zero}

We continue using the notation of the previous two sections. In addition, we will be identifying the torus $A_\ad$ with $\Gm$ via the positive root character, so for a character $\chi$ of $A_\ad$ we may interchangeably write $L(\chi,\frac{\check \alpha}{2},s)$ or $L(\chi,s)$ for its local Dirichlet $L$-function.

The group $\Gm=A_\ad$ acts on measures or functions on $\bar X\sslash \SL_2$ and $N\backslash G\sslash N$, and we can consider either the unnormalized translation action, or the normalized one, descending from \eqref{actionnormagain-measures}, \eqref{actionnormagain-functions}, \eqref{actionnorm-unfolding-meas}, and \eqref{actionnorm-unfolding-fns}; that is, for measures, the normalized action is the product of the unnormalized one with $\delta^{-\frac{1}{2}}$. In this section, we will be working, by default, via the normalized action (which is more convenient for keeping track of various constants), but will sometimes switch to the unnormalized action, explicitly stating so (as in the main theorem, Theorem \ref{thmVenkatesh}). Notice that, under the identification $A_\ad\simeq\Gm$ with the positive root character, $\delta^\frac{1}{2}$ becomes the character $|\bullet|^\frac{1}{2}$; therefore, when (below) we apply the Mellin transform of \eqref{MellinVTdef} to the variable $c$ of the quotient space $\bar X \sslash \SL_2$ (on which $A_\ad$ acts by $\delta$), its argument (a character of $\Gm\simeq A_\ad$) is compatible with the normalized action.

As we have seen in the proof of Theorem \ref{RSHankel}, for an element $f\in \mathcal S(\bar X/\SL_2)$, the quotient $\frac{f}{dt}$ is a section, over $t\ne \pm 2$, of the family of spaces $\mathcal S([V/T_t])$, where $T_t$ is the special orthogonal group of the quadratic form $c$ (the coordinate on $\bar X\sslash \SL_2$) of discriminant $d=-\frac{1}{4}(t^2-4)$ on the symplectic vector space $V$. Hence, as $c\to 0$, it has the asymptotic behavior of \eqref{germs-non-split} or \eqref{germs-split}, where $\xi$ there stands for the coordinate $c$:  
\begin{align*}\frac{f(c,t)}{dt} &= C_{1,t}(c) + C_{2,t}(c) \eta_{t^2-4}(c),\mbox{ if $t^2-4$ is not a square};\\
\frac{f(c,t)}{dt} &= C_{1,t}(c) + C_{2,t}(c) \log(|c|),\mbox{ if $t^2-4$ is a square.}
  \end{align*}
I remind that the $C_{i,t}$'s ($i=1,2$) are smooth measures in the variable $c$.

Let $\mathcal S(\bar X/\SL_2)_0$ denote the collection of functions $t\ne \pm 2\mapsto \frac{C_{2,t}}{dc}(0)$, as $f$ varies over all elements of $\mathcal S(\bar X/\SL_2)$. By definition, we have a surjective map
\begin{equation}\label{zerofiber} \mathcal S(\bar X/\SL_2) \twoheadrightarrow \mathcal S(\bar X/\SL_2)_0,
\end{equation}
that we will denote by $f\mapsto [f]_0$.

For elements of the subspace $\mathcal S(\bar X/\SL_2)^\circ$, the collection of measures $\frac{f}{dt}$ belongs to the subspaces $\mathcal S([V/T_t])^\circ$ of \eqref{germs-circ}; that is, their asymptotics are of the form
$$\frac{f(c,t)}{dt} = C_{t}(c) \eta_{t^2-4}(c).$$ 

Let $\mathcal S(\bar X/\SL_2)^\circ_0$ denote the collection of functions $t\ne \pm 2\mapsto \frac{C_{t}}{dc}(0)$, as $f$ varies over all elements of $\mathcal S(\bar X/\SL_2)^\circ$. By definition, we have a surjective map
\begin{equation}\label{zerofibercirc} \mathcal S(\bar X/\SL_2)^\circ \twoheadrightarrow \mathcal S(\bar X/\SL_2)^\circ_0,
\end{equation}
that we will denote by $f\mapsto [f]^\circ_0$.

The following is immediate:
\begin{lemma}
 The kernels of the maps \eqref{zerofiber}, \eqref{zerofibercirc} are $A_\ad$-stable; hence, $A_\ad$ acts on the spaces $\mathcal S(\bar X/\SL_2)_0$, $\mathcal S(\bar X/\SL_2)^\circ_0$. Under the normalized action on $\mathcal S(\bar X/\SL_2)$, these quotients decompose under the $A_\ad$-action into eigenspaces with eigencharacters $\eta \delta^\frac{1}{2}$, where $\eta$ ranges over all quadratic characters of $A_\ad$. For any such $\eta$, the $\eta \delta^\frac{1}{2}$-eigenvectors are supported on the set of $t\ne \pm 2$ with $\eta_{t^2-4}=\eta$.
\end{lemma}

Notice that under the unnormalized action of $A_\ad$, the eigencharacters become $\eta\delta$.

The quotients \eqref{zerofiber}, \eqref{zerofibercirc} can also be expressed in terms of Mellin transforms.  By Corollary \ref{MellinVcirc} we have:
\begin{equation}\label{germasMellin-circ}
[f]^\circ_0(t)= \frac{\eta_{t^2-4} \cdot f}{dc\, dt}(0,t) = - \AvgVol(F^\times)^{-1}\Res_{s=\frac{1}{2}} \widecheck{\frac{f}{dt}}(\eta_{t^2-4}|\bullet|^s,t)
\end{equation}
on $\mathcal S(\bar X/\SL_2)^\circ$. The Mellin transform, here, is, for any fixed $t\ne \pm 2$, the scalar-valued Mellin transform of \eqref{MellinVTdef} of the quotient $\frac{f}{dt}(\bullet,t)$, understood as a measure in the variable $c$. (In particular, $c=1$ is the base point that we use to identify the set of points $(c,t)$ with $c\ne 0$ and $t$ fixed with the torus $A_\ad$.) We will also be thinking of the product of the function $t\mapsto \widecheck{\frac{f}{dt}}(\chi, t)$ by $dt$ as a measure in the variable $t$, denoted $\check f(\chi)$, and varying meromorphically with $\chi$.

The analog of \eqref{germasMellin-circ} for $[f]_0$ is easiest expressed if we multiply the Mellin transform of an element $f\in \mathcal S(\bar X/\SL_2)$ by $L(\chi\delta^s,-\frac{\check\alpha}{2}, \frac{1}{2})^{-1}= L(\chi^{-1},\frac{1}{2})^{-1}$ (identifying $A_\ad$ with $\Gm$ via the positive root character), so that for $\eta_{t^2-4}=1$  the double pole of the Mellin transform at $\chi=\eta_{t^2-4}|\bullet|^\frac{1}{2}$ becomes simple. Then, by Proposition \ref{VmodTimage}, for $f\in \mathcal S(\bar X/\SL_2)$ we have 
\begin{equation}\label{germasMellin}
[f]_0(t) =  - L^*(\eta_{t^2-4}, -0) \cdot  \frac{\Res_{s=\frac{1}{2}} L(\eta_{t^2-4} , \frac{1}{2}-s)^{-1} \widecheck{\frac{f}{dt}}(\eta_{t^2-4}|\bullet|^s,t)}{\AvgVol(F^\times)},
\end{equation}
where $L^*(\eta_{t^2-4}, -0)$ denotes the leading term of the local Dirichlet $L$-function $L^*(\eta_{t^2-4}, -s)$ at $s=0$. The same notation will be used in what follows for the leading coefficients of other $L$- or $\gamma$-factors.

\begin{proposition} \label{fibercomp}
The map 
\begin{equation}
 [f]_0\mapsto \gamma^*(\eta_{t^2-4},1,\psi)^{-1} [f]_0
\end{equation}
is an $A_\ad$-equivariant isomorphism: 
$$\mathcal S(\bar X/\SL_2)_0 \xrightarrow\sim \mathcal S(\bar X/\SL_2)^\circ_0$$
which, in the non-Archimedean unramified case, maps the image of the basic vector $f_{\bar X}$ to $L(\eta_{t^2-4},1)$ times the image of the basic vector $f_{\bar X}^\circ$.
\end{proposition}

\begin{proof}
 The idea is, roughly, to define a map $\mathcal S(\bar X/\SL_2)\to \mathcal S(\bar X/\SL_2)^\circ$ which acts on Mellin transforms as multiplication by $L(\chi,-\frac{\check\alpha}{2}, \frac{1}{2})^{-1}$. This can be done, literally, in the non-Archimedean case, where $L(\chi,-\frac{\check\alpha}{2}, \frac{1}{2})^{-1}$ is the Mellin transform of an element $h$ in the completed Hecke algebra $\widehat{\mathcal S(A_\ad)}$; the convolution action of $h$ (under the normalizated action of $A_\ad$ on $\mathcal S(\bar X/\SL_2)$) defines a surjection (in fact, bijection):
 $$\mathcal S(\bar X/\SL_2)\xrightarrow\sim \mathcal S(\bar X/\SL_2)^\circ,$$
 which, in particular, satisfies
 $$ [h\star f]^\circ_0 = L^*(\eta_{t^2-4},-0)^{-1} [f]_0 $$
 
 Indeed, the Mellin transforms satisfy
 $$ \widecheck{h\star f}(\chi) = \check h(\chi) \check f(\chi),$$
 and by \eqref{germasMellin-circ}, \eqref{germasMellin} we have
 \begin{align*} [h\star f]^\circ_0(t) &=  - \frac{\Res_{s=\frac{1}{2}} \check h(\eta_{t^2-4}|\bullet|^s) \widecheck{\frac{f}{dt}}(\eta_{t^2-4}|\bullet|^s,t)}{\AvgVol(F^\times)} \\
 & =  - \frac{\Res_{s=\frac{1}{2}}  L(\eta_{t^2-4} , \frac{1}{2}-s)^{-1} \widecheck{\frac{f}{dt}}(\eta_{t^2-4}|\bullet|^s,t)}{\AvgVol(F^\times)} \\
 & = L^*(\eta_{t^2-4}, -0)^{-1} [f]_0(t).
  \end{align*}

At Archimedean places, we need to argue more carefully. Recall from \S \ref{ssSym2} that $\mathcal S(\bar X/\SL_2)^\circ$ is obtained as the push-forward of a space of measures on $\bar X_\ad= [\bar X/\{\pm 1\}]$ whose Mellin transforms are Paley--Wiener sections of a sheaf $\mathscr E_\ad^\circ$ over $\widehat{A_\ad}_\CC$. The bundle $\mathscr E_\ad^\circ$ was obtained from the sheaf $\mathscr E_\ad$ describing the Mellin transforms of elements of $\mathcal S(\bar X_\ad)$ by imposing a condition of simple zeroes at the poles of the $L$-function $L(\chi,-\frac{\check\alpha}{2}, \frac{1}{2})$. In particular, if $h\in \widehat{\mathcal S(A_\ad)}$ is any element of the completed Schwartz algebra (see \S \ref{ssmultipliers}) whose Mellin transform has the \emph{same} order and leading coefficient ($=L^*(\eta,-0)^{-1}$) as $L(\chi,-\frac{\check\alpha}{2}, \frac{1}{2})^{-1}$ at $\chi=\eta \delta^\frac{1}{2}$ for any quadratic character $\eta$, then convolution by $h$ defines a surjective map:
$$ \mathcal S(\bar X_\ad) \twoheadrightarrow \mathcal S(\bar X_\ad)',$$
where the space on the right is the subspace of those elements of $\mathcal S(\bar X_\ad)$ whose Mellin transform in small neighborhoods of the points $\chi=\eta \delta^\frac{1}{2}$, for all $\eta$ quadratic, belongs to the subsheaf $\mathscr E_\ad^\circ$.

By Lemma \ref{globalsections} the sheaf $\mathscr E_\ad^\circ$ is generated by its global (Paley--Wiener) sections. 
That means, in particular, that for every element $\Phi'\in \mathcal S(\bar X_\ad)'$ we can find an element $\Phi^\circ$ of (its subspace) $\mathcal S(\bar X_\ad)^\circ$ so that the Mellin transforms of $\Phi'$ and $\Phi^\circ$, as elements of the fiber of $\mathscr E_\ad^\circ$, coincide at $\chi=\eta \delta^\frac{1}{2}$ with $\eta$ quadratic, and vice versa. Let $f', f^\circ$ be the corresponding push-forwards in $\mathcal S(\bar X/\SL_2)$. As Corollary \ref{MellinVcirc} shows, the class $[f^\circ]_0^\circ$ depends only on the values of $\check \Phi^\circ(\chi)$ at $\chi=\eta \delta^\frac{1}{2}$, with $\eta$ ranging over all quadratic quaracters. Thus, extending the definition of $[\,\,]_0^\circ$ to the space $\mathcal S(\bar X_\ad)'$, and taking $f'= h\star f$, for any $f\in \mathcal S(\bar X/\SL_2)$, we deduce that the map
$$ [f]_0\to [f']^\circ_0= [f^\circ]^\circ_0$$
is an isomorphism
$$\mathcal S(\bar X/\SL_2)_0 \xrightarrow\sim \mathcal S(\bar X/\SL_2)^\circ_0,$$
and given as in the non-Archimedean case by multiplication by $L^*(\eta_{t^2-4},-0)^{-1}$.

Multiplying by the non-zero factor $\frac{L(\eta_{t^2-4},1)}{\epsilon(\eta_{t^2-4},1,\psi)}$, we obtain the isomorphism of the proposition.
 
The statement on basic vectors follows from the relation \eqref{Mellinbasic} between their Mellin transforms.
 
\end{proof}

\subsection{$\kappa$-orbital integrals}

Now we will interpret $\mathcal S(\bar X/\SL_2)_0$ in terms of \emph{$\kappa$-orbital integrals} for $\SL_2$. These are the usual orbital integrals for $\SL_2$ over a stable conjugacy class, twisted by a character $\kappa$ of the Galois cohomology group which parametrizes conjugacy classes in the given stable class, but since here there is at most one non-trivial character $\kappa$, we will use this letter just as a symbol for the following:

Let $\eta$ be a quadratic character, $T\subset \SL_2$ a torus associated to this quadratic character (in the sense that its splitting field is the quadratic field associated to $\eta$, or $\eta=1$ and $T$ is split), and let $\C_\eta$ be the image of $T(F)$ in $\C:=\Dfrac{\SL_2}{\SL_2}$; that is, $\C_\eta$ consists of the closure of the set of $t\ne \pm 2$ such $\eta_{t^2-4}=\eta$. Notice that here we are not using the notation $\C_\eta$ for an algebraic variety, but for a subset of $\C(F)$.

We let $\mathcal S^\kappa_\eta(\frac{\SL_2}{\SL_2})\subset \Meas(\C_\eta)$ denote the space of ``$\kappa$-twisted push-forward measures'', that is, the image of a twisted push-forward map
$$ \tr_!^{\eta,\kappa}: \mathcal S(\SL_2) \to \Meas(\C_\eta)$$
defined as follows: \label{defkappaorbital}

\begin{itemize}
 \item If $\eta=1$, then this is just the restriction of the push-forward map $\tr_!: \mathcal S(\SL_2) \to \Meas(\Dfrac{\SL_2}{\SL_2})$ to $\C_\eta$; its image $\mathcal S^\kappa_\eta(\frac{\SL_2}{\SL_2})$ consists of the restrictions of elements of $\mathcal S(\frac{\SL_2}{\SL_2})$ to $\C_\eta$, the closure of the set of hyperbolic conjugacy classes. 
 \item If $\eta\ne 1$, then every $t\ne \pm 2$ in $\C_\eta$ has two conjugacy classes in its preimage, and every class is represented by an element of the form $\begin{pmatrix} t & -c^{-1} \\ c \end{pmatrix}$.  We let $F_\eta$ be the function which is $+1$ on the conjugacy class where $\eta(c)=1$ (that is, where $c$ is a norm for the associated quadratic field) and $-1$ on the conjugacy class where $\eta(c)=-1$. Then we let $\tr_!^{\eta,\kappa} \varphi$ be the push-forward to $\C_\eta$ of the (restriction to the preimage of $\C_\eta$ of the) measure $\varphi\cdot F_\eta$.  
\end{itemize}

Notice that, by using the coordinate $c$ of $\Dfrac{\SL_2}{N}$ we have trivialized the torsor of conjugacy classes inside of any regular stable conjugacy class. This is the usual trivialization by the rational canonical form. We let $A_\ad$ act on the space $\mathcal S^\kappa_\eta(\frac{\SL_2}{\SL_2})$ by the character $\eta$.

\begin{proposition} \label{tokappa}
The map 
\begin{align}\nonumber \mathcal P_\kappa: f\mapsto [f]_0(t)\cdot \frac{\lambda(\eta_{t^2-4},\psi) }{\gamma^*(\eta_{t^2-4}, 1,\psi)} |t^2-4|^{\frac{1}{2}}dt  \\ 
=L^*(\eta_{t^2-4}, -0) \frac{|t^2-4|^{\frac{1}{2}}}{\gamma^*_{E_t}(1,\psi\circ \tr)} \Res_{s=\frac{1}{2}} L(\eta_{t^2-4} , \frac{1}{2}-s)^{-1} \check f(\eta_{t^2-4}|\bullet|^s,t) \label{Pkappa}
\end{align}
defines an $A_\ad$-equivariant isomorphism
$$\mathcal S(\bar X/\SL_2)_0 \xrightarrow\sim \bigoplus_\eta \mathcal S^\kappa_\eta(\frac{\SL_2}{\SL_2}) \otimes \delta^{\frac{1}{2}}.$$

Here, $\gamma^*(\eta_{t^2-4}, 1,\psi)$ denotes the leading term in the Laurent expansion of the function $\gamma(\eta_{t^2-4}, 1+s,\psi)$ at $s=0$, and $\gamma^*_{E_t}(1,\psi\circ \tr)$ is, similarly, the leading term for the local Dedekind gamma factor of $E_t=F(\sqrt{t^2-4})$ at $1$.

For $F$ non-Archimedean and unramified over the base field, with the symplectic form integral and residually non-vanishing, it sends the basic vector $f_{\bar X}\in \mathcal S(\bar X/\SL_2)$ (the image of the basic vector $1_{\bar X(\mathfrak o)} dx$, with $dx(\bar X(\mathfrak o))=1$) to the image of the unit element of the Hecke algebra of $\SL_2$ in $\bigoplus_\eta \mathcal S^\kappa_\eta(\frac{\SL_2}{\SL_2}) \otimes \delta^{\frac{1}{2}}$, and, more generally, for any $h$ in the unramified Hecke algebra of $\tilde G$, sends $h\cdot f_{\bar X}$ to the image of $\diag^*h$, where $\diag^*$ is the map of Hecke algebras $\mathcal H(\tilde G,\tilde K)\to \mathcal H(\SL_2, \SL_2(\mathfrak o))$ which under Satake isomorphism corresponds to pullback under the diagonal embedding of dual groups
$$  \PGL_2 \overset{\diag}{\hookrightarrow} \check{\tilde G}.$$

\end{proposition}

\begin{proof}
 Recall that $\bar X \simeq V\times \SL(V)$. Fix a Haar measure $dg$ on $\SL(V)$, and let $dx$ be the measure on $\bar X$ which is its product with the measure $|\omega|$ defined by the symplectic form on $V$. Write $f=\Phi dx$ for an element of $\mathcal S(\bar X)$, where $\Phi\in \mathcal F(\bar X)$. The product $\Phi |\omega|$, restricted to the fiber $\bar X_g$ of the map $\bar X\to \SL(V)$ over a point $g$, is a Schwartz measure on that fiber, while the product $\Phi|_Y dg$ is a Schwartz measure on the zero section $Y=\SL(V)$. The restriction map $\Phi dx \to \Phi|_Y dg$ is an equivariant map
 $$ \mathcal S(\bar X) \to \mathcal S(\SL(V)) \otimes \delta^\frac{1}{2}$$
 which does not depend on the choice of $dg$. The map $\mathcal P_\kappa$ of the proposition will descend from that; we just need to compute the (surjective) composition of the arrows 
 $$ \xymatrix{ \mathcal S(\bar X) \ar[r] & \mathcal S(\SL_2) \otimes \delta^\frac{1}{2} \ar[r] & \bigoplus_\eta \mathcal S^\kappa_\eta(\frac{\SL_2}{\SL_2}) \otimes \delta^{\frac{1}{2}},
 }$$
 showing that it factors through the injective map from $\mathcal S(\bar X/\SL_2)_0$ given by \eqref{Pkappa}.
 
 Assume that $t=\tr(g)\ne \pm 2$, so that the coordinate $c$ of $\bar X\sslash \SL(V)$ defines a non-degenerate quadratic form on $\bar X_g$. As we saw in the proof of Theorem \ref{RSHankel}, the discriminant of this quadratic form with respect to the symplectic form is $d=-\frac{1}{4}(t^2-4)$. If we identify $\bar X_g$ with $\Res_{E_t/F}\Ga$, where $E_t= F(\sqrt{t^2-4})$ is the corresponding quadratic extension (possibly split) in such a way that the quadratic form $c$ corresponds to its norm, then I claim that the average volume of $E_t^\times$ (see \eqref{avgvol}) with respect to the measure $d^\times e = \frac{|\omega|}{|c|}$ is 
 \begin{equation}\label{AvgVol}
  \AvgVol(E_t^\times) = \begin{cases} |t^2-4|^{-\frac{1}{2}} \Res_{s=0} \gamma_E(1-s,\psi\circ \tr),& \mbox{ in the non-split case}; \\
                         |t^2-4|^{-\frac{1}{2}} \left(\Res_{s=0} \gamma(1-s,\psi)\right)^2,&\mbox{ in the split case.}
                        \end{cases}
 \end{equation}

 Indeed, let us compare the given volume of $E_t$ with the volume obtained by identifying $E_t=F\oplus F$ by means of a dual basis for the trace pairing $(x,y)\mapsto \tr(xy)$ which, by \cite[(2.26)]{SaBE2} is equal to 
 $$\Res_{s=0} \gamma_{E_t}(1-s,\psi\circ \tr)$$ 
 in the non-split case, where $\gamma_{E_t}(1-s,\psi\circ \tr)$ is the gamma-factor of the functional equation of the Tate zeta integral for the trivial character on $E_t$, with respect to a self-dual measure. It decomposes as in \eqref{lambda}
 $$ \gamma_{E_t}(1-s,\psi\circ \tr) = \lambda(\eta_{t^2-4},\psi)^{-1} \gamma(1-s,\psi) \gamma(\eta_{t^2-4}, 1-s,\psi),$$
 with $\Res_{s=0}\gamma(1-s,\psi)$ being, by the same result, the average volume of $F^\times$.
 In the split case, the average volume with respect to the self-dual measure for $\psi\circ \tr$ will be, by the same result, $\left(\Res_{s=0} \gamma(1-s,\psi)\right)^2$. 
 
 Now, as in the proof of  Proposition \ref{H02D}, the trace pairing can be related to the symplectic form by $\omega(x,y) = \tr(Dx\bar y)$, where $\bar y$ is the Galois conjugate of $y$ and $D$ is a traceless element such that the discriminant of the quadratic form $c$ is $-\frac{1}{4D^2}$, and given this we can compute that the average volume of $E_t^\times$ with respect to the symplectic volume is $\sqrt{|D^2|_F} = |t^2-4|^{-\frac{1}{2}}$ times the average volume with respect to the self-dual measure for the character $\psi\circ \tr$. This proves \eqref{AvgVol}.

 Now we apply Corollary \ref{corresidues}. Disintegrating the measure $dg$ against $dt = d(\tr)$, we obtain measures $d\dot g_t$ on the stable conjugacy classes. Let $f_t$ be the push-forward of the measure $\Phi d \dot g_t$ (for $t\ne \pm 2$) with respect to the function $c$. Let $a_t$ denote the leading coefficient of the Laurent expansion of $\check f_t(\chi \cdot |\bullet|^s)$ at $\chi=\eta_{t^2-4}$. Applying Corollary \ref{corresidues}, we have  
 $$a_t = \pm \AvgVol(E_t^\times) \cdot O_t^\kappa(\Phi|_Y),$$ 
 where $\Phi|_Y$ is the restriction of $\Phi$ to the zero section $Y\simeq \SL_2\subset \bar X$, and $O_t^\kappa(\Phi|_Y)$ is its ``$\kappa$-orbital integral'' of \eqref{residue-non-split-group}, taken against the measure $d\dot g_t$. The sign is positive in the split case (where the pole of $\check f_t(\chi \cdot |\bullet|^s)$ is of order two), and negative in the non-split case (where it is of order one).
 
 The product $a_t dt$ is thus equal to $\AvgVol(E_t^\times)$ times the image of $\Phi|_Y dg$ in $\bigoplus_\eta \mathcal S^\kappa_\eta(\frac{\SL_2}{\SL_2})$, up to a sign that depends on $t$. By Proposition \ref{VmodTimage}, the image of $\Phi dx$ in $\mathcal S(\bar X/\SL_2)_0$ is equal to 
\[[\Phi dx]_0(t) = \frac{a_t}{-\AvgVol (F^\times)} = \mp \frac{\AvgVol(E_t^\times) \cdot O_t^\kappa(\Phi|_Y)}{\AvgVol (F^\times)},\] 
 hence the image of $\Phi|_Y dg$ in $\bigoplus_\eta \mathcal S^\kappa_\eta(\frac{\SL_2}{\SL_2})$ is equal to:
  \begin{align*} [\Phi dx]_0(t) \cdot \frac{\mp \AvgVol(F^\times)}{\AvgVol(E_t^\times)} dt &= [\Phi dx]_0(t) \cdot \frac{\mp \Res_{s=0}\gamma(1-s,\psi)}{ |t^2-4|^{-\frac{1}{2}}{\gamma_{E_t}}^*_{s=0}(1-s,\psi\circ \tr)} dt \\
   & = [\Phi dx]_0(t) \cdot \frac{\lambda(\eta_{t^2-4},\psi) |t^2-4|^{\frac{1}{2}}}{\gamma^*(\eta_{t^2-4}, 1,\psi)} dt,
  \end{align*}
  where the signs on the first line are negative in the split case and positive in the non-split case, ${\gamma_{E_t}}^*_{s=0}$ denotes the leading term of the Laurent expansion at $s=0$. The second equality of the formula \eqref{Pkappa} follows from \eqref{germasMellin}. 
 
The normalized action of $A_\ad$  multiplies the measure on $V_\ad$ by $\delta^\frac{1}{2}$ and shifts the base point in the definition of $\tr_!^{\eta,\kappa}$; thus, it acts on the projection to $\mathcal S^\kappa_\eta(\frac{\SL_2}{\SL_2})$ by $\eta \delta^{\frac{1}{2}}$.

For the statement on basic vectors, take $\Phi = 1_{\bar X(\mathfrak o)}$, and assume that $dg(\SL_2(\mathfrak o))=1$. For unramified data, i.e., the symplectic form on $V$ being integral and residually non-vanishing, this means that $dx(\bar X(\mathfrak o))=1$, so the product $\Phi dx$ maps to the basic vector $f_{\bar X}$. Then, by the above, the measure $$[\Phi dx]_0(t) \cdot \frac{\lambda(\eta_{t^2-4},\psi) |t^2-4|^{\frac{1}{2}}}{\gamma^*(\eta_{t^2-4}, 1,\psi)} dt$$
is equal to the image of the unit element of the Hecke algebra in $\bigoplus_\eta \mathcal S^\kappa_\eta(\frac{\SL_2}{\SL_2}) \otimes \delta^{\frac{1}{2}}$ and, more generally, for any element $h$ in the Hecke algebra of $\tilde G$, the image of $h\cdot 1_{\bar X(\mathfrak o)} dx$ under \eqref{Pkappa} is equal to the image of the element
$(h\cdot 1_{\bar X(\mathfrak o)})|_{\SL_2} dg$, where by restriction to $\SL_2$ we mean restriction to the zero section $\bar X\smallsetminus X$. This is the image of $h$ under the push-forward map under $m: \tilde G \to \SL_2$ induced by $A\times\SL_2^2\ni (a,g_1,g_2)\mapsto g_1^{-1}g_2$, and by Lemma \ref{convolution}, this coincides with the restriction of the Satake transform of $h$ to the diagonal $\PGL_2\hookrightarrow \check{\tilde G}$.
\end{proof}

Now we will define a similar map 
$$ \mathcal S(\bar X/\SL_2)_0^\circ \xrightarrow{\sim} \bigoplus_\eta \mathcal S^\kappa_\eta(\frac{\SL_2}{\SL_2}) \otimes \delta^\frac{1}{2}.$$
It will \emph{not} be the map that is induced from $\mathcal P_\kappa$ under the embedding $\mathcal S(\bar X/\SL_2)^\circ \hookrightarrow \mathcal S(\bar X/\SL_2)$; such a map would be zero on hyperbolic conjugacy classes. Rather, it will be compatible with the map of Proposition \ref{fibercomp} between germs at zero, and such that the following diagram commutes:
\begin{equation}\label{tokappa-all} \xymatrix{ \mathcal S(\bar X/\SL_2)^\circ \ar@{<-->}[r]\ar@{<->}[d]^{\mathcal H_X^\circ}\ar@/^2pc/[rr]^{\mathcal P_\kappa^\circ} & \mathcal S(\bar X/\SL_2) \ar@{<->}[d]^{\mathcal H_X}\ar[r]^{\mathcal P_\kappa} & \bigoplus_\eta \mathcal S^\kappa_\eta(\frac{\SL_2}{\SL_2}) \otimes \delta^{\frac{1}{2}}, \\
\mathcal S(\bar X/\SL_2)^\circ \ar@{^{(}->}[r] & \mathcal S(\bar X/\SL_2) 
}\end{equation}
although $\mathcal H_X^\circ$ and $\mathcal H_X$ do not commute with the embedding of $\mathcal S(\bar X/\SL_2)^\circ$ in $\mathcal S(\bar X/\SL_2)$. The dotted arrow in this diagram does not represent an actual map between the spaces, but the isomorphism of Proposition \ref{fibercomp} between their germs at zero. Recalling the relation \eqref{relationbetweenH-Mellin} between the operators, we have:

\begin{theorem}\label{transfertokappa}
 The map 
 \begin{align}\nonumber \mathcal P_\kappa^\circ: f\mapsto & \left[f\right]^\circ_0(t) \cdot \lambda(\eta_{t^2-4},\psi) |t^2-4|^{\frac{1}{2}} dt \\
 & = \frac{\lambda(\eta_{t^2-4},\psi) |t^2-4|^{\frac{1}{2}}}{\gamma^*(1,\psi)} \Res_{s=\frac{1}{2}}  \check f(\eta_{t^2-4}|\bullet|^s,t)\label{Pkappacirc}
 \end{align}
 is an $A_\ad$-equivariant (for the normalized action) isomorphism:
 \begin{equation}\label{circisom}\mathcal S(\bar X/\SL_2)_0^\circ \xrightarrow{\sim}\bigoplus_\eta \mathcal S^\kappa_\eta(\frac{\SL_2}{\SL_2}) \otimes \delta^\frac{1}{2}
 \end{equation}
 which makes the diagram \eqref{tokappa-all} commute. 
 
 Its restriction to the $\eta$-summand satisfies the fundamental lemma for the Hecke algebra, for $F$ non-Archimedean and unramified over the base field, mapping the basic vector $f_{\bar X}^\circ$ to $L(\eta,1)^{-1}$ times the image of the unit element in the Hecke algebra of $\SL_2$, and similarly for its Hecke translates, as in Proposition \ref{tokappa}.
 
 Moreover, the transfer operator 
 \begin{equation}\label{Tkappa}\mathcal T_\kappa:=  \mathcal P_\kappa^\circ \circ \mathcal H_X^\circ \circ \bar{\U}^{-1}
 \end{equation}
 is an $A_\ad$-\emph{anti-equivariant} surjection (for the normalized action):
 $$\mathcal S^-_{L(\Sym^2,1)} (N,\psi\backslash G/N,\psi) \twoheadrightarrow \bigoplus_\eta \mathcal S^\kappa_\eta(\frac{\SL_2}{\SL_2}) \otimes \delta^\frac{1}{2}.$$
 The projection to the $\eta$-summand is an $(A_\ad,\eta\delta^{-\frac{1}{2}})$-equivariant quotient, given by the formula: 
\begin{equation}\label{tokappa-simple}
 \mathcal T_\kappa f(t) =  \int_c \int_x f \left( c, \frac{t}{x} \right) \eta(xct) \psi(x) dx 
\end{equation}
(when $t$ is such that $\eta_{t^2-4} = \eta$).

For any element $h$ of the unramified Hecke algebra of $\tilde G$, the image of $h\cdot f_{L(\Sym^2,1)}$ under $\mathcal T_\kappa$ in $\mathcal S^\kappa_\eta(\frac{\SL_2}{\SL_2}) \otimes \delta^\frac{1}{2}$ is equal to the image of $\frac{\zeta(1)\zeta(2)}{L(\eta,1)}\diag^* h$, where $\diag^*$ denotes the map of Hecke algebras defined in Proposition \ref{tokappa}. 

\end{theorem}

\begin{remarks} \begin{enumerate}
                 \item  For the \emph{unnormalized} action of $A_\ad$ the projection to the $\eta$-summand is $\eta$-equivariant, exactly as it should be for the functorial lift corresponding to the representation $\bar r$ of $L$-groups (as in the beginning of this section).
                 \item The theorem will be complemented with a statement on relative characters, see Proposition \ref{transfertokappa-characters} below.
                \end{enumerate}
\end{remarks}

\begin{proof}
 This is just a combination of results already proven: 
 
 By Proposition \ref{fibercomp}, division by $\gamma^*(\eta_{t^2-4},1,\psi)$ induces an isomorphism of quotient spaces:
 $$ \mathcal S(\bar X/\SL_2)^\circ_0 \to \mathcal S(\bar X/\SL_2)_0,$$
 and Proposition \ref{tokappa} proves the isomorphism \eqref{circisom}. 
 
  The second expression in \eqref{Pkappacirc}, in terms of a residue, follows from \eqref{germasMellin-circ}, using the equality 
  \[-\AvgVol(F^\times) = - \Res_{s=0}\gamma(1-s,\psi) = \gamma^*(1,\psi).\]

 For the commutativity of \eqref{tokappa-all}, apply the relation \eqref{relationbetweenH-Mellin} between the operators $\mathcal H_X$ and $\mathcal H_X^\circ$ to the formulas \eqref{Pkappa}, \eqref{Pkappacirc}.

 By Theorem \ref{thmsubspaceX} and the definition of $\mathcal S^-_{L(\Sym^2,1)} (N,\psi\backslash G/N,\psi)$, the map $\mathcal H_X^\circ \circ \bar{\U}^{-1}$ is an $A_\ad$-anti-equivariant isomorphism (with respect to the normalized action of $A_\ad$): 
 $$ \mathcal S^-_{L(\Sym^2,1)} (N,\psi\backslash G/N,\psi) \xrightarrow\sim \mathcal S(\bar X/\SL_2)^\circ.$$

 For the statement on basic vectors, we see from Theorems \ref{thmbasicvector} and \ref{thmbasicvectorX} that $\mathcal H_X^\circ\circ \bar{\U}^{-1}$ takes the basic vector $f_{L(\Sym^2,1)}$ to $\zeta(1)\zeta(2) f_{\bar X}^\circ$; moreover, since it descends from a $(\tilde G,\iota)$-equivariant transform (where $\iota$ is the involution on $\tilde G$ induced by inversion on $A$), it will also send $h\cdot f_{L(\Sym^2,1)}$ to $\zeta(1)\zeta(2) (\iota^*h)\cdot f_{\bar X}^\circ$, for every $h\in \mathcal H(\tilde G, \tilde K)$. Under the correspondence of Proposition \ref{fibercomp} between stalks at zero,  $f_{\bar X}^\circ$ corresponds to $L(\eta, 1)^{-1}$ times the basic vector $f_{\bar X}$, and by Proposition \ref{tokappa} and the commutativity of \eqref{tokappa-all}, it is mapped under $\mathcal P_\kappa^\circ$ to $\frac{\zeta(1)\zeta(2)}{L(\eta, 1)}$ times the basic vector of $\mathcal S^\kappa_\eta(\frac{\SL_2}{\SL_2})$. More generally, the operator $\mathcal P_\kappa^\circ$ carries $h\cdot f_{L(\Sym^2,1)}$ to the image of $\frac{\zeta(1)\zeta(2)}{L(\eta,1)} \diag^* h$ in $\bigoplus_\eta \mathcal S^\kappa_\eta(\frac{\SL_2}{\SL_2}) \otimes \delta^\frac{1}{2}$.

We can write an explicit formula for $\mathcal T_\kappa$, using the formulas \eqref{barUinverse}, \eqref{HankelXcirc-measures} for $\bar{\U}^{-1}$ and $\mathcal H_X^\circ$. Fix a quadratic character $\eta$ (possibly trivial), and assume that $t$ is such that $\eta=\eta_{t^2-4}$; then
\begin{multline*} \mathcal T_\kappa f (t) =  \mathcal P_\kappa^\circ \circ \mathcal H_X^\circ \circ \bar{\U}^{-1} f (t)=  \frac{\lambda(\eta,\psi) |t^2-4|^{\frac{1}{2}}}{\gamma^*(1,\psi)} \cdot \\
 \cdot \Res_{s=\frac{1}{2}} \left( \lambda(\eta,\psi)^{-1}  \left|\frac{c^2}{4-t^2}\right|^{\frac{1}{2}} \left( (\psi(\frac{1}{\bullet}) \eta(\bullet) d^\times\bullet) \star_1  \bar\U^{-1}f \right)((4-t^2)c^{-1}, t)\right)^\vee (\eta|\bullet|^s) \\
=\frac{1}{\gamma^*(1,\psi)} \cdot \Res_{s=0} 
\left( \left( (\psi(\frac{1}{\bullet}) \eta(\bullet) d^\times\bullet) \star_1  \bar\U^{-1}f \right)((4-t^2)c^{-1}, t)\right)^\vee (\eta|\bullet|^{s-\frac{1}{2}}).
\end{multline*}
Here, $\star_1 $ denotes multiplicative convolution in the first argument of the function $\bar\U^{-1}f$, but this function is then evaluated at $((4-t^2)c^{-1}, t)$, so eventually we will are looking at the Mellin transform of the function 
\[ (\psi(\frac{1}{\bullet}) \eta(\bullet) d^\times\bullet) \star_1  \bar\U^{-1}f\] at the inverse character. The variable $c$ is a ``dummy variable'' that serves as the argument of the Mellin transform.
Explicitly, this reads (recalling from \eqref{MellinVTdef} the normalization of Mellin transform):
$$\frac{\Res_{s=0} \int_c 
(\psi(\frac{1}{\bullet}) \eta(\bullet) d^\times\bullet) \star_1  \bar\U^{-1}f ((4-t^2)c^{-1}, t) \eta(c) |c|^{-s}}{\gamma^*(1,1,\psi)}=$$ 
$$ = \frac{\Res_{s=0} \int_c 
(\psi(\frac{1}{\bullet}) \eta(\bullet) d^\times\bullet) \star_1  \bar\U^{-1}f (c, t) \eta(c) |c|^s}{\gamma^*(1,1,\psi)}.$$

We can use the definiton of $\bar{\U}^{-1}$ to expand this completely to
\begin{align}\mathcal T_\kappa f(t) &= \frac{1}{\gamma^*(1,1,\psi)} \cdot \Res_{s=0}  \nonumber \\&
\int_c \int_z \int_x f\left(\frac{zx(4-t^2)}{ct}, \frac{t}{x}\right) \eta(c) |c|^{-s} \eta(zc) \psi(x+z) dx d^\times z  \nonumber \\
\label{explicit-tokappa}  & = \frac{\Res_{s=0} \int_c \int_z \int_x f\left(zc, \frac{t}{x}\right) |c|^s \eta(xzct) \psi(x+z) dx d^\times z}{\gamma^*(1,1,\psi)}
\end{align}
(given that $\eta(4-t^2)=1$),
or use the functional equation of Tate zeta integrals \eqref{FE}, to get
\begin{multline*} \frac{1}{\gamma^*(1,\psi)} \cdot \Res_{s=0} 
\gamma(1+s,\psi) \int_c \bar\U^{-1}f (c, t) \eta(c) |c|^s =   \int_c \bar\U^{-1}f (c, t) \eta(c) ,
 \end{multline*}
which, using the formula for $\bar{\U}^{-1}$, becomes
$$\int_c \int_x f \left( \frac{xc}{t}, \frac{t}{x} \right) \psi(x) dx \eta(c).$$
This proves \eqref{tokappa-simple}.

\end{proof}

\begin{remark}
 The simplicity of \eqref{tokappa-simple} is best understood through the commutativity of \eqref{tokappa-all}, and the fact that the map from the bottom-right $\mathcal S(\bar X/\SL_2)$ to $\mathcal S_\eta^\kappa(\frac{\SL_2}{\SL_2})\otimes\delta^\frac{1}{2}$ is essentially a twisted push-forward; see the proof of Proposition \ref{transfertokappa-characters}.
\end{remark}

\subsection{Transfer to tori} \label{sstransfer-tori}

We now combine this with the Labesse--Langlands theory of endoscopy \cite{Langlands-basechange, LL} for $\SL_2$, to get the transfer operator from the Kuznetsov formula to a torus. First, let me summarize the results of Langlands and Labesse:

We fix, as in Theorem \ref{thmVenkatesh}, a quadratic character $\eta$, the corresponding quadratic algebra $E$ over $F$, and the torus $T$,  the kernel of the norm map $\Res_{E/F}\Gm \to \Gm$, considered as a subgroup of $\SL_2\simeq \SL_F(\Res_{E/F}\Ga)$. In the non-Archimedean, unramified case, $F$ will be replaced by its ring of integers $\mathfrak o$ in these definitions. The $L$-group of $T$ has an $L$-embedding 
\begin{equation} \bar r: {^LT} \simeq \check T\rtimes \Gal(\bar F/F) \hookrightarrow {^LG}=\Gm \times \PGL_2 \times \Gal(\bar F/F),
\end{equation}
as in the beginning of this section, whose projection to the $\PGL_2$-factor factors through the Galois group of $E/F$, sends $\check T$ to the diagonal torus, and the non-trivial element of $\Gal(E/F)$ to the Weyl element $\begin{pmatrix} & 1 \\ 1 \end{pmatrix}$, if $T$ is non-split.

We let $\ell$ denote the isomorphism
$$ \ell: \mathcal S(T)_W =\mathcal S(T/W) \xrightarrow\sim \mathcal S(T)^W$$ whose inverse is \emph{$\frac{1}{2}$ times} the composition of $\mathcal S(T)^W \hookrightarrow \mathcal S(T)\to \mathcal S(T/W)$. In other words,  \emph{$ \ell$ preserves the measure on every measurable subset of $T$ on which the map to $T\sslash W$ is injective}. Explicitly, there is a Haar measure $d a$ on $T$ which, locally (that is, on subsets which injecto to $T\sslash W$), satisfies $d a = \frac{dt}{\sqrt{|t^2-4|}}$, and if we write an element $f\in \mathcal S(T/W)$ as $\Phi(t) dt$, we have
\begin{equation}
 \ell(f)(a) = \sqrt{|t(a)^2-4|} \Phi(t(a)) d a.
\end{equation}

\begin{theorem}\label{LLtheorem}
There is an isomorphism
$$ \mathcal T_{LL}:\mathcal S^\kappa_\eta(\frac{\SL_2}{\SL_2}) \to \mathcal S(T)^W$$
given by
$$ f \mapsto \ell\left(\frac{\lambda(E/F,\psi)}{\sqrt{|t^2-4|}} f\right),$$
such that for every unitary\footnote{``Unitary'' is automatic for $\eta\ne 1$, but the theorem includes the case of split tori.} character $\theta$ of $T$ we have
$$ \mathcal T_{LL}^* \theta = \Theta_\theta,$$
where:
\begin{itemize}
 \item in the split case, $\Theta_\theta$ is the character of the principal series representation obtained by unitary parabolic induction from $\theta$;
 \item in the non-split case, there is an $L$-packet (=restriction of an irreducible representation of $\GL_2$) of $\SL_2$, consisting of two (possibly reducible)\footnote{The representations $\pi_+, \pi_-$ are reducible if and only if $\theta=\theta^{-1}$ --- in which case the packet consists of $4$ irreducible elements. However, at no point will we need to distinguish between the irreducible components of $\pi_+$ and $\pi_-$.} representations $\pi_+$ and $\pi_-$ labelled so that $\pi_+$ admits a Whittaker model for the character $\begin{pmatrix} 1 & x \\ & 1  \end{pmatrix} \mapsto \psi^{-1}(x) $, such that $\Theta_\theta$ is equal to the difference of characters 
$$ \Theta_{\pi_+} - \Theta_{\pi_-}.$$ 
\end{itemize}

For $E/F$ unramified non-Archimedean, $\mathcal T_{LL}$ takes the basic vector of $\mathcal S^\kappa_\eta(\frac{\SL_2}{\SL_2})$ (the image of the unit element of the Hecke algebra) to the basic vector of $\mathcal S(T)$ (the unit of the Hecke algebra). Moreover, for every $h$ in the unramified Hecke algebra of $\SL_2$, it takes the image of $h$ to $\bar r^*h$, where $\bar r^*$, under Satake isomorphism, corresponds to pullback under the $L$-embedding $\bar r$.\end{theorem}

\begin{remark} 
The unramified Hecke algebra of $T$ is trivial ($\simeq \CC$) if $T$ is non-split; in that case, $\bar r^*h$ is the trace of $h$ on the principal series representation of $\SL_2$ induced from the unique non-trivial unramified character of its Cartan; this is the same as the integral $\int_{\SL_2} h F_\eta$, where $F_\eta$ is as in the definition of $\mathcal S^\kappa_\eta(\frac{\SL_2}{\SL_2})$ in \S \ref{defkappaorbital}.
\end{remark}

This formulation of the results of Labesse and Langlands  is certainly well-known to experts, but for the convenience of the reader I will outline its reduction to the statements of the original papers.

\begin{proof}
To compare with the results of \cite{Langlands-basechange, LL}, first we need to relate our trivialization of $\kappa$-orbital integrals (along the rational canonical  form $\begin{pmatrix} & -1 \\ 1 & t\end{pmatrix}$) with that used in \cite[Lemma 2.1]{LL}; the significant factor is the factor 
\begin{equation}\label{kappafactor} \kappa'\left(\frac{\gamma-\gamma}{\gamma_1^0 - \gamma_2^0}\right)
\end{equation}
of \emph{loc.cit.}, where $\kappa'$ is what we denote here by $\eta$, and the expression in brackets can be identified with 
$$ \frac{e-\bar e}{e^0 - \overline{e^0}},$$
where $e, e^0 \in T(F)\subset E^\times$, and $e^0$ fixed. The same factor appears in the definition of a certain function $\chi_\theta$ on $T$ in \cite[Lemma 7.19]{Langlands-basechange}, except that $\kappa'$ (our $\eta$) was there denoted by $\omega$. 

These factors depend on the choice of a base point $e^0\in T(F)$; it is easy to see that, if we choose the isomorphism of $\SL_F(\Res_{E/F}\Ga)$ with $\SL_2$, and a point $e^0\in T(F)$, in such a way that the lower-left entry of the matrix of $e^0$ is equal to $1$, these factors become equal to the restriction to $T$ of the function $F_\eta$ in our definition of $\kappa$-orbital integrals in \S \ref{defkappaorbital}. Thus, when we compare the statement of the above theorem with the results of \cite{Langlands-basechange, LL}, these factors will disappear, since we identified elements of $\mathcal S^\kappa_\eta(\frac{\SL_2}{\SL_2})$ as measures on $T\sslash W$ by trivializing regular stable conjugacy classes along the rational canonical  form as above. 

The statement on characters is equivalent to saying that, pulled back to the space $\frac{\SL_2}{N}$ with coordinates $(c,t)$ as in \S \ref{ssRSthespace}, the difference 
$\Theta_{\pi_+} - \Theta_{\pi_-}$ 
is equal to the function
$$ \begin{cases} 0, & \mbox{ if } F(\sqrt{t^2-4})\nsimeq E;\\
\eta(c) \lambda(E/F,\psi) \frac{\theta(\gamma) + \theta(\bar\gamma)}{\sqrt{|t^2-4|}}, & \mbox{ if } \gamma, \bar\gamma\in T\mbox{ are the two elements with }\tr = t.
   \end{cases}
$$

This is Lemma 7.19 of \cite{Langlands-basechange}, except that to fix the sign $\pm$ that appears at the statement of the lemma, one needs to look into the proof. Langlands' normalization is that $\pi_+$ has a Whittaker model for the upper triangular unipotent group with respect to the character $\psi$, not $\psi^{-1}$ as in our statement of the theorem. (The description of $\pi_+$ appears on p.\ 132 of \cite{Langlands-basechange}.) But the result on difference of characters, then, which appears in his equation (7.11), includes a factor of $\eta(-c)$ ($\omega(-c)$ in his notation), which means that the torsor of conjugacy classes inside of a stable one is trivialized at $c=-1$. Applying the automorphism of $\SL_2$ given by conjugation by the diagonal element $\diag(1,-1)$, we obtain the statement with Whittaker character $\psi^{-1}$ and trivialization at $c=1$. 

For a function $\Phi$ on $\SL_2$, the image of the measure $\Phi dg$ in $\mathcal S^\kappa_\eta(\frac{\SL_2}{\SL_2}) $ is equal to $O_t^\kappa(\Phi) \sqrt{|t^2-4|} dt$, where $O_t^\kappa(\Phi)$ is the $\kappa$-orbital integral restricted to values of $t$ with $F(\sqrt{t^2-4})=E$, and taken with a measure on the conjugacy classes that matches $dg$ (in the sense that it disintegrates it against $\sqrt{|t^2-4|} dt$).
 The statement about transfer to $\mathcal S(T)^W$ is then \cite[Lemma 2.1]{LL}, which states that a certain function $\gamma\mapsto \Phi^T(\gamma,\Phi)$ extends to a smooth function on $T$. That function, using our $\kappa$-orbital integral which is based at $c=1$ hence absorbing the factor \eqref{kappafactor} of \cite{LL} is the function 
$$ \lambda(E/F,\psi) \sqrt{|t^2-4|} O^\kappa_t(\Phi).$$
Thus, the measure $e\mapsto \lambda(E/F,\psi) \sqrt{|\tr(e)^2-4|} O^\kappa_{\tr e}(\Phi) de,$
where $de$ is any Haar measure on $T$, will live in $\mathcal S(T)$, and it is evidently $W$-invariant. Its push-forward to $T\sslash W$, one can check, is the measure $\lambda(E/F,\psi) O^\kappa_{t}(\Phi) dt = 2 \lambda(E/F,\psi) \frac{f(t)}{\sqrt{|t^2-4|}}$,  for the choice of $de$ which is locally (in neighborhoods mapping isomorphically onto their image in $T\sslash W$) equal to $\frac{dt}{\sqrt{|t^2-4|}}$. Thus, the ``pullback'' of  $\frac{\lambda(E/F,\psi) f(t)}{\sqrt{|t^2-4|}}$ to $T$ (in the sense of the lift $\ell$) is indeed an element of $\mathcal S(T)^W$. The fact that we obtain all elements of $\mathcal S(T)^W$ is implicit in \cite{LL}, and easy to see. (One can produce many elements out of one by multiplying by a smooth function on $T\sslash W$.)

The fundamental lemma, i.e., the statement about basic vectors, is proven in \cite{LL} before Lemma 2.4, and the equivariance with respect to the Hecke algebra follows either by a similar direct calculation, or by the statement on characters (once it is known, in the non-split case, that only for the trivial character $\theta$ and only when $E/F$ is unramified is the packet $\{\pi_+, \pi_-\}$ unramified, corresponding to the principal series of the non-trivial unramified character). 
\end{proof}

Having introduced the virtual character $\Theta_\theta$, we backtrack for a moment to complement the statement of Theorem \ref{transfertokappa} with a statement about characters:

\begin{theorem}\label{twistedRudnick}
For $\Theta_\theta$ as above, its evaluation at $c=1$, understood as a function in one variable $t$, satisfies
 $$ (\mathscr F_{\Id,1}\circ \eta)^* \Theta_\theta|_{c=1} = J_{\pi^+},$$
  where $J_\Pi = J_{\pi_+}$ is the Kuznetsov relative character (``Bessel distribution'') of the $L$-packet associated to $\Theta_\theta$ (or, equivalently, its $(N,\psi^{-1})$-generic element $\pi_+$), and the composition of Fourier convolution $\mathscr F_{\Id, 1}$ with multiplication by $\eta$ is applied to the space $\mathcal S(N,\psi\backslash \SL_2/N,\psi)$ of standard test measures for the Kuznetsov formula of $\SL_2$.
\end{theorem}

The Fourier convolution should be understood again as the Fourier transform of a distribution.

\begin{proof}
 The proof is identical to the proof of statement \eqref{two} of Theorem \ref{thmRudnick}, so I will only present the formal argument. Notice that the only difference to that statement is the appearance of the character $\eta$ in the Fourier convolution, and the replacement of the stable character $\Theta_\Pi$ by the endoscopic character $\Theta_\theta$.
 
 As in the proof of Theorem \ref{thmRudnick}, let $f\in \mathcal S((N,\psi)\backslash \SL_2/(N,\psi))$, and $f\times dn$ its pullback to $\frac{\SL_2}{N}$, an $(N,\psi)$-equivariant measure (under left or right translation) on the set of $N$-conjugacy classes on $\SL_2$. The pairing (which should be understood as a regularized integral, as in the proof of Theorem \ref{thmRudnick})
 $$ \left< f, \Theta_{\pi_+}\right> := \int_{\frac{\SL_2}{N}} \Theta_{\pi_+}(g) f\times dn (g)$$
 is equal to the analogous pairing 
 $$\left< f, \Theta_\Pi\right>$$
 with the stable character, since $\pi_-$ is not generic (a fact that was used in the proof of Theorem \ref{thmRudnick}),
 or to 
 $$ \left< f, \Theta_\theta\right>,$$
 for the same reason. On the other hand, as was explained in the proof of that theorem, it is also equal to 
 $$\int_{N\backslash G/N} J_{\pi_+} \cdot f.$$
 
 But $\Theta_\theta$ is only supported on the values of $t$ which are contained in the image of $T(F)$ in $T\sslash W$, and is $\eta$-equivariant in the $c$-variable. Thus, we have, formally,
 $$\left< f, \Theta_\theta\right>=  \int_F \Theta_\theta(t) \left(\int_{F^\times} \eta(c) \frac{f\times dn(c,t)}{dt}  \right) dt.$$
 
 The identification of the inner integral with $\mathscr F_{\Id, 1}\circ \eta$ now follows as in \eqref{rshriek}; hence, $ (\mathscr F_{\Id,1}\circ \eta)^* \Theta_\theta =  J_{\pi_+}$. 
\end{proof}

In \S \ref{GGPS} I will present a derivation of the formula for stable characters of $\SL_2$ based on the proposition above.

\begin{proposition}\label{transfertokappa-characters}
The pullback of $\Theta_\theta \in \left(\mathcal S^\kappa_\eta(\frac{\SL_2}{\SL_2}) \otimes\delta^{\frac{1}{2}}\right)^*$ under the transfer operator $\mathcal T_\kappa$ of Theorem \ref{transfertokappa} is equal to 
$J_\Pi,$
where $\Pi$ is the $L$-packet associated to $\theta$.
\end{proposition}

\begin{proof}
As in \eqref{quotdiagram}, the transform $\bar\U$ descends to a meromorphic family of morphisms 
$$ \mathcal S(\bar X/\SL(V))^\circ_{(A_\ad,\chi)} \to \mathcal S^-_{L(\Ad, \chi^{-1}\delta^\frac{1}{2} \circ e^{\frac{\check\alpha}{2}})} (N,\psi\backslash \SL_2/N,\psi).$$

The proof of Proposition \ref{descent-general} shows that, for almost all $\chi$, this map factors through the image $\mathcal S\left(\frac{\SL_2}{B_\ad,\chi\delta^{\frac{1}{2}}}\right)$ of a twisted push-forward map, and computes its inverse,  
$$ \mathcal S^-_{L(\Ad, \chi^{-1}\delta^\frac{1}{2} \circ e^{\frac{\check\alpha}{2}})} (N,\psi\backslash \SL_2/N,\psi) \to \mathcal S\left(\frac{\SL_2}{B_\ad,\chi\delta^{\frac{1}{2}}}\right).$$
For $\chi = \eta \delta^{-\frac{1}{2}}$,  this inverse is given by the operator $\eta\circ \mathscr F_{\Id,\eta,1} = \mathscr F_{\Id, 1}\circ \eta$. 

Notice that at $\chi = \eta\delta^{\frac{1}{2}}$ both the Mellin transforms of elements of \newline $\mathcal S^-_{L(\Ad, \chi^{-1}\delta^\frac{1}{2} \circ e^{\frac{\check\alpha}{2}})} (N,\psi\backslash \SL_2/N,\psi)$ and of elements of $\mathcal S(\bar X/\SL(V))^\circ$ have poles; the former is stated in Proposition \ref{propdescent}, and the latter is Theorem \ref{transfertokappa}, which also shows that the residue, multiplied by the scalar factor $m_\eta^\circ:=\frac{\lambda(\eta,\psi)|t^2-4|^{\frac{1}{2}}}{\gamma^*(1,\psi)}$, lives in $\mathcal S^\kappa_\eta(\frac{\SL_2}{\SL_2}) \otimes \delta^\frac{1}{2}$.

Hence, taking residues of the Mellin transforms, we have a commutative diagram:

\begin{equation}\label{diagramthetacirc} \xymatrix{ 
\mathcal S^-_{L(\Sym^2,1)} (N,\psi\backslash G/N,\psi) \ar[d]\ar[r]^{\bar\U^{-1}} &
{^1\mathcal S}(\bar X/\SL_2)^\circ \ar[r]^{\mathcal H_X^\circ} \ar[d]& 
{^2\mathcal S}(\bar X/\SL_2)^\circ \ar[d]^{m_\eta^\circ\cdot \Res_{s=0}\check f(\eta\delta^{\frac{1}{2}+s})}   
\\
\mathcal S^-_{L(\Ad, \eta\delta \circ e^{\frac{\check\alpha}{2}})} (N,\psi\backslash \SL_2/N,\psi) \ar[r]^{\mathscr F_{\Id,1}\circ\eta} &
\mathcal S(\bar X/\SL_2)^\circ_{(A_\ad,\eta\delta^{-\frac{1}{2}})} \ar[r] &
\mathcal S^\kappa_\eta(\frac{\SL_2}{\SL_2}) \otimes\delta^{\frac{1}{2}},
} \end{equation}
where we have put exponents $1, 2$ above the two copies of $\mathcal S(\bar X/\SL_2)^\circ$ in order to distinguish them. 

I claim that the pullback of $\Theta_\theta$ to ${^1\mathcal S}(\bar X/\SL_2)^\circ$ under this diagram is equal to $\Theta_\theta$ itself, understood as a function on $X/\SL_2 = \frac{\SL_2}{N}$. Granted that, the proposition follows from Theorem \ref{twistedRudnick}.

The claim follows from the commutativity of the diagram \eqref{tokappa-all}. 
Let us recall that $\mathcal H_X$ descends from symplectic Fourier transform $\mathfrak F$ on the symplectic vector bundle $\bar X\to \SL_2$; and that the map $\mathcal S(\bar X/\SL_2)\to \mathcal S^\kappa_\eta(\frac{\SL_2}{\SL_2}) \otimes\delta^{\frac{1}{2}}$ descends from the map that takes a measure $\Phi \cdot |\omega|\times dg \in \mathcal S(\bar X)$ to its restriction to the zero section $\Phi|_{\SL_2} dg \in \mathcal S(\SL_2)$ (see Proposition \ref{tokappa}). By elementary properties of Fourier transform, the pullback of any generalized function $\Theta$ on $\SL_2$ under the adjoint of 
$$ \xymatrix{\mathcal S(\bar X) \ar[r]^{\mathfrak F} & \mathcal S(\bar X)  \ar[rr]^{\mbox{\tiny restriction to}}_{\mbox{\tiny zero section}} &&\mathcal S(\SL_2)\otimes\delta^\frac{1}{2}}$$
is its pullback $\pi^*\Theta$ under the canonical projection $\pi:\bar X\to \SL_2$. We conclude that the pullback of the virtual character $\Theta_\theta$ to ${^1\mathcal S}(\bar X/\SL_2)^\circ$ is equal to $\Theta_\theta$. 

\end{proof}

Combining the above with the theorem of Labesse and Langlands, Theorem \ref{LLtheorem}, setting $\mathcal T_T = \mathcal T_{LL} \circ \mathcal T_\kappa$, we can now prove Theorem \ref{thmVenkatesh}.

\begin{proof}[Proof of Theorem \ref{thmVenkatesh}]

Composing the Labesse--Langlands transfer $\mathcal T_{LL}$ of Theorem \ref{LLtheorem} with the transfer operator $\mathcal T_\kappa$ of Theorem \ref{transfertokappa}, we obtain the desired transfer operator $\mathcal T_T$. To compute it, we combine \eqref{tokappa-simple} with the formula of Theorem \ref{LLtheorem} to write
\[ \mathcal T_T(f) = \ell\left( t\mapsto \frac{\lambda(E/F,\psi)}{\sqrt{|t^2-4|}} \int_c \int_x f \left( c, \frac{t}{x} \right) \eta(xct) \psi(x) dx  \right)\]
Explicating the lift $\ell$, we can write its argument as the product of the function $t\mapsto (dt)^{-1}\lambda(E/F,\psi)\int_c \int_x f \left( c, \frac{t}{x} \right) \eta(xct) \psi(x) dx$ with the measure $\frac{dt}{\sqrt{|t^2-4|}}$, and the definition of $\ell$ states that this lifts to the pullback of the same function to $T$, multiplied by the measure $da$.  Thus, we get
\begin{eqnarray} \nonumber
  \frac{\mathcal T_T(f)(a)}{da} & =  (dt)^{-1}\lambda(E/F,\psi)\int_c \int_x f \left( c, \frac{t}{x} \right) \eta(xct) \psi(x) dx \\ & = (dt)^{-1}\lambda(E/F,\psi)\int_x f_\eta \left(\frac{t}{x} \right) \eta(xt) \psi(x) dx,  \label{Venkatesh-formula-2}
\end{eqnarray}
with $f_\eta$ as in \eqref{Venkatesh-final-eta}.

The stated properties of this operator follow from the corresponding properties of the transfer operators $\mathcal T_\kappa$ and $\mathcal T_{LL}$, per Theorems \ref{transfertokappa}, \ref{LLtheorem} and Proposition \ref{transfertokappa-characters}.

\end{proof}

\begin{remark}
 As a check for our formulas, let us calculate the relative characters on basic vectors. Hence, let $F$ be non-Archimedean, with ring of integers $\mathfrak o$, and assume that the torus $T$ is unramified (possibly split). 
 
 Let $f\in \mathcal S(N,\psi\backslash G/N,\psi)$ be the image of the identity element of the Hecke algebra, and $\pi = \chi\otimes\tau$ an unramified representation of $G=\Gm\times\SL_2$, whose $L$-packet we will denote by $\Pi$. By \eqref{Bessel-basic}, 
 $$ \left< f, J_\Pi\right> = \frac{\zeta(2)}{L(\tau,\Ad,1)},$$
 Hence, if we replace $f$ by the basic vector 
 $f_{L(\Sym^2, 1)} \in \mathcal S^-_{L(\Sym^2, 1)} (N,\psi\backslash G/N,\psi)$, 
 $$ \left< f_{L(\Sym^2, 1)} , J_\Pi\right> = \frac{\zeta(2)L(\pi, \Sym^2,1)}{L(\tau,\Ad,1)}.$$
 Note that $\Sym^2$ denotes the descent of the $\Sym^2$-representation of $\GL_2$ under $\GL_2\to \Gm\times\PGL_2$; that is, the adjoint representation of $\PGL_2$ tensored by the scalar action of $\Gm$.

 If $f_T$ is the identity in the Hecke algebra of $T$, and $\theta$ is an unramified character, we have $\left<f_T, \theta\right> = 1$. On the other hand, the theorem states that 
 $$1= \left<f_T, \theta\right>  = \left< \frac{L(\eta,1)}{\zeta(1)\zeta(2)}\mathcal T_T f_{L(\Sym^2, 1)}, \theta\right> = \frac{L(\eta,1)}{\zeta(1)\zeta(2)} \left< f_{L(\Sym^2, 1)}, \mathcal T_T^* \theta\right> = $$
 $$=\frac{L(\eta,1)}{\zeta(1)\zeta(2) } \left< f_{L(\Sym^2, 1)}, J_\Pi \right> =  \frac{L(\eta,1)}{\zeta(1)\zeta(2)}  \cdot \frac{\zeta(2)L(\pi, \Sym^2,1)}{L(\tau,\Ad,1)},$$
 where $\Pi$ is the endoscopic lift of $\theta$ --- that is: if $T$ is split then $\Pi = 1\boxtimes$ the principal series representation unitarily induced from $\theta$, and if $T$ is non-split (so $\theta=1$), then $\Pi = \eta\boxtimes$ the principal series representation unitarily induced from $\eta$, the unique non-trivial quadratic unramified character.
 
 One easily confirms directly that the right hand side is, indeed, equal to $1$. 
\end{remark}

\subsection{An application: the Gelfand--Graev--Piatetski-Shapiro formula} \label{GGPS}

Theorem \ref{thmVenkatesh} provides the transfer operator from tori to the Kuznetsov formula. On the other hand, we have already seen in Theorem \ref{thmRudnick} how to transfer between the Kuznetsov formula and the stable trace formula for $\SL_2$. We will combine those two in order to obtain the well-known formula for the stable transfer of characters from tori to $\SL_2$, when $T$ is non-split:
\begin{equation}\label{GGPSformula}
 \Theta^\st_\theta (t) = \frac{2}{\Vol(T)} \frac{\eta}{|\bullet|} \star_{\text{\tiny +}} S_\theta (t),
\end{equation}
where:
\begin{itemize}
 \item $\Theta_\theta^\st$ is the stable character lifted from the character $\theta$ of $T$, expressed as usual in the trace variable $t$;
 \item $S_\theta(t) = \frac{\theta(a)+\theta^{-1}(a)}{\sqrt{|t^2-4|}}$ if $a\in T$ has trace $t$ (and $S_\theta$ is zero otherwise);
 \item $\star_{\text{\tiny +}}$ denotes \emph{additive convolution} in the variable $t$, with respect to an additive Haar measure $dt$;
 \item the volume of $T$ is taken with respect to the Haar measure $da$ which is locally equal to $\frac{dt}{\sqrt{|t^2-4|}}$, as before.
\end{itemize}

This is formula 2.5.4.(7) in the book \cite{GGPS} of Gelfand, Graev, and Piatetski-Shapiro. In the split case, $S_\theta$ is equal to the character of the lift, i.e., the measure $\frac{2}{\Vol(T)} \frac{\eta}{|\bullet|}$ is replaced by the delta measure at zero, and it will become clear why this is so.

Theorem \ref{thmVenkatesh} states that the adjoint of the operator $\mathcal T_T^* \theta $ takes the character $\theta$ of the torus $T$ to the relative character $J_\Pi$ for the Kuznetsov formula of $G=\SL_2\times \Gm$. This will be $\eta$-equivariant in the $\Gm$-factor; let us now ignore the $\Gm$-factor, and use $J_\Pi$ to denote the restriction of the relative character to $\SL_2$. For $f\in \mathcal S(G)$, let $f_\eta (t) = \int_c f(c,t) \eta(c) $, a measure in the variable $t$. Then, \eqref{Venkatesh-final} can be written:
$$ \frac{\mathcal T_T(f)(a)}{da} = \lambda(\eta,\psi) \int_x \frac{f_\eta\left( \frac{t}{x} \right)}{dt}  \eta(xt) \psi(x) dx,$$
where $t=t(a)$.

By Theorem \ref{thmVenkatesh}, 
$$J_\Pi =  \mathcal T_T^* \theta,$$ 
which from the above formula can easily be computed as 
\begin{equation} J_\Pi(\zeta) = \lambda(\eta,\psi) \eta(\zeta) \left(\psi(\bullet^{-1})|\bullet|^{-1} d^\times\bullet \right)\star S_\theta(\zeta),
\end{equation}
where $S_\theta(t) = \frac{\theta(a) + \theta(a^{-1})}{\sqrt{|t^2-4|}}$, with $a, a^{-1}$ being the elements of $T$ with $t(a^{\pm 1})=t$. 

On the other hand, by Theorem \ref{thmRudnick},
$$J_\Pi = T_{\SL_2}^* \Theta_\Pi,$$
where $\mathcal T_{\SL_2}$ is given by multiplicative convolution by the measure $\psi(z) |z| d^\times$. Its adjoint (on generalized functions)
will be given by multiplicative convolution by the measure $\psi(\frac{1}{z}) |z|^{-1} d^\times z$, and the inverse of that will be given by convolution by $\psi^{-1}(z) d^\times z$.

Hence, 
\begin{align*} \Theta_\Pi(t) & = \left(\psi^{-1}(\bullet) d^\times \bullet\right) \star J_\Pi (t) \\
& = \lambda(\eta,\psi) \left(\psi^{-1}(\bullet) d^\times \bullet\right) \star \left(\eta(t)\left(\psi(\bullet^{-1})|\bullet|^{-1} d^\times\bullet \right)\star S_\theta(t)\right).
\end{align*}

If $\hat\Phi$ denotes usual Fourier transform $\hat\Phi(y) = \int \Phi(x) \psi(x) dx$ (with self-dual Haar measure with respect to $\psi$), the above can be written
$$\Theta_\Pi(t) = \lambda(\eta,\psi)\int \psi^{-1}(tu) \eta(u)\widehat{S_\theta}(u) du.$$

In the split case, where $\eta=1$, $\lambda(\eta,\psi)=1$, this just says that 
\begin{equation}
 \Theta_\Pi = S_\theta,
\end{equation}
the well-known formula for the character of a principal series representation. 

Assume now that we are in the non-split case. Then $\eta$ is the Fourier transform of $\gamma(\eta,1,\psi)^{-1} \frac{\eta}{|\bullet|}$, understood as a generalized function by a Tate integral, and the above can be written 
$$ \Theta_\Pi = \frac{\lambda(\eta,\psi)}{\gamma(\eta,1,\psi)} \frac{\eta}{|\bullet|} \star_{\text{\tiny +}} S_\theta.$$

By \eqref{lambdaconstant}, $\frac{\lambda(\eta,\psi)}{\gamma(\eta,1,\psi)} = \frac{2}{\Vol(T)}$, where the volume of $T$ is taken with respect to the Haar measure $da$ that is the quotient of the self-dual measure of the associated quadratic extension $E$ with respect to the character $\psi\circ\tr$, and the self-dual, with respect to $\psi$, measure $dt$ on $F$. I claim that this is the same as the measure $da$ which, locally in $T$, is equal to $\frac{dt}{\sqrt{|t^2-4|}}$. Indeed, if $E=F(\sqrt{D})$ is the quadratic extension associated to the torus $T$, and we write $z=x+y\sqrt{D}$ for an element, the self-dual measure with respect to $\psi\circ\tr$ is $|4D|^\frac{1}{2} dx dy$. Computing differentials for the sequence
$$ 1 \to T \to \Res_{E/F} \Gm \to \Gm \to 1,$$
one sees that the multiplicative measure $d^\times z = \frac{dz}{|x^2-Dy^2|}$ factorizes as the product of the measure on $T$ which is locally equal to $ \frac{dx}{|Dy^2|^\frac{1}{2}} $, and the multiplicative measure on $\Gm$. On $T$, the kernel of the norm map, we have $|Dy^2|^\frac{1}{2} = \sqrt{|x^2-1|}$, and together with the fact that $x = \frac{t}{2}$, this shows that the quotient measure on $T$ is the measure that is locally equal to $\frac{dt}{\sqrt{|t^2-4|}}$. This proves the formula \eqref{GGPSformula} for the stable characters.

\bibliographystyle{alphaurl}
\bibliography{biblio}

\begin{thebibliography}{GGPS69}

\bibitem[BK14]{BeKr}
Joseph Bernstein and Bernhard Kr{\"o}tz.
\newblock Smooth {F}r\'echet globalizations of {H}arish-{C}handra modules.
\newblock {\em Israel J. Math.}, 199(1):45--111, 2014.
\newblock \href {https://doi.org/10.1007/s11856-013-0056-1}
  {\path{doi:10.1007/s11856-013-0056-1}}.

\bibitem[GGPS69]{GGPS}
I.~M. Gelfand, M.~I. Graev, and I.~I. Pyatetskii-Shapiro.
\newblock {\em Representation theory and automorphic functions}.
\newblock Translated from the Russian by K. A. Hirsch. W. B. Saunders Co.,
  Philadelphia, Pa.-London-Toronto, Ont., 1969.

\bibitem[Her12]{Herman}
P.~Edward Herman.
\newblock The functional equation and beyond endoscopy.
\newblock {\em Pacific J. Math.}, 260(2):497--513, 2012.
\newblock \href {https://doi.org/10.2140/pjm.2012.260.497}
  {\path{doi:10.2140/pjm.2012.260.497}}.

\bibitem[Jac72]{Jacquet-Auto2}
Herv\'e Jacquet.
\newblock {\em Automorphic forms on {${\rm GL}(2)$}. {P}art {II}}.
\newblock Lecture Notes in Mathematics, Vol. 278. Springer-Verlag, Berlin-New
  York, 1972.

\bibitem[Jac03]{Jacquet}
Herv\'e Jacquet.
\newblock Smooth transfer of {K}loosterman integrals.
\newblock {\em Duke Math. J.}, 120(1):121--152, 2003.
\newblock \href {https://doi.org/10.1215/S0012-7094-03-12015-3}
  {\path{doi:10.1215/S0012-7094-03-12015-3}}.

\bibitem[JL70]{JL}
H.~Jacquet and R.~P. Langlands.
\newblock {\em Automorphic forms on {${\rm GL}(2)$}}.
\newblock Lecture Notes in Mathematics, Vol. 114. Springer-Verlag, Berlin-New
  York, 1970.

\bibitem[KL06]{KLP}
Andrew Knightly and Charles Li.
\newblock A relative trace formula proof of the {P}etersson trace formula.
\newblock {\em Acta Arith.}, 122(3):297--313, 2006.
\newblock \href {https://doi.org/10.4064/aa122-3-5}
  {\path{doi:10.4064/aa122-3-5}}.

\bibitem[Lan70]{Langlands-epsilon}
Robert~P. Langlands.
\newblock On the functional equation of the {A}rtin {L}-functions.
\newblock 1970.
\newblock URL: \url{https://publications.ias.edu/sites/default/files/a-ps.pdf}.

\bibitem[Lan80]{Langlands-basechange}
Robert~P. Langlands.
\newblock {\em Base change for {${\rm GL}(2)$}}, volume~96 of {\em Annals of
  Mathematics Studies}.
\newblock Princeton University Press, Princeton, N.J.; University of Tokyo
  Press, Tokyo, 1980.

\bibitem[LL79]{LL}
J.-P. Labesse and R.~P. Langlands.
\newblock {$L$}-indistinguishability for {${\rm SL}(2)$}.
\newblock {\em Canad. J. Math.}, 31(4):726--785, 1979.
\newblock \href {https://doi.org/10.4153/CJM-1979-070-3}
  {\path{doi:10.4153/CJM-1979-070-3}}.

\bibitem[Sak]{SaTransfer1}
Yiannis Sakellaridis.
\newblock Transfer operators and {H}ankel transforms between relative trace
  formulas, {I}: {C}haracter theory.
\newblock To appear in \emph{Advances in Mathematics}.
\newblock \href {http://arxiv.org/abs/1804.02383} {\path{arXiv:1804.02383}}.

\bibitem[Sak13]{SaBE1}
Yiannis Sakellaridis.
\newblock Beyond endoscopy for the relative trace formula {I}: local theory.
\newblock In {\em Automorphic Representations and L-functions}, pages 521--590.
  Amer. Math. Soc., Providence, RI, 2013.
\newblock Edited by: D. Prasad, C. S. Rajan, A. Sankaranarayanan, and J.
  Sengupta, Tata Institute of Fundamental Research, Mumbai, India, 2013.
\newblock \href {http://arxiv.org/abs/1207.5761} {\path{arXiv:1207.5761}}.

\bibitem[Sak16]{SaStacks}
Yiannis Sakellaridis.
\newblock The {S}chwartz space of a smooth semi-algebraic stack.
\newblock {\em Selecta Math. (N.S.)}, 22(4):2401--2490, 2016.
\newblock \href {https://doi.org/10.1007/s00029-016-0285-3}
  {\path{doi:10.1007/s00029-016-0285-3}}.

\bibitem[Sak18]{SaStacks-erratum}
Yiannis Sakellaridis.
\newblock Correction to: {T}he {S}chwartz space of a smooth semi-algebraic
  stack.
\newblock {\em Selecta Math. (N.S.)}, 24(5):4961--4965, 2018.
\newblock \href {https://doi.org/10.1007/s00029-018-0445-8}
  {\path{doi:10.1007/s00029-018-0445-8}}.

\bibitem[Sak19a]{SaBE2}
Yiannis Sakellaridis.
\newblock Beyond endoscopy for the relative trace formula {II}: global theory.
\newblock {\em J. Inst. Math. Jussieu}, 18(2):347--447, 2019.
\newblock \href {https://doi.org/10.1017/s1474748017000032}
  {\path{doi:10.1017/s1474748017000032}}.

\bibitem[Sak19b]{SaHanoi}
Yiannis Sakellaridis.
\newblock Relative functoriality and functional equations via trace formulas.
\newblock {\em Acta Math. Vietnam.}, 44(2):351--389, 2019.
\newblock \href {https://doi.org/10.1007/s40306-018-0295-7}
  {\path{doi:10.1007/s40306-018-0295-7}}.

\bibitem[SV17]{SV}
Yiannis Sakellaridis and Akshay Venkatesh.
\newblock Periods and harmonic analysis on spherical varieties.
\newblock {\em Ast\'erisque}, (396):360, 2017.

\bibitem[Ven02]{Venkatesh-thesis}
Akshay Venkatesh.
\newblock {\em Limiting forms of the trace formula}.
\newblock ProQuest LLC, Ann Arbor, MI, 2002.
\newblock Thesis (Ph.D.)--Princeton University.

\bibitem[Ven04]{Venkatesh}
Akshay Venkatesh.
\newblock ``{B}eyond endoscopy'' and special forms on {GL}(2).
\newblock {\em J. Reine Angew. Math.}, 577:23--80, 2004.
\newblock \href {https://doi.org/10.1515/crll.2004.2004.577.23}
  {\path{doi:10.1515/crll.2004.2004.577.23}}.

\end{thebibliography}

\end{document}